\documentclass[a4paper, 11pt, reqno]{amsart}

\usepackage{amsmath, amssymb, mathabx, amsthm, mathrsfs, hyperref, color, enumerate, tikz, tikz-cd, mathtools, mathdots,lscape, graphicx}
\usepackage[capitalize]{cleveref}
\usetikzlibrary{cd, calc, matrix, shapes, decorations.pathreplacing, intersections}
\usetikzlibrary{decorations.markings}
\usepackage[export]{adjustbox}
\usepackage{verbatim, caption}

\newif\ifshowtikz

\let\oldtikzcd\tikzcd
\let\oldendtikzcd\endtikzcd

\renewenvironment{tikzcd}{%
    \ifshowtikz\expandafter\oldtikzcd%
    \else\comment%
    \fi
}{%
    \ifshowtikz\oldendtikzcd%
    \else\endcomment%
    \fi
}

\let\oldtikzpicture\tikzpicture
\let\oldendtikzpicture\endtikzpicture

\renewenvironment{tikzpicture}{%
    \ifshowtikz\expandafter\oldtikzpicture%
    \else\comment%
    \fi
}{%
    \ifshowtikz\oldendtikzpicture%
    \else\endcomment%
    \fi
}

\showtikztrue

\usepackage{booktabs, array}
\newcolumntype{C}{>{$}c<{$}}
\renewcommand{\P}{\mathbb{P}}

\newcommand{\C}{\mathbb{C}}
\newcommand{\A}{\mathbb{A}}
\usepackage{lscape}
\usepackage{adjustbox}
\usetikzlibrary{shapes}
\newcommand{\Z}{\mathbb{Z}}
\newcommand{\comma}{,}
\newcommand{\OO}{\mathcal{O}}

\newcommand{\Coh}{\operatorname{Coh}}
\newcommand{\Pic}{\operatorname{Pic}}

\newcommand{\perf}{\operatorname{perf}}
\newcommand{\Hom}{\operatorname{Hom}}
\newcommand{\Ext}{\operatorname{Ext}}

\newcommand{\Spec}{\operatorname{Spec}}

\newcommand{\w}{\mathbf{w}}
\newcommand{\W}{\mathbf{W}}

\newcommand{\G}{\mathbb{G}_m}
\newcommand{\eps}{\varepsilon}
\newcommand{\Aut}{\text{Aut}}
\newcommand{\Exc}{\mathbf{Exc}}

\newcommand{\xt}{\check{x}}
\newcommand{\yt}{\check{y}}
\tikzset{->-/.style={decoration={
			markings,
			mark=at position #1 with {\arrow{>}}},postaction={decorate}}}
\theoremstyle{plain}
\newtheorem{mthm}{Theorem}
\newtheorem{mdef}{Definition}
\newtheorem{mconj}{Conjecture}
\newtheorem{thm}{Theorem}[section]
\newtheorem{lem}[thm]{Lemma}

\newcommand{\vc}[2]{{{}^{#2}V_{#1}}}

\theoremstyle{remark}
\newtheorem{ex}[thm]{Example}
\newtheorem{rmk}[thm]{Remark}

\title{Homological mirror symmetry for nodal stacky curves}
\address{Universit\"at Hamburg\\ Fachbereich Mathematik\\ Bundesstra{\ss}e 55\\ 20146 Hamburg\\ Germany}
\author{Matthew Habermann}
\email{matthew.habermann@uni-hamburg.de}
\begin{document}

	\begin{abstract}
		In this paper, we establish homological mirror symmetry where the A-model is a finite quotient of the Milnor fibre of an invertible curve singularity, proving a conjecture of Lekili and Ueda from \cite{LekiliUeda} in this dimension. Our strategy is to view the B--model as a cycle of stacky projective lines and generalise the approach of Lekili and Polishchuk in \cite{LekiliPolishchukAuslander} to allow the irreducible components of the curve to have non-trivial generic stabiliser. We then prove that the A--model which results from this strategy is graded symplectomorphic to the corresponding quotient of the Milnor fibre.  
	\end{abstract}
	\maketitle
\section{Introduction}
To begin with, consider an $n\times n$ matrix $A$ with non-negative integer entries $a_{ij}$. From this, we can define a polynomial $\w\in\C[x_1,\dots,x_n]$ given by
\begin{align*}
\w(x_1,\dots,x_n)=\sum_{i=1}^n\prod_{j=1}^nx_j^{a_{ij}}. 
\end{align*} 
If $\w$ is quasi-homogeneous, we can associate to it a \emph{weight system} $(d_0,d_1,\dots, d_n;h)$, where 
\begin{align*}
\w(t^{d_1}x_1,\dots, t^{d_n}x_n)=t^h\w(x_1,\dots,x_n),
\end{align*}
and $d_{0}:=h-d_1-\dots-d_n$. In \cite{BerglundHubsch}, the authors define the transpose of $\w$, denoted by $\widecheck{\w}$, to be the polynomial associated to $A^T$, 
\begin{align*}
\widecheck{\w}(\check{x}_1,\dots,\check{x}_n)=\sum_{i=1}^n\prod_{j=1}^n\check{x}_j^{a_{ji}},
\end{align*}
and we call this the \emph{Berglund--H\"ubsch transpose}. One can associate a weight system for $\widecheck{\w}$, denoted by $(\check{d}_{0},\check{d}_1,\dots,\check{d}_n;\check{h})$, in the same way as for $\w$. We call a polynomial $\w$ \emph{invertible} if the matrix $A$ is invertible over $\mathbb{Q}$, and if both $\w$ and $\widecheck{\w}$ define isolated singularities at the origin.\\

A corollary of Kreuzer--Skarke's classification of quasi-homogeneous polynomials, \cite{kreuzer1992}, is that any invertible polynomial can be decoupled into the Thom--Sebastiani sum of \textit{atomic} polynomials of the following three types:
\begin{itemize}
	\item Fermat: $\w=x_1^{p_1}$,
	\item Loop: $\w=x_1^{p_1}x_2+x_2^{p_2}x_3+\dots+x_n^{p_n}x_1$,
	\item Chain: $\w=x_1^{p_1}x_2+x_2^{p_2}x_3+\dots+x_{n-1}^{p_{n-1}}x_n+x_n^{p_n}$.
\end{itemize}
The Thom--Sebastiani sums of polynomials of Fermat type are also called Brieskorn--Pham.
\begin{rmk}
	In this paper, we will use the word `chain' to refer to both a chain of curves, as well as a curve corresponding to an invertible polynomial of chain type (which is in fact a cycle of curves with two irreducible components). We believe that this distinction will be clear from context.
\end{rmk}
To any invertible polynomial, one can associate its \textit{maximal symmetry group}
\begin{align*}
\Gamma_{\w}:=\{(t_1,\dots,t_{n+1})\in(\C^*)^{n+1}|\ \w(t_1x_1,\dots, t_nx_n)=t_{n+1}\w(x_1,\dots, x_n)\}.
\end{align*}
In general, this is a finite extension of $\C^*$, and is the group of diagonal transformations of $\A^n$ which keep $\w$ semi-invariant with respect to the character $(t_1,\dots, t_{n+1})\mapsto t_{n+1}$, which we denote as $\chi_\w$. It follows from the fact that $\w$ is quasi-homogeneous that there is a natural inclusion
\begin{align}
\begin{split}\label{GradingInclusionEq}
\phi:\C^*&\hookrightarrow \Gamma\\
t&\mapsto(t^{d_1},\dots,t^{d_n}),
\end{split}
\end{align}
and a subgroup $\Gamma\subseteq \Gamma_\w$ is defined to be \emph{admissible} if it is of finite index and contains $\mathrm{im}(\phi)$. Given an admissible subgroup $\Gamma\subseteq \Gamma_{\w}$, we define the \emph{Berglund--Henningson} dual group as 
\begin{align}\label{DualGpDefn}
\widecheck{\Gamma}:=\Hom(\Gamma_{\w}/\Gamma,\C^*).
\end{align}
The extension to consider invertible polynomials with different grading groups was initiated by Berglund and Henningson in \cite{Berglund1995LandauGinzburgOM} and systematically studied by Krawitz in \cite{Krawitz}. By construction, we have that $\widecheck{\Gamma}\subseteq \mathrm{SL}(n,\C)$. Note that this definition differs from that of \cite[Definition 15]{Krawitz}; however, \cite[Proposition 3]{EbelingTakahashi}, shows that the above definition is equivalent\footnote{Strictly speaking, the cited result pertains to the closely related maximal group of symmetries which keeps $\w$ invariant, although the resulting quotient, and therefore dual groups, are the same.}. With this setup, homological Berglund--H\"ubsch--Henningson mirror symmetry predicts:
\begin{mconj}\label{MFConjecture} For any invertible polynomial $\w$ with admissible symmetry group $\Gamma\subseteq \Gamma_\w$ and corresponding dual group $\widecheck{\Gamma}$, there is a quasi-equivalence
	\begin{align*}
	\mathrm{mf}(\A^n,\Gamma, \w)\simeq \mathcal{FS}_{\widecheck{\Gamma}}(\widecheck{\w})
	\end{align*}
	of pre-triangulated $A_\infty$-categories over $\C$.
\end{mconj}
In the above, $\mathrm{mf}(\A^n,\Gamma, \w)$ is the dg-category of $\Gamma$-equivariant matrix factorisations of $\w$ on $\A^n$, and $\mathcal{FS}_{\widecheck{\Gamma}}(\widecheck{\w})$ is the orbifold Fukaya--Seidel category of $(\widecheck{\w},\widecheck{\Gamma})$. Presently, there does not exist a definition of such a category in full generality, although the work of \cite{ChoChoaJeong} gives a candidate in the context of invertible polynomials where $d_0\leq 0$. In the maximally graded case where $\Gamma=\Gamma_\w$, this is known as homological Berglund--H\"ubsch mirror symmetry, and goes back to \cite{Weightedprojectivelines}, \cite{2006math......4361U}. Recently, there have been many results in the direction of establishing this conjecture. In the maximally graded case it has been proven in several situations -- in particular, for Brieskorn--Pham polynomials in any number of variables in \cite{Futaki2011}, and for Thom--Sebastiani sums of polynomials of type $A$ and $D$ in \cite{FU2}. The conjecture is also established for all invertible polynomials in two variables in \cite{HabermannSmith}. Progress on the conjecture was made in \cite{PolishchukVarolgunes}, where the authors demonstrate that the exceptional collection in the category of matrix factorisations constructed in \cite{AramakiTakahashi} satisfies a recursion relation with respect to the number of variables. Recently, in \cite{HabBHHHMS}, a reformulation of \cref{MFConjecture} in two variables (\cite[Conjecture 6.1]{FU2}) was established, incorporating the equivariance in the A--model by pulling back $\widecheck{\w}$ to the total space of the crepant resolution of $\C^2/\widecheck{\Gamma}$.\\

By the definition of $\widecheck{\Gamma}$ (see \cref{Applications section} for the case of two variables which we consider), we have that it is a subgroup of $(\C^*)^n$ which keeps $\widecheck{\w}$ \emph{invariant}, and so, in particular, preserves its fibres when considering it as a map $\widecheck{\w}:\C^n\rightarrow \C$. In this case, the (completed) Milnor fibre is $\widecheck{\w}^{-1}(1)=\widecheck{V}$, so it makes sense to define the equivariant Milnor fibre as the quotient stack $[\widecheck{V}/\widecheck{\Gamma}]$, although it should be emphasised that symplectic techniques in this setting are still in their infancy. Nevertheless, once an appropriate definition of the wrapped Fukaya category for such a quotient stack is made sense of, one expects:
\begin{mconj}[{\cite[Conjecture 1.4]{LekiliUeda}}]\label{LUConj}
	For any invertible polynomial $\w$ with admissible symmetry group $\Gamma\subseteq \Gamma_{\w}$ and corresponding dual group $\widecheck{\Gamma}$, there is a quasi equivalence 
	\begin{align*}
	\mathcal{W}([\widecheck{V}/\widecheck{\Gamma}])\simeq \mathrm{mf}(\A^{n+1},\Gamma,\w+x_0x_1\dots x_n)
	\end{align*}
	of pre-triangulated $A_\infty$-categories over $\C$. 
\end{mconj}
Here, $\mathrm{mf}(\A^{n+1},\Gamma,\w+x_0x_1\dots x_n)$ is the dg-category of $\Gamma$--equivariant matrix factorisations of $\w+x_0\dots x_n$ on $\A^{n+1}$, where the action of $\Gamma$ has been extended in a prescribed way (\cite[Section 2]{LekiliUeda}). In the maximally graded case, this conjecture was recently established in the case of $n\geq 3$ for all simple singularities in \cite{LekiliUedaSimple}, and the case of Brieskorn--Pham polynomials of the form $x_1^2+x_2^2+\w$ in \cite{LiNonexistence}. Recently, progress towards a $\Z/2$-graded equivalence was made for the Milnor fibre of any maximally graded invertible polynomial in \cite{GammageHMS}. \\

There is a trichotomy of cases depending on the sign of $d_0$, and in the case of $d_0>0$ (log general type), a generalisation of Orlov's theorem (\cite[Theorem 3.11]{Orlov2009}) yields an equivalence 
\begin{align}
\mathrm{mf}(\A^{n+1},\Gamma,\w+x_0x_1\dots x_n)\simeq D^b\Coh(Z_{\w,\Gamma}),
\end{align}
where
\begin{align}\label{InvertibleHypersurface}
Z_{\w,\Gamma}:=\big[\big(\Spec\C[x_0,x_1,\dots, x_n]/(\w+x_0x_1\dots x_n)\setminus(\boldsymbol{0})\big)/\Gamma\big].
\end{align}
The generalisation to the case where $\Gamma_\w$ is a finite extension of $\C^*$ is straightforward, and the extension to the setting of dg-categories was studied in \cite{Shipmandg}, \cite{Isikorlov}, \cite{CaldararuTu}. In two variables, every invertible polynomial is of log general type except for $x^2+y^2$. This, however, corresponds to the well-understood HMS statement for $\C^*$.\\

Recall that the subcategory $\perf Z_{\w,\Gamma}\subseteq D^b\Coh(Z_{\w,\Gamma})$ is comprised of the objects which are Ext-finite, since $Z_{\w,\Gamma}$ is a proper stack. On the symplectic side, it is clear that compact Lagrangians have finite dimensional Hom-spaces; however, it is not known in general if non-compact Lagrangians necessarily do not. This is reasonable to expect though, and is certainly the case in all known circumstances. In the equivariant setting, it makes sense that the same statement should be true, since morphisms should be given by the $\widecheck{\Gamma}$-invariant pieces of the Floer complexes on $\widecheck{V}$. Therefore, Conjecture \ref{LUConj} in the log general type case would imply an equivalence
\begin{align}\label{CompactImplication}
\mathcal{F}([\widecheck{V}/\widecheck{\Gamma}])\simeq\perf Z_{\w,\Gamma}.
\end{align}
In the maximally graded case, the first instance of this was given in \cite{DehnSurgery} for $x_1^3+x_2^2$. The equivalence was subsequently establish in the cases of $\w=\sum_{i=1}^n x_i^{n+1}$ and $\w=x_1^2+\sum_{i=2}^n x_i^{2n}$, both for any $n>1$, in \cite{LekiliUeda}, and for all invertible polynomials in two variables in \cite{HabermannHMS}. Our main result is a proof of \cref{MFConjecture} and the implication \eqref{CompactImplication} in the case of $n=2$, also allowing for the B--model to be non-maximally graded.
\begin{mthm}\label{InvertiblePolyTheorem} Let $\w$ be an invertible polynomial in two variables with admissible symmetry group $\Gamma\subseteq \Gamma_{\w}$ and corresponding dual group $\widecheck{\Gamma}$. Then, the action of $\widecheck{\Gamma}$ on $\widecheck{V}$ is free, and there are quasi-equivalences
	\begin{align*}
	\mathcal{F}(\widecheck{V}/\widecheck{\Gamma})&\simeq \perf Z_{\w,\Gamma}\\
	\mathcal{W}(\widecheck{V}/\widecheck{\Gamma})&\simeq D^b\mathrm{Coh} (Z_{\w,\Gamma})
	\end{align*}
	of $\Z$-graded pre-triangulated $A_\infty$-categories over $\C$.
\end{mthm}
\begin{rmk}
	It should be reiterated that, although there is a trichotomy of cases depending on the weight $d_0$, all but one invertible polynomials in two variables are of log general type, and this exception is well-understood. We are therefore free to state Theorem \ref{InvertiblePolyTheorem} in the context of invertible polynomials of log general type without further assumptions. 
\end{rmk}
\begin{rmk}
In general, it is not true that the action of $\widecheck{\Gamma}$ on the Milnor fibre will be free; however, the fixed points of the action on $\C^n$ always lie on the divisor $\{\check{x}_1\dots\check{x}_n=0\}$, meaning that the action on the \emph{very affine Milnor fibre} $\widecheck{V}\cap(\C^*)^n$ considered in \cite{GammageHMS} is free.
\end{rmk} 
Our proof of \cref{InvertiblePolyTheorem} is orthogonal to previous approaches to \cref{LUConj}, and instead follows the more abstract construction of \cite{LekiliPolishchukAuslander}, before checking that this agrees with the predictions of \cref{LUConj}. \\
In \cite{LekiliPolishchukAuslander}, the authors consider chains and cycles of weighted projective lines meeting nodally at a point whose isotropy group is finite and cyclic, and where at most two points of an irreducible component can have non-trivial isotropy group. The mirror symplectic surfaces are constructed by gluing together the mirrors to the irreducible components of the curve via a permutation determined by the stacky structure of the node. \\
Stacky curves in the above presentation were first considered in \cite{STZ} as part of the coherent-constructible correspondence, although the work of Lekili and Polishchuk generalises the class of curves able to be considered by allowing the stacky structure at the nodes to be \emph{non-balanced}. Roughly speaking, the condition that the node be balanced means that the gluing maps of the mirror are the identity. Moreover, allowing non-balanced curves is crucial in the construction of mirrors to symplectic surfaces of genus greater than one, and this will continue to be important for us. \\
In our approach to \cref{InvertiblePolyTheorem}, we expand on the results of Lekili--Polishchuk to allow for the irreducible components to have non-trivial generic stabiliser, and then consider the B--model of \cref{InvertiblePolyTheorem} in the abstract as such a cycle of nodal stacky curves. We then show that the A--model which we construct as its mirror is graded symplectomorphic to the quotient of the Milnor fibre of the transpose polynomial by the corresponding Berglund--Henningson dual group. Fortunately, the construction of the \emph{Auslander sheaf} -- a certain sheaf of non-commutative algebras -- in the orbifold case studied in \cite{LekiliPolishchukAuslander} is robust, and carries through with only minor alteration when one includes non-trivial generic stabilisers. \\

Let $\P_{r_{i,-},r_{i,+}}$ be the orbifold $\P^1$ with orbifold points $(q_{i,-},q_{i,+})$ such that $\Aut\ q_{{i,-}} \simeq\mu_{r_{{i,-}}}$ and $\Aut\ q_{i,+}\simeq\mu_{r_{i,+}}$ (in the chain case we are allowing $r_{1,-}=0$, and/ or $r_{n,+}=0$, so that the corresponding irreducible component is a (stacky) $\A^1$).
\begin{mthm}\label{WrappedTheorem}
	Let $\mathcal{C}$ be the Deligne--Mumford stack such that:
	\begin{itemize}
		\item The coarse moduli space of $\mathcal{C}$ is a cycle or chain of $n$ $\P^1$'s.
		\item Each irreducible component, $\mathcal{C}_i$, has underlying orbifold $\P_{r_{i,-},r_{i,+}}$ and generic stabiliser $\mu_{d_i}$ such that $r_{i,+}d_i=r_{i+1,-}d_{i+1}$ (we allow $r_{1,-}$ and/ or $r_{n,+}=0$ in the case of a chain of curves). 
		\item The node $q_i:=|\mathcal{C}_{i}|\cap|\mathcal{C}_{i+1}|$ has isotropy group $H_i$ and is presented as the quotient of  $\Spec \C[x,y]/(xy)$ by $H_i$, where the action is given by
		\begin{align*}
		h\cdot (x,y)=(\psi_{i,+}(h)x, \psi_{i+1,-}(h)y)
		\end{align*}
		for some surjective $\psi_{i,+}: H_i\rightarrow \mu_{r_{i,+}}$ and $\psi_{i+1,-}:H_i\rightarrow \mu_{r_{i+1,-}}$.
	\end{itemize}
	Then
	\begin{align*}
	D^b(\mathcal{A}_\mathcal{C}-\mathrm{mod})\simeq \mathcal{W}(\Sigma; \Lambda)
	\end{align*}
	is a quasi-equivalence of $\Z$-graded pre-triangulated $A_\infty$-categories over $\C$, where $\Sigma$ is a $\Z$-graded, $b$-punctured surface of genus $g$ such that the genus, boundary components, and collection of stops, $\Lambda$, are determined by the $r_{i,\pm}$, $d_i$, and the local presentation of the nodes as the quotient by $H_i$. 
\end{mthm}
\begin{rmk}
	Note that our presentation agrees with the orbifold case when each $d_i=1$ by observing that one can always arrange the action of $H_i\simeq\mu_{r_{i}}$ to be such that $\psi_{i+1,-}=\text{id}$, and $\psi_{i,+}:\mu_{r_i}\xrightarrow{\wedge^{\kappa_i}}\mu_{r_i}$ for some $\kappa_i\in(\Z/r_{i})^\times$. The notion of a node being \emph{balanced}, mentioned above and appearing in the work of \cite{STZ}, is the condition that $\kappa_i=-1$.
\end{rmk}
\begin{rmk} 
	Whilst we do not use this language, the above result, as well as \cref{HMSTheorem} below, can also be interpreted as an application of \cite{GPSDescent}.
\end{rmk}
It is well-known that passing to the quotient by an abelian group on the B--side corresponds to a cover on the A--side by the dual group; the main difficulty in proving \cref{WrappedTheorem} is in keeping track of these groups and how this structure `fits together' at intersection points of the irreducible components. This is the data of the presentation of the $H_i$-action at a node, and corresponds to a \emph{choice} of gerbe structure on each irreducible component, as reviewed in \cref{root stack section}. As part of the proof, we show that the resulting categories are independent of this choice. \\
Following \cite{LekiliPolishchukAuslander}, when referring to a specific configuration of points on the $b$ boundary components of $\Sigma$, we will denote the partially wrapped Fukaya category by $\mathcal{W}(\Sigma; m_1,m_2,\dots, m_b)$, where $m_i$ is the number of stops on the $i^{\text{th}}$ boundary component. When there are $d$ boundary components with $m$ stops, we shall notate this as $(m)^d$.  \\

By identifying localising subcategories on the A-- and B--sides in \cref{WrappedTheorem}, we get: 
\begin{mthm}\label{HMSTheorem}
	Let $\mathcal{C}$ and $\Sigma$ be as in Theorem \ref{WrappedTheorem}. Then
	\begin{align*}
	\perf \mathcal{C}&\simeq\mathcal{F}(\Sigma)\\
	D^b\Coh\mathcal{C}&\simeq \mathcal{W}(\Sigma),
	\end{align*}
	are quasi-equivalences of $\Z$-graded pre-triangulated $A_\infty$-categories over $\C$ in the case of a cycle of curves. In the case of a chain of curves, there are quasi-equivalences of pre-triangulated $A_\infty$-categories over $\C$
	\begin{align*}
	\perf_c \mathcal{C}&\simeq\mathcal{F}(\Sigma;(r_{1,-})^{d_1},(0)^{b-d_1-d_n},(r_{n,+})^{d_n})\\
	D^b\Coh(\mathcal{C})&\simeq\mathcal{W}(\Sigma;(r_{1,-})^{d_1},(0)^{b-d_1-d_n},(r_{n,+})^{d_n}),
	\end{align*}
	where $\perf_c\mathcal{C}$ is the full subcategory of $\perf\mathcal{C}$ consisting of objects with proper support. 
\end{mthm}
It should be emphasised that the choice of grading on the surface in the above theorems is a crucial piece of data. Changing it would change the grading of the endomorphism algebra of the generating Lagrangians, and, in general, would not yield a derived equivalent algebra. Moreover, taking $\perf_c\mathcal{C}$ in the case of a cycle of curves is only necessary when $r_{1,-}$ and/ or $r_{n,+}=0$. The category $\mathcal{F}(\Sigma;m_1,m_2,\dots,m_n)$ is the \emph{infinitesimally wrapped Fukaya category} of \cite{NadlerZaslow} (cf. \cite{GanatraPardonShendeMicrolocal}).\\

With this in hand, we establish \cref{InvertiblePolyTheorem} by first explicitly exhibiting the B--model as a cycle of weighted projective lines, and then constructing the mirror manifold according to Theorems \ref{WrappedTheorem} and \ref{HMSTheorem}. To finish the proof, we then demonstrate that these manifolds are graded symplectomorphic to the A--models appearing in \cref{InvertiblePolyTheorem}.

\subsection{Structure of paper}
In Section \ref{OrderSection}, we review the theory of Auslander orders over nodal (stacky) curves, incorporating the necessary alterations of the orbifold case to allow for non-trivial generic stabilisers. \cref{Pwrapped Section} is an application of the theory developed in \cite{HKK} to the curves on the A--side which we consider, whilst Sections \ref{LocalisationSection} and \ref{PerfectObjectsSection} exposit the changes to the localisation argument and perfect complex characterisation of \cite{LekiliPolishchukAuslander} required to prove Theorems \ref{WrappedTheorem} and \ref{HMSTheorem}. The computational heart of the paper is in \cref{Applications section}, which is devoted to the proof of \cref{InvertiblePolyTheorem}. We recall the basic constructions of root stacks, both with and without section, in the context of how use them in \cref{root stack section}.
\subsection{Conventions} We work over $\C$ throughout. For a Deligne--Mumford (DM) stack $\mathcal{X}$ we write $x\in\mathcal{X}$ to mean $x:\Spec \C\rightarrow \mathcal{X}$, and let $|\mathcal{X}|$ be its underlying topological space. We refer to a DM stack with trivial generic stabiliser as an orbifold. All Fukaya categories are completed with respect to cones and direct summands.
\subsection{Acknowledgements} 
This work was completed while the author was a PhD student, and he would like to thank his PhD supervisor Yank\i~Lekili for suggesting the project and his guidance throughout.  He would also like to thank Jack Smith, Dougal Davis, Ed Segal, Alessio Corti, and Tim Magee for their interest in this work, and answering the author's many questions during the early stages of this project. This work was supported by the Engineering and Physical Sciences Research Council [EP/L015234/1], The EPSRC Centre for Doctoral Training in Geometry and Number Theory (The London School of Geometry and Number Theory), University College London. The author gratefully acknowledges support from the University of Hamburg and the Deutsche Forschungsgemeinschaft under Germany's Excellence Strategy -- EXC 2121 ``Quantum Universe" -- 390833306.
\section{Auslander orders}\label{OrderSection}
In this section, we give a brief account of the theory of Auslander orders over curves, as introduced in \cite{Burban2009TiltingON} and expanded upon in \cite{LekiliPolishchukAuslander}, before constructing the relevant (mild) generalisation. These are sheaves of non-commutative algebras, initially introduced to study non-commutative resolutions of the subcategory consisting of perfect complexes of the derived category of coherent sheaves on certain curves. \\

Before moving on to the situation we are focusing on, it is instructive to review the non-stacky case, as in \cite{Burban2009TiltingON}. Let $C$ be a chain or cycle of $\P^1$'s joined nodally, and $\pi:\widetilde{C}\rightarrow C$ its normalisation (i.e. a disjoint union of $\P^1$'s). Let $\mathcal{I}$ be the ideal sheaf of the singular locus, and define the sheaf of $\OO_{C}$-algebras
\begin{align}\label{Fdefn}
\mathcal{F}:=\begin{pmatrix}
\mathcal{I}\\
\OO_C
\end{pmatrix}.
\end{align}
One can then define the \emph{Auslander sheaf} as 
\begin{align}\label{ACDefn}
\mathcal{A}_C=\mathcal{E}nd_{\OO_C}(\mathcal{F})=\begin{pmatrix}
\widetilde{\OO}_C & \mathcal{I}\\
\widetilde{\OO}_C & \OO_C
\end{pmatrix},
\end{align}
where $\widetilde{\OO}_C=\pi_*\OO_{\widetilde{C}}$. In \cite{Burban2009TiltingON}, the authors study the category of finitely generated left $\mathcal{A}_C$-modules on the ringed space $(C,\mathcal{A}_C)$. Their main result is that $D^b(\mathcal{A}_C-\mathrm{mod})$ has a tilting object, and is a categorical resolution of $\perf C$. They also show that $D^b\Coh(C)$ is equivalent to the localisation of $D^b(\mathcal{A}_C-\mathrm{mod})$ by a certain subcategory of torsion modules, yielding the sequence
\begin{align}\label{AGSESCategories}
\perf{C}\hookrightarrow D^b(\mathcal{A}_C-\mathrm{mod})\rightarrow D^b(C)
\end{align}
\begin{rmk}
In \cite{Burban2009TiltingON}, the authors work with triangulated categories, however, these have unique dg-enhancements by the work of \cite{OrlovLuntsDG}. In particular, Section 8 of \textit{loc. cit.} shows that all categories considered here, including the categories which result from localisation which we will discuss in Section \ref{LocalisationSection}, have unique dg-enhancements. We are therefore free to work with triangulated categories, since results established here lift to the dg-setting. 
\end{rmk}
In \cite{LekiliPolishchukAuslander}, the authors build on the construction of \cite{Burban2009TiltingON} to allow for the nodes to have stacky structure, meaning that the irreducible components are orbifold curves of the form $\P_{a,b}$, and where two irreducible components meet at an orbifold point. We further extend this approach to allow for the irreducible components to have non-trivial generic stabiliser, although the arguments in \cite{LekiliPolishchukAuslander} carry over to our situation with only minor alterations.  \\

Let $\mathcal{C}$ be as in Theorem \ref{WrappedTheorem} and choose a compatible gerbe structure on each irreducible component, meaning that the local model about $q_{i,\pm}$ is compatible with the maps $\psi_{i,\pm}$. This can always be done by taking the root of a line bundle on $\P_{r_{i,-},r_{i,+}}$ whose restriction under \eqref{projectingMV} yields short exact sequences compatible with the action of the isotropy group at the nodes. Two compatible gerbe structures on an irreducible component will differ by how the two patches are identified on overlaps, but by \eqref{IshiiUedaRootStack}, this does not affect our theory. To ease notation, we let $\P_i=\P_{r_{i,-},r_{i,+}}$  be the rigidified $i^{\text{th}}$ irreducible component of $\mathcal{C}$. Let 
\begin{align*}
\pi:\widetilde{\mathcal{C}}=\bigsqcup_{i=1}^n\widetilde{\mathcal{C}}_i\rightarrow \mathcal{C}
\end{align*}
 be the normalisation map, and $H_i$ the isotropy group at the node $q_i=|\mathcal{C}_i|\cap| \mathcal{C}_{i+1}|$, and $H_0$ and $H_n$ the isotropy groups of the points $q_{1,-}$ and $q_{n,+}$, respectively, in the chain case. At the points $q_{i,+}$ and $q_{i+1,-}$, there are, by construction, short exact sequences
\begin{align}
&1\rightarrow \mu_{d_{i}}\rightarrow H_{i,+}\xrightarrow{\psi_{i,+}} \mu_{r_{i,+}}\rightarrow 1,\quad\text{and} \label{SES1}\\
&1\rightarrow \mu_{d_{i+1}}\rightarrow H_{i+1,-}\xrightarrow{\psi_{i+1,-}} \mu_{r_{i+1,-}}\rightarrow 1 \label{SES2}.
\end{align}
There are (non-canonical) isomorphisms $H_i\simeq H_{i,+}\simeq H_{i+1,-}$, although by choosing the representatives of \eqref{SES1} and \eqref{SES2} such that $\psi_{i,+}$ (resp. $\psi_{i+1,-}$) are as in Theorem \ref{WrappedTheorem}, one can take these identifications to be the identity. This yields the local model of $q_i$.
\begin{rmk}
It should be emphasised that, even when it would make sense, we do not require that the short exact sequences \eqref{SES1}, \eqref{SES2} are equivalent such that $H_{i,+}\simeq H_{i+1,-}$ via the identity map, only that the groups $H_{i,+}$ and $H_{i+1,-}$ can be identified with $H_i$. 
\end{rmk}
Recall that the ideal sheaf of a closed substack is the sheaf which pulls back to the ideal sheaf of the preimage in any atlas. As such, we define
\begin{align*}
\mathcal{I}=\bigoplus_{i=1}^n \pi_{i*}\OO_{{\widetilde{\mathcal{C}}}_i}(-q_{i,-}-q_{i,+})
\end{align*}
for a cycle of curves, and analogously for a chain. Here $\pi_i:\widetilde{\mathcal{C}}_i\rightarrow \mathcal{C}$ is again the natural projection. We let $\mathcal{F}$ be as in \eqref{Fdefn} and  $\mathcal{A}_{\mathcal{C}}$ be as in \eqref{ACDefn}. For any integers $j$, $m$, and $k\in\{0,\dots, d_i-1\}$, we define distinguished $\mathcal{A}_{\mathcal{C}}$-modules
\begin{align*}
\mathcal{P}_i(j,m,k)=\begin{pmatrix}
\pi_{i*}\big(\OO_{\widetilde{\mathcal{C}}_i}(jq_{i,-}+mq_{i,+})\otimes \mathcal{N}_i^{\otimes k}\big)\\
\pi_{i*}\big(\OO_{\widetilde{\mathcal{C}}_i}(jq_{i,-}+mq_{i,+})\otimes \mathcal{N}_i^{\otimes k}\big)
\end{pmatrix}.
\end{align*}
For fixed integers $j$, $m$, and $0\leq k\leq d_i-1$, let $\Exc_i(j,m,k)$ be the collection 
\begin{equation}\label{Auslander exceptional collection}
\begin{tikzcd}
\mathcal{P}_i(j,m,k) \arrow[r,"x_i"] \arrow[d,equal] & \mathcal{P}_i(j+1,m,k) \arrow[r,"x_i"] & \dots \arrow[r,"x_i"] & \mathcal{P}_i(j+r_{i,-}-1,m,k) \arrow[r,"x_i"] & \mathcal{P}_i(j+r_{i,-},m,k)  \arrow[d,equal]\\
\mathcal{P}_i(j,m,k) \arrow[r,"y_i"] &\mathcal{P}_i(j,m+1,k) \arrow[r,"y_i"] & \dots \arrow[r,"y_i"] & \mathcal{P}_i(j,m+r_{i,+}-1,k) \arrow[r,"y_i"]& \mathcal{P}_i(j,m+r_{i,+},k)
\end{tikzcd}
\end{equation}
Note that by the decomposition \eqref{IshiiUedaRootStack}, we have that $\mathcal{P}_i(j,m,k)$ is orthogonal to $\mathcal{P}_{i'}(j',m',k')$ unless $k=k'$. With this, it follows directly from the proof of \cite[Lemma 1.2.1]{LekiliPolishchukAuslander} that the modules $\mathcal{P}_i(j,m,k)$ are exceptional, and $\Exc_i(j,m,k)$ is an exceptional collection for any fixed $j,\ m$, and $k\in\{0,\dots, d_i-1\}$. In the case of $d_{i}=1$ we omit $k$ from the notation.\\

As in the non-stacky and orbifold cases we also define simple modules at each node, given by 
\begin{align*}
\mathcal{S}_q=\begin{pmatrix}
0\\
\OO_q
\end{pmatrix}.
\end{align*}

Fixing an identification of the isotropy group of the node $q_i=|\mathcal{C}_i|\cap|\mathcal{C}_{i+1}|$ with $H_{i}$ (for $i$ counted modulo $n$ in the cycle case), let $\psi_{i,+}$ and $\psi_{i+1,-}$ be as in Theorem \ref{WrappedTheorem} and fit into the short exact sequences \eqref{SES1} and \eqref{SES2}, respectively. We have that locally, around $q_i$, we can view $\mathcal{A}_{\mathcal{C}}$--modules as equivariant $H_{i}$ modules on $\Spec \C[x,y]/(xy)=\Spec S$, where the $H_i$ action is given by $h\cdot (x,y)=(\psi_{i,+}(h)x,\psi_{i+1,-}(h)y)$. We fix the $\mu_{r_{i+1,-}}$ action on the fibre of the sheaf $\OO_{\P_{i+1}}(-q_{i+1,-}) $ at $q_{i+1,-}$ to be via its natural character, and similarly for the action of $\mu_{r_{i,+}}$ on the fibre of the sheaf $\OO_{\P_{i}}(-q_{i,+})$ at $q_{i,+}$. Moreover, we define the character corresponding to the weight of the action of $H_i$ on the fibre of $\OO_{\widetilde{\mathcal{C}}_{i+1}}(-q_{i+1,-})$ to be the character induced from the natural character of $\mu_{r_{i+1,-}}$ under the dual of $\psi_{i+1,-}$, and we call this $\chi_{r_{i+1,-}}$. We define $\chi_{r_{i,+}}$ similarly as the character of $H_i$ induced by the natural character under the dual of $\psi_{i,+}$. For the chosen gerbe structure, choose $\chi_{d_{i,\pm}}$ such that $d_{i,\pm}\chi_{d_{i,\pm}}=\chi_{r_{i,\pm}}$ as in \cref{root stack section}.\\

 Since $H_{i}$ is diagonalisable (is isomorphic to a closed subgroup of an algebraic torus), we have an eigenspace decomposition of an $H_{i}$-equivariant $S$-module $M$ as 
\begin{align*}
M=\bigoplus_{\chi\in \widehat{H}_{i}}M_{\chi},
\end{align*}
where $\widehat{H}_{i}$ is the group of characters of $H_{i}$. Furthermore, for any $\chi\in \widehat{H}_{i}$ there is a twisting operation $M\mapsto M\{\chi\}$ which identifies the $\gamma$-eigenspace of $M\{\chi\}$ with the $(\chi+\gamma)$-eigenspace of $M$. \\

For a chain (resp. cycle) of nodal stacky curves, consider a tuple of characters $\boldsymbol{\chi}=(\chi_0,\dots,\chi_{n+1})\in \widehat{H}_0\oplus \dots \oplus \widehat{H}_{n+1}$ (resp. $\boldsymbol{\chi}=(\chi_1,\dots,\chi_{n})\in \widehat{H}_1\oplus\dots\oplus\ \widehat{H}_{n}$). We call such a tuple \emph{admissible} if there exists a line bundle $\OO_{{\widetilde{\mathcal{C}}}_i}(jq_{i,-}+ mq_{i,+})\otimes \mathcal{N}_i^{\otimes k}$ on each $\widetilde{\mathcal{C}}_i$ such that $H_{i-1}$ acts on the fibre at $q_{i,-}$ with character $\chi_{i-1}$ and $H_i$ on the fibre at $q_{i,+}$ with character $\chi_i$. Denote by $\widehat{H}_{\text{ad}}$ the set of admissible characters. It is not true that $\widehat{H}_{\text{ad}}$ contains any tuple of characters; however, for any character $\chi\in \widehat{H}_i$ there is a tuple in $\widehat{H}_{\text{ad}}$ such that $\chi_i=\chi$. For each admissible $\boldsymbol{\chi}$, we define the sheaf $\mathcal{M}\{\boldsymbol{\chi}\}$ by gluing the line bundles of the above form together at the nodes. \\

Consider the map $p:\mathcal{C}\rightarrow C$, where $C$ is the coarse moduli space of the stacky curve, i.e. is a chain or cycle of $\P^1$ joined nodally. Following \cite{OlssonStarr}, we call a sheaf $\mathcal{E}$ on $\mathcal{C}$ an \emph{generator} of $\text{QCoh}(\mathcal{C})$ with respect to $p$ if the natural map 
\begin{align*}
p^*(p_*\mathcal{H}om_{\OO_{\mathcal{C}}}(\mathcal{E},\mathcal{G}))\otimes \mathcal{E}\rightarrow \mathcal{G}
\end{align*}
is a surjection for any $\mathcal{G}$. 

\begin{lem}
The sheaf
\begin{align*}
\bigoplus_{\boldsymbol{\chi}\in\widehat{H}_{\text{ad}}} \mathcal{M}\{\boldsymbol{\chi}\}
\end{align*}
is a generator of $\text{QCoh}(\mathcal{C})$ with respect to $p$.
\end{lem}
\begin{proof}
Let $x$ be a point of $\mathcal{C}$, considered as a map $x:\Spec \C\rightarrow \mathcal{C}$. Let $G_x$ be its isotropy group, and denote by $\tilde{x}:BG_x\rightarrow \mathcal{C}$ the corresponding natural map. Then, \cite[Theorem 5.2]{OlssonStarr} stipulates that if $\mathcal{E}$ is a locally free sheaf such that $\tilde{x}^*\mathcal{E}$ contains every irreducible representation of $G_x$ for every geometric point $x$, then $\mathcal{E}$ is a generator of $\text{QCoh}(\mathcal{C})$ with respect to $p$. \\

From the fact that for each $\chi\in\widehat{H}_i$ there is a $\boldsymbol{\chi}\in\widehat{H}_{\text{ad}}$ such that $\chi_i=\chi$, it is clear that the fibre of $\bigoplus_{\boldsymbol{\chi}}\mathcal{M}\{\boldsymbol{\chi}\}$ at any nodal point (as well as at $q_{1,-}$ and $q_{n,+}$ in the chain case) contains every irreducible representation of $H_i$. Since $\chi_{d_i}$ pushes down to a generator of $\Z/d_{i}$, the fibre of  $\bigoplus_{\boldsymbol{\chi}}\mathcal{M}\{\boldsymbol{\chi}\}$ at a point whose isotropy group is $\mu_{d_i}$ contains every irreducible representation of $\mu_{d_i}$, and this establishes the claim. 
\end{proof}

To calculate the morphisms between the modules $\mathcal{S}_{q_i}$, and their twists $\mathcal{S}_{q_i}\{\chi\}$ for $\chi\in\widehat{H}_i$, with the $\mathcal{P}_i(j,m,k)$, we can work locally in the patch $U=\Spec S$, as above, and consider $H_i$ equivariant $\mathcal{A}_U$-modules. We begin by observing that, as in the non-stacky and orbifold cases, the only relevant Ext-class is given by the short exact sequence of $\mathcal{A}_{U}$-modules 
\begin{align}\label{ExtClassSES}
0\rightarrow\begin{pmatrix}
I\\
I
\end{pmatrix}\rightarrow \begin{pmatrix}
I\\
\mathcal{O}_{U}
\end{pmatrix}\rightarrow \mathcal{S}_{q_i}\rightarrow 0,
\end{align}
and that this class is $H_{i}$-equivariant. Therefore, we have morphisms
\begin{align*}
\Ext^1(S_{q_i},\mathcal{P}_i(j,m,0))=a_i(m,0)\\
\Ext^1(S_{q_i},\mathcal{P}_{i+1}(j,m,0))=b_i(j,0)
\end{align*}
for any $m\equiv -1\bmod r_{i,+}$, and $j\equiv -1\bmod r_{i+1,-}$, respectively. Consider $\mathcal{M}\{\boldsymbol{\chi}\}$ such that the character at $q_i$ is $\chi_i$. It is clear that we have
\begin{align*}
\mathcal{S}_{q_i}\otimes\mathcal{M}\{\boldsymbol{\chi}\}\simeq \mathcal{S}_{q_i}\{\chi_i\}.
\end{align*}
In particular, as in \eqref{SheafWeight}, we have that for each $\chi\in\widehat{H}_i$, and any $m_i,\ j_i,\ m_{i+1},\ j_{i+1}\in \Z$, there exists $m\in\{m_i,\dots,m_i+r_{i,+}-1\},\ k_+\in\{0,\dots d_i-1\}$ and $j\in\{j_{i+1},\dots, j_{i+1}+r_{i+1,-}-1\},\ k_-\in\{0,\dots,d_{i+1}-1\}$ such that $H_i$ acts on the fibres of the sheaves
\begin{align*}
\OO_{\widetilde{\mathcal{C}}_i}(mq_{i,+})\otimes \mathcal{N}_i^{\otimes k_+},\\
\OO_{\widetilde{\mathcal{C}}_{i+1}}(jq_{i+1,-})\otimes\mathcal{N}_{i+1}^{\otimes k_-}
\end{align*}
at $q_{i,+}$ and $q_{i+1,-}$, respectively, with character $\chi$. \\

By twisting the sequence \eqref{ExtClassSES} by $\mathcal{M}\{\boldsymbol{\chi}\}$, we obtain morphisms
\begin{align}\label{morphismsInEnd}
\begin{split}
&\Ext^1(S_{q_i}\{\chi\},\mathcal{P}_i(j_i,m_0+m,k_{+}))=\C\cdot a_i(m,k_{+}),\quad \text{and}\\
&\Ext^1(S_{q_i}\{\chi\},\mathcal{P}_{i+1}(j_0+j,m_{i+1},k_{-}))=\C\cdot b_i(j,k_{-}),
\end{split}
\end{align}
for each $\chi\in \widehat{H}_i$, where $m_0\in \{m_i,\dots,m_i+r_{i,+}-1\}$ is a distinguished element such that \\$m_0\equiv -1\bmod r_{i,+}$, and $(m,k_{+})$ as above solves
\begin{align}\label{WeightMatching+}
-m\chi_{r_{i,+}}+k_{+}\chi_{d_i,+}=\chi,
\end{align} 
 $j_0\in \{j_{i+1},\dots,j_{i+1}+r_{i+1,-}-1\}$ is a distinguished element such that $j_0\equiv -1\bmod r_{i+1,-}$, and $(j,k_{-})$ as above solves
\begin{align}\label{WeightMatching-}
-j\chi_{r_{i+1-}}+k_{-}\chi_{d_{i+1},-}=\chi.
\end{align} 
Now, we have constructed a full, strong exceptional collection consisting of the objects:
\begin{itemize}
	\item For any fixed $j_i,\ m_i\in \Z$, and each irreducible component, being a $\mu_{d_i}$-gerbe over $\P_{i}$, the collections 
	\begin{align*}
	\bigoplus_{k=0}^{d_i-1}\Exc_i(j_i,m_i,k),
	\end{align*} 
	\item For each node $q_i=|\mathcal{C}_{i}|\cap|\mathcal{C}_{i+1}|$, the objects
	\begin{align*}
	\mathcal{S}_{q_i}\{\chi_k\}\quad\text{for each }\chi_k\in \widehat{H}_i.
	\end{align*}
\end{itemize}
The endomorphism algebra of this collection is generated by the morphisms $x_i$, $y_i$ in \eqref{Auslander exceptional collection}, as well as the morphisms given by \eqref{morphismsInEnd}. The relations are $ya=0$ and $xb=0$ whenever the composition is possible. The proof of this, as well as the claim that the collection is full and strong, can be seen from following through the proof of \cite[Theorem 1.2.3]{LekiliPolishchukAuslander} \textit{mutatis mutandis} (cf. \cite[Theorem 9]{Burban2009TiltingON}). Of course, the resulting category $D^b(\mathcal{A}_\mathcal{C}-\mathrm{mod})$ only depends on the parameters stated in Theorem \ref{WrappedTheorem}, ultimately for the same reason as $D^b\Coh(\mathcal{C})$ does. 

\section{The partially wrapped Fukaya category}\label{Pwrapped Section}
Partially wrapped Fukaya categories were first defined by Auroux in \cite{Auroux_fukayacategories} and further developed by Sylvan in \cite{SylvanPwrapped}. In this section, we briefly recount the construction of the surfaces under consideration (\cite[Section 3.2]{HabermannHMS}) as well as the strategy of \cite{HKK} for computing the partially wrapped Fukaya category of a graded symplectic surface.
\subsection{Gluing annuli}\label{GluingAnnuli}
Let $A(\ell,r;d)$ denote $d$ annuli, each with $r$ ordered marked points, $p_{rk}^+,\dots p_{r(k+1)-1}^+$, on the first boundary component, and $\ell$ ordered marked points, $p_{\ell k}^-,\dots, p_{\ell(k+1)-1}^-$, on the second boundary component, which have been placed in a column. Here $k\in\{0,\dots,d-1\}$  refers to which annulus the marked points are on, where we count top-to-bottom.  We visualise this as $d$ disjoint rectangles which have been placed in a column, each rectangle has top and bottom identified, and with the left boundary components containing the points $p_{\ell k+i}^-$ and the right boundary components containing the points $p_{r k+i}^+$. The reasoning for the labelling is that we would like to keep track of where the marked points are on each individual annulus, as well as where each marked point is on the left (respectively right)  side of the column of annuli with respect to the ordering $p_0^-,\dots, p_{d_i\ell_i-1}^-$ (respectively $p_0^+,\dots, p_{d_ir_i-1}^+$). \\

Given a collection of annuli 
\begin{align*}
A(\ell_1,r_1;d_1),\dots, A(\ell_{n},r_{n};d_{n}),
\end{align*}
such that $r_id_i=\ell_{i+1}d_{i+1}$, and corresponding permutations $\sigma_i\in \mathfrak{S}_{d_ir_i}$, we can glue these annuli together in the following way. For each $j\in\{0,\dots,d_ir_i-1\}$, we glue a small neighbourhood around the stop $p_j^+$ in $A(\ell_i,r_i;d_i)$ to  a small neighbourhood around the stop $p_{\sigma_i(j)}^-$ in $A(\ell_{i+1},r_{i+1};d_{i+1})$ by attaching a strip. We call such a gluing \emph{circular} if $A(\ell_{n},r_{n};d_{n})$ is glued back to $A(\ell_1,r_1;d_1)$, and in this case we count $i\bmod n$ -- see Figure \ref{fig:constructionexample} for an example. Otherwise, we call a gluing \emph{linear}, and take $i\in\{1,\dots, n\}$. In the case of linear gluing we refer to the left boundary components of $A(\ell_1,r_1,d_1)$ and the right boundary components of $A(\ell_n,r_n,d_n)$ as the left and right distinguished boundary components, respectively.\\

For each $i$, the number of boundary components arising from gluing the $i^{\text{th}}$ and $(i+1)^{\text{st}}$ columns can be computed as follows. Consider the permutations
\begin{align*}
\tau_{r_i}=\big(0,r_i-1, r_i-2,\dots,1\big)\big(r_i,2r_i-1,2r_i-2,\dots, r_i+1\big)\dots\big((d_i-1)r_i,m_ir_i-1,\dots, (d_i-1)r_i+1\big)
\end{align*}
and 
\begin{align*}
\tau_{\ell_i}=\big(0,1,\dots,\ell_{i+1}-1\big)\big(\ell_{i+1},\dots, 2\ell_{i+1}-1\big)\dots\big((d_{i+1}-1)\ell_{i+1},\dots, d_{i+1}\ell_{i+1}-1\big).
\end{align*}
The number of boundary components between the $i^{\text{th}}$ and $(i+1)^{\text{st}}$ columns will then be given by the number of cycles in the decomposition of $\sigma_i^{-1}\tau_{\ell_{i+1}}\sigma_i\tau_{r_i}\in\mathfrak{S}_{m_ir_i}$. Note that if $d_i=d_{i+1}$, then we have $\tau_{r_i}=\tau_{\ell_{i+1}}^{-1}$, and we simply get the commutator. When there is no risk of confusion we will simply refer to the surface which has been constructed as $\Sigma$. \\

\begin{figure}[h]
	\includegraphics[width=0.8\linewidth]{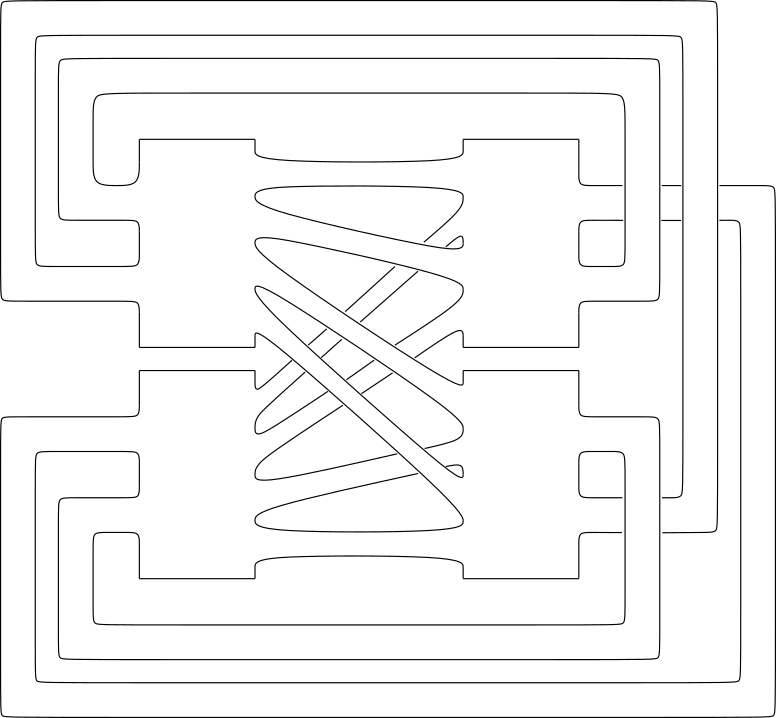}
	\caption{A genus 5 surface with 4 boundary components constructed by gluing $A(2,4; 2)$ to $A(4,2;2)$ via the permutations $\sigma_1=\protect\begin{psmallmatrix}
		0&1&2&3&4&5&6&7\\
		0&2&4&6&1&3&5&7
		\protect\end{psmallmatrix}$ and $\sigma_2=\protect\begin{psmallmatrix}
		0&1&2&3\\
		2&0&3&1
		\protect\end{psmallmatrix}$.} 
	\label{fig:constructionexample}
\end{figure}

To compute the homology groups of $\Sigma$, one can construct a ribbon graph
\begin{align}\label{graph}
\Gamma(\ell_1,\dots,\ell_n;r_1,\dots,r_n;m_1,\dots,m_n;\sigma_1,\dots,\sigma_n)\subseteq \Sigma,
\end{align}
on to which the surface deformation retracts. To do this, let there be a topological disc $\mathbb{D}^2$ for each of the annuli. For each disc, attach a strip which has one end on the top, and the other end on the bottom. Then, attach a strip which connects two discs if there is a strip which connects the corresponding annuli. These strips must be attached in such a way as to respect the cyclic ordering given by the gluing permutation. One can then deformation retract this onto a ribbon graph, whose cyclic ordering at the nodes is induced from the ordering of the strips on each annulus. If there is no ambiguity, we will refer to this graph as $\Gamma(\Sigma)$. \\

Since the embedding of $\Gamma(\Sigma)$ in to $\Sigma$ induces an isomorphism on homology, the homology groups of $\Sigma$ can be easily computed. Namely, since the graph is connected, we have $H_0(\Sigma)=\Z$. For circular gluing we have $\chi(\Sigma)=V-E=\text{rk}H_0(\Sigma)-\text{rk}H_1(\Sigma)=-\sum_{i=1}^nr_id_i$, and for linear gluing we have $\chi(\Sigma)=-\sum_{i=1}^{n-1}r_id_i$, yielding $H_1(\Sigma)=\Z^{(1-\chi)}$ in both cases. A basis for the first homology of the graph is given by an integral cycle basis, and so the basis of the first homology for $\Sigma$ is given by loops which retract onto these cycles. \\

A $\Z$-grading of the surface is given by a homotopy class of (unorientated) line field, as explained in \cite[Section 13(c)]{SeidelBook}, and a Lagrangian is gradable with respect to this line field if and only if its winding number is zero. Note that in the case where the winding number with respect to a given line field vanishes on each embedded Lagrangian in a basis of the first homology of $\Sigma$, the homotopy class of the line field is unique. In the remainder of this paper, we will only consider the case where the line field used to grade the surface is given by the horizontal line field on each annulus, and is parallel to the boundary components on each attaching strip. With this, we see that the line field comes from the projectivisation of a vector field by the same proof as in \cite[Lemma 4.1.1]{LPGentle}. With such a description of a surface, it is possible to determine when two surfaces are graded symplectomorphic (\cite[Corollary 1.2.6]{LPGentle}); however, in order to do this one must (in many cases) compute the Arf invariant. Whilst this is theoretically simple, it is computationally intractable to do in generality without imposing restrictions on the form of the gluing permutations being considered, as in \cite[Section 4.3]{LPGentle}, or \cite[Section 4.3]{HabermannHMS}. 
\subsection{Computation of the partially wrapped Fukaya category}\label{ComputationPwrapped}
Once we have constructed the surfaces in question, our approach to mirror symmetry involves computing the partially wrapped Fukaya category via the method given in \cite{HKK}.\\

Given a surface with non-empty boundary, $\Sigma$, and a collection of stops on its boundary $\Lambda$, \cite[Section 3]{HKK} shows that if $\{L_i\}$ is a collection of pairwise disjoint and non-isotopic Lagrangians such that $\Sigma\setminus\big(\sqcup_{i}L_i\big)$ is topologically a union of discs, each of which with exactly one marked point on its boundary, then the $L_i$ generate $\mathcal{W}(\Sigma;\Lambda)$. Moreover, it is also shown that the total endomorphism algebra of the generators is formal, and can be described as the algebra of a quiver with monomial relations. A connection to the representation theory of finite dimensional algebras is given by the observation that the endomorphism algebra of such a generating collection of objects is \emph{gentle}, a class of finite dimensional algebras first introduced in \cite{AssemSkowronski}. These continue to be of interest to representation theorists, particularly for their relationship with Fukaya categories of surfaces -- see, for example, \cite{Opper2018AGM}, \cite{Amiot2019ACD}.\\

To construct the partially wrapped Fukaya category, it was shown that there exists a ribbon graph dual to the collection of Lagrangians. This graph has an $n-$valent vertex at the centre of each $2n$-gon cut out by the Lagrangians, and the half edges connect two vertices if that edge is dual to a Lagrangian on the boundary of both of the corresponding discs. The cyclic ordering is induced from the orientation of the surface. From this, it was shown in \cite[Theorem 3.1]{HKK} that the partially wrapped Fukaya category is given as the global sections of a constructible cosheaf of $A_\infty$-categories on the ribbon graph. In particular, for each $n$-valent vertex at the centre of a $2n$-gon, there is a fully faithful inclusion functor
\begin{align}\label{CosheafInclusion}
\mathcal{W}(\mathbb{D}^2;n+1)\rightarrow \mathcal{W}(\Sigma;\Lambda),
\end{align} 
where $\mathcal{W}(\mathbb{D}^2;n+1)$ is the partially wrapped Fukaya category of the disc with $n+1$ stops on its boundary.\\

The two prototypical examples from which our strategy is built are the disc with $m$ points on its boundary, as well as the cylinder with $a$ stops on one boundary, and $b$ stops on the other. Consider the disc with $m$ stops on its boundary, and $m-1$ Lagrangians, $L_1,\dots, L_{m-1}$ supported near these stops, as in Figure \ref{DiscGeneratingLags}. The morphisms between Lagrangians is given by the Reeb flow along the boundary of the disc in the anticlockwise direction. Let $a_i:L_i\rightarrow L_{i+1}$ be such a morphism, and observe that $a_{i+1}a_i=0$ for $i=1,\dots, m-2$. It is clear that the endomorphism algebra of the direct sum $\bigoplus_{i=1}^{m-1}L_i$ is the $A_{m-1}$ quiver with relations given by disallowing any composition. 

\begin{figure}[h!]
\begin{tikzpicture}[scale=.99]
\draw (0,0) circle (2cm);
\fill[black] (-30:2cm) circle[radius=2pt];
\draw (-35: 2.5) node["$m$"]{};
\fill[black] (30:2cm) circle[radius=2pt];
\fill[black] (90:2cm) circle[radius=2pt];
\fill[black] (210:2cm) circle[radius=2pt];
\fill[black] (270:2cm) circle[radius=2pt];
\draw[blue, thick]  (50:2cm) .. controls (40:1.75cm) and (20:1.75cm) .. (10:2cm);
\draw (20: 1.25) node["$L_{1}$"]{};
\draw (20: 2.2) node["$1$"]{};
\draw[blue, thick]  (110:2cm) .. controls (100:1.75cm) and (80:1.75cm) .. (70:2cm);
\draw (90: 1.1) node["$L_{2}$"]{};
\draw (90: 2) node["$2$"]{};
\draw[blue, thick]  (230:2cm) .. controls (220:1.75cm) and (200:1.75cm) .. (190:2cm);
\draw (230: 1.5) node["$L_{m-2}$"]{};
\draw (215: 2.9) node["$m-2$"]{};
\draw[blue, thick]  (290:2cm) .. controls (280:1.75cm) and (260:1.75cm) .. (250:2cm);
\draw (270: 1.85) node["$L_{m-1}$"]{};
\draw (270: 2.7) node["$m-1$"]{};
\fill[black] (135:2.2) circle[radius=1.5pt];
\fill[black] (145:2.2) circle[radius=1.5pt];
\fill[black] (155:2.2) circle[radius=1.5pt];
\draw (155: 1.5) node["$L_{\bullet}$"]{};
\end{tikzpicture}
		\caption{A collection of generating Lagrangians for $\mathbb{D}^2$ with $m$ stops. The Reeb flow is in the counterclockwise direction.\label{DiscGeneratingLags} }
\end{figure}
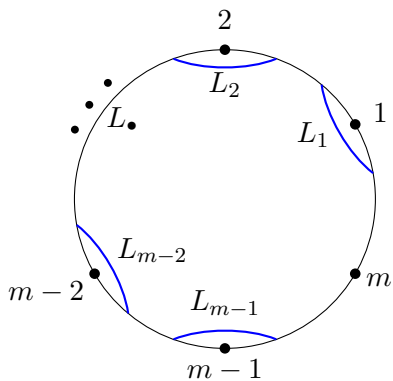
There are two key facts about the collection of Lagrangians $L_1,\dots,L_{m-1}$. The first is that the Lagrangian $L_m$ is quasi-isomorphic to the twisted complex
\begin{equation}\label{LESQuasiIsom}
\begin{tikzcd}
L_1[m-2] \arrow[r] & L_2[m-3]\arrow[r]&\dots\arrow[r] & L_{m-2}[-1] \arrow[r] & L_{m-1}.
\end{tikzcd}
\end{equation}
This is first observed in \cite[Section 3.3]{HKK}, and will be important later in our localisation argument. The second key observation is that the complement $\mathbb{D}^2\setminus\big(\sqcup_{i=1}^{m-1}L_i\big)$ is a collection of topological discs, each with exactly one marked point on the boundary. Therefore, the collection $\{L_1,\dots, L_{m-1}\}$ generates the partially wrapped Fukaya category $\mathcal{W}(\mathbb{D}^2;m)$. \\

The second fundamental example which forms the cornerstone of our strategy is the annulus, $A$, with $a$ stops on one boundary component, and $b$ on the other. A generating collection of Lagrangians on such an annulus is given in Figure \ref{HKKcollection example}, and its corresponding quiver in Figure \ref{PwrapQuiver}. Observe that the quiver algebra of the generators for the single annulus with $a$ stops on one boundary component and $b$ on the other matches precisely the quiver algebra of the exceptional collection of $\P_{a,b}$ given in \eqref{exceptional collection}. This establishes that
\begin{align*}
\mathcal{W}(A;a,b)\simeq D^b\Coh(\P_{a,b}),
\end{align*}
and this observation is at the heart of our strategy. 
\begin{figure}[h!]
	\begin{minipage}{0.5\textwidth}
		\centering
		\begin{tikzpicture}[blob/.style={circle, draw=black, fill=black, inner sep=0, minimum size=\blobsize}, arrow/.style={->, shorten >=6pt, shorten <=6pt}, scale=0.9]
		\def\blobsize{1.2mm}
		\draw (0, 1.5) node[blob]{};
		\draw (0, 4) node[blob]{};
		\draw (0, 3.4) node[blob]{};
		\draw (5, 1) node[blob]{};
		\draw (5, 1.6) node[blob]{};
		\draw (5, 4) node[blob]{};
		\draw (5, 3.4) node[blob]{};
		
		\draw [blue, thick] (0, 0) -- (5, 0);
		\draw [blue, thick] (0, 5) -- (5, 5);
		\draw (0, 0) -- (0, 2);
		\draw[->-=.25] (0,5) to (0,3);
		\draw[dashed] (0,2) -- (0,3);
		\draw[->-=.25] (5, 0) to (5, 2);
		\draw (5,3) -- (5,5);
		\draw[dashed] (5,2) -- (5,3);
		\draw[blue, thick] (0,0.3) -- (5,1.3);
		\draw[blue, thick] (0,0.6) -- (5,2.5);
		\draw[blue, thick] (0,0.9) -- (5,3.7);
		\draw[blue, thick] (0,1.2) -- (5,4.4);		
		\draw[blue, thick] (0,3.7) -- (5,4.8);		
		\draw[blue, thick] (0,2.5) -- (5,4.6);

		\draw (1.5,2) node["$P_{0}^{\pm}$"]{};
		\draw (.75,2.75) node["$P_{\bullet}^{-}$"]{};
		\draw (.75,3.7) node["$P_{a-1}^{-}$"]{};	
		\draw (2.5,4.2) node["$P_{a}^{-}$"]{};	
		\draw (4.5,2.5) node["$P_{1}^{+}$"]{};	
		\draw (4.5,1.4) node["$P_{\bullet}^{+}$"]{};	
		\draw (4.5,.2) node["$P_{b-1}^{+}$"]{};	
		\draw (2.5,-.25) node["$P_{b}^{+}$"]{};	
		\end{tikzpicture} 
		\caption{A collection of generating Lagrangians for $A(a,b;1)$. Top and bottom identified. \label{HKKcollection example}}
	\end{minipage}
	\begin{minipage}{0.475\textwidth}
		\centering
		\begin{tikzcd}[scale=0.4]
			P_{0}^- \arrow[r,"x_1"] \arrow[d,equal] & P_1^- \arrow[r,"x_2"] & \dots \arrow[r,"x_{a-1}"]&P_{a-1}^- \arrow[r,"x_a"] &P_a^-\arrow[d,equal]\\
			P_{0}^+ \arrow[r,"y_1"] & P_1^+ \arrow[r,"y_2"] &\dots \arrow[r,"y_{b-1}"]&P_{b-1}^+\arrow[r,"y_{b}"] & P_{b}^{+}
		\end{tikzcd}
		\caption{Quiver for $A(a,b;1)$. \label{PwrapQuiver}}
	\end{minipage}
\end{figure}
\vspace{-2.5mm} 
\subsubsection{Circular Gluing} Here we compute the partially wrapped Fukaya category for columns of annuli glued circularly, with notation as in Section \ref{GluingAnnuli}. To begin with, we add two stops on each attaching strip -- one on the top, and one on the bottom. We will refer to this collection as $\Lambda$. On the $k^\text{th}$ annulus in the $i^\text{th}$ column we have a collection of Lagrangians $\boldsymbol{P}_{i,k}$ of the same form as in Figure \ref{HKKcollection example}. This collection consists of the objects
\begin{align*}
\{P_{i,0,k}^{+},P_{i,1,k}^+,\dots,P_{i,r_i,k}^+,P_{i,0,k}^-,\dots P_{i,\ell_i,k}^-\}.
\end{align*}
The morphisms within this collection are of the same form as in Figure \ref{PwrapQuiver}. For each attaching strip, we consider a Lagrangian which spans it in such a way that the two stops are in the clockwise direction of its endpoints. We label the Lagrangian which spans the attaching strip beginning at the neighbourhood around the $j^{\text{th}}$ stop between the $i^{\text{th}}$ and $(i+1)^{\text{st}}$ columns by $S_{i,j}$. Here $j\in\{0,\dots, r_{i}m_i-1\}$ and $i\in\Z/n$.\\

As well as the morphisms within each collection $\boldsymbol{P}_{i,k}$, if we write $j=k_+r_i+c_+$ and $\sigma(j)=k_-\ell_{i+1}+c_-$ for $k_+\in\{0,\dots, d_i-1\}$, $c_+\in\{0,\dots,r_{i}-1\}$, $k_-\in\{0,\dots d_{i+1}-1\}$, and $c_-\in\{0,\dots,\ell_{i+1}-1\}$, we also have morphisms
\begin{align*}
a_{i,j}:S_{i,j}&\rightarrow P_{i,c_+,k_+}^+\\
b_{i,j}: S_{i,j}&\rightarrow P_{i+1,\ell_{i+1}-1-c_-,k_-}^-.
\end{align*}

By construction, the complement of this collection of Lagrangians is the disjoint union of hexagons, each with exactly one stop on its boundary. Therefore, we have that the collection of Lagrangians consisting of all of the $\boldsymbol{P}_{i,k}$, as well as the $S_{i,j}$ is a generating collections of Lagrangians for $\mathcal{W}(\Sigma; \Lambda)$.

\subsubsection{Linear Gluing} The case of linear gluing is almost identical to that of circular gluing; however, the first and last columns are now no longer glued to each other. Due to this, we include the stops on the distinguished boundary components in $\Lambda$, although we allow the number of stops on the distinguished boundary components to be empty. In dividing the surface into topological discs for the computation of the partially wrapped Fukaya category, observe that a topological disc with a Lagrangian $S_{i,j}$ on its boundary is a hexagon, as in the circular gluing case, and a quadrilateral otherwise. The generating collection is again given by all of the $\boldsymbol{P}_{i,k}$, as well as the $S_{i,j}$. See Figure \ref{GeneratingCollection} for an example, where its corresponding quiver is given in Figure \ref{GeneratingQuiver}.
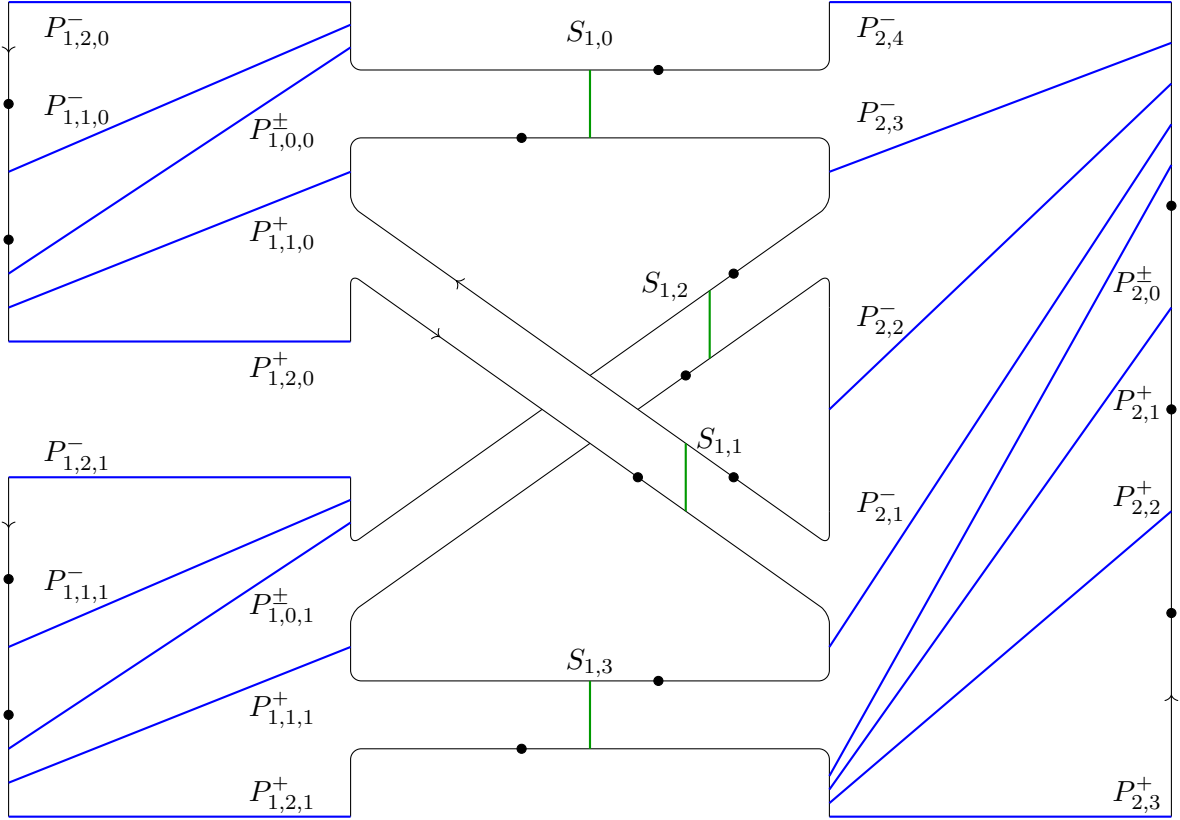
\begin{figure}[h!]
	\begin{tikzpicture}[blob/.style={circle, draw=black, fill=black, inner sep=0, minimum size=\blobsize}, arrow/.style={->, shorten >=6pt, shorten <=6pt}, scale=0.9]
	\def\blobsize{1.2mm}
	\draw[->-=0.15] (5,-12) -- (5,0);
	\draw[blue, thick] (5,0) -- (0,0);
	\draw[blue, thick] (5,-12) -- (0,-12);
	
	\draw[blue, thick] (-7,0) -- (-12,0);
	\draw[->-=0.15] (-12,0) to (-12,-5);
	\draw[blue, thick] (-12,-5) -- (-7,-5);
	
	\draw [rounded corners] (-7,0) -- (-7,-1) -- (0,-1) -- (0,0);
	\draw [rounded corners] (-7,-2.5) -- (-7,-2) -- (0,-2) -- (0,-2.5);
	\draw[->-=0.75][rounded corners] (0,-7.5) --(0,-8) --(-7,-3) -- (-7,-2.5);

	\draw [rounded corners] (-7,-12) -- (-7,-11) -- (0,-11) -- (0,-12);
	\draw [rounded corners] (-7,-9.5) -- (-7,-10) -- (0,-10) -- (0,-9.5);
	\draw[->-=0.25][rounded corners] (-7,-5) to (-7,-4) to (0,-9) to (0,-9.5);
	\draw [rounded corners] (-7,-7) -- (-7,-8) -- (-4.2,-6);
	\draw [rounded corners] (-7,-9.5) -- (-7,-9) -- (-3.5,-6.5);
	\draw [rounded corners] (-3.5,-5.5) -- (0,-3) -- (0,-2.5);
	\draw [rounded corners] (-2.8,-6) -- (0,-4) -- (0,-4.5);
	\draw (0,-4.5) -- (0,-7.5);

	\draw[blue, thick] (-7,-7) -- (-12,-7);
	\draw[->-=0.15] (-12,-7) to (-12,-12);
	\draw[blue, thick] (-12,-12) -- (-7,-12);

	\draw (5, -3) node[blob]{};
	\draw (5, -6) node[blob]{};
	\draw (5, -9) node[blob]{};
	
	\draw (-12, -1.5) node[blob]{};
	\draw (-12, -3.5) node[blob]{};
	
	\draw (-12, -8.5) node[blob]{};
	\draw (-12, -10.5) node[blob]{};
	
	\draw (-2.5, -1) node[blob]{};
	\draw (-4.5, -2) node[blob]{};
	
	\draw (-2.5, -10) node[blob]{};
	\draw (-4.5, -11) node[blob]{};
	
	\draw (-1.4, -4) node[blob]{};
	\draw (-2.1, -5.5) node[blob]{};
	
	\draw (-1.4, -7) node[blob]{};
	\draw (-2.8, -7) node[blob]{};

	\draw[black!40!green, thick]  (-3.5, -1) -- (-3.5,-2); 
	\draw[black!40!green, thick]  (-3.5, -10) -- (-3.5,-11); 
	\draw[black!40!green, thick]  (-1.75, -4.25) -- (-1.75,-5.25); 	
	\draw[black!40!green, thick]  (-2.1, -6.5) -- (-2.1,-7.5);
	
	\draw[blue, thick] (-12,-4.5) -- (-7, -2.5);
	\draw[blue, thick] (-12,-4) -- (-7, -0.667);
	\draw[blue, thick] (-12,-2.5) -- (-7, -.333);
	
	\draw[blue, thick] (-12,-11.5) -- (-7, -9.5);
	\draw[blue, thick] (-12,-11) -- (-7, -7.667);
	\draw[blue, thick] (-12,-9.5) -- (-7, -7.333);
	
	\draw[blue, thick] (0,-11.8) -- (5, -7.5);
	\draw[blue, thick] (0,-11.6) -- (5, -4.5);
	\draw[blue, thick] (0,-11.4) -- (5, -2.4);
	\draw[blue, thick] (0,-9.5) -- (5, -1.8);
	\draw[blue, thick] (0,-6) -- (5, -1.2);
	\draw[blue, thick] (0,-2.5) -- (5, -0.6);
	
	\draw (-8,-2.5) node["$P_{1,0,0}^{\pm}$"]{};
	\draw (-8,-4) node["$P_{1,1,0}^{+}$"]{};
	\draw (-8,-6) node["$P_{1,2,0}^{+}$"]{};
	\draw (-11,-2.15) node["$P_{1,1,0}^{-}$"]{};
	\draw (-11,-1) node["$P_{1,2,0}^{-}$"]{};
	
	\draw (-8,-9.5) node["$P_{1,0,1}^{\pm}$"]{};
	\draw (-8,-11) node["$P_{1,1,1}^{+}$"]{};
	\draw (-8,-12.25) node["$P_{1,2,1}^{+}$"]{};
	\draw (-11,-9.15) node["$P_{1,1,1}^{-}$"]{};
	\draw (-11,-7.25) node["$P_{1,2,1}^{-}$"]{};
	
	\draw (-3.5,-1) node["$S_{1,0}$"]{};
	\draw (-2.4,-4.7) node["$S_{1,2}$"]{};
	\draw (-1.6,-7) node["$S_{1,1}$"]{};
	\draw (-3.5,-10.25) node["$S_{1,3}$"]{};
	
	\draw (0.75,-8) node["$P_{2,1}^{-}$"]{};
	\draw (0.75,-5.25) node["$P_{2,2}^{-}$"]{};
	\draw (0.75,-2.25) node["$P_{2,3}^{-}$"]{};
	\draw (0.75,-1) node["$P_{2,4}^{-}$"]{};
	\draw (4.5,-4.75) node["$P_{2,0}^{\pm}$"]{};
	\draw (4.5,-6.5) node["$P_{2,1}^{+}$"]{};
	\draw (4.5,-7.9) node["$P_{2,2}^{+}$"]{};
	\draw (4.5,-12.25) node["$P_{2,3}^{+}$"]{};
	\end{tikzpicture}
	\caption{Generating collections of Lagrangians for linear gluing of $A(2,2;2)$ to $A(4,3;1)$ via $\sigma_1:(0,1,2,3)\mapsto (0,2,1,3)$. Top and bottom of each annulus is identified. }
	\label{GeneratingCollection}
\end{figure}
\begin{figure}[h!]
	\begin{tikzcd}
		& P_{1,0,0}^- \arrow[r,"x_{1,1,0}"] \ar[equal]{d}& P_{1,1,0}^- \arrow[r,"x_{1,2,0}"] &P_{1,2,0}^-\ar[equal]{d}& P_{1,0,1}^- \arrow[r,"x_{1,1,1}"] \ar[equal]{d}& P_{1,1,1}^- \arrow[r,"x_{1,2,1}"] &P_{1,2,1}^-\ar[equal]{d}\\
		& P_{1,0,0}^+ \arrow[r,"y_{1,1,0}"]& P_{1,1,0}^+ \arrow[r,"y_{1,2,0}"] &P_{1,2,0}^+& P_{1,0,1}^+ \arrow[r,"y_{1,1,1}"]& P_{1,1,1}^+ \arrow[r,"y_{1,2,1}"] &P_{1,2,1}^+\\
		& S_{1,0} \arrow[u,"a_{1,0}"] \arrow[d,"b_{1,0}"]& S_{1,1} \arrow[u,"a_{1,1}"] \arrow{drr}[anchor=center,yshift=2.5ex, xshift=-4ex]{b_{1,1}}& & S_{1,2} \arrow[u,"a_{1,2}"] \arrow{dll}[anchor=center,yshift=2.5ex, xshift=4ex]{b_{1,2}}& S_{1,3} \arrow[u,"a_{1,3}"] \arrow [d,"b_{1,3}"]\\
		P_{2,4}^-\ar[equal]{d}& P_{2,3}^- \arrow[l,"x_{2,4}"]& P_{2,2}^-\arrow[l,"x_{2,3}"]& & P_{2,1}^- \arrow[ll,"x_{2,2}"] & P_{2,0}^- \arrow[l,"x_{2,1}"] \ar[equal]{d}\\
		P_{2,3}^+ & P_{2,2}^+ \arrow[l,"y_{2,3}"] & & P_{2,1}^+ \arrow[ll,"y_{2,2}"] &&  P_{2,0}^+ \arrow[ll,"y_{2,1}"]
	\end{tikzcd}
	\caption{Quiver describing the endomorphism algebra of the generating collection of Figure \ref{GeneratingCollection}. Relations given by $xb=0$ and $ya=0$.}
		\label{GeneratingQuiver}
\end{figure}
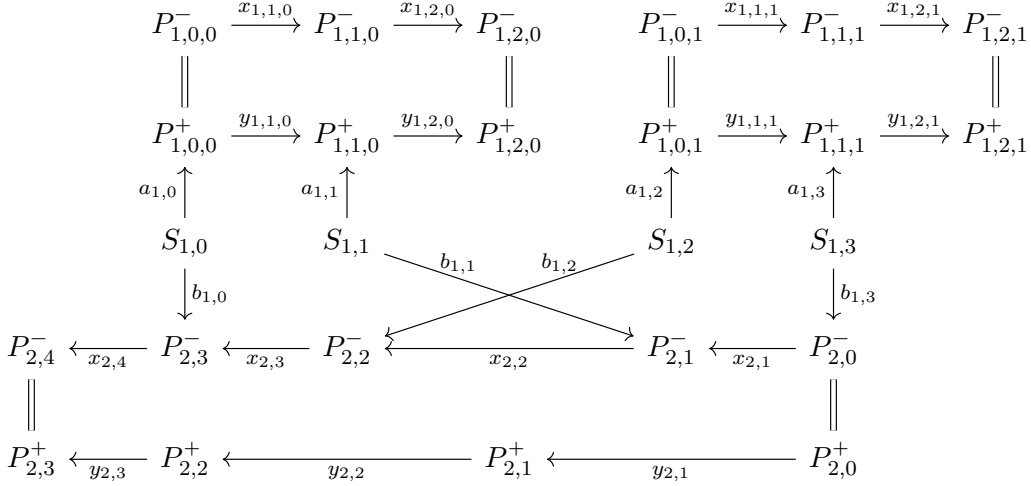
\section{Localisation}\label{LocalisationSection}
As mentioned in the introduction, there are natural localisation functors on the A-- and B--sides. The strategy to establishing Theorem \ref{HMSTheorem} is to show that the quasi-equivalence in Theorem \ref{WrappedTheorem} intertwines localisation on both sides. In this section, we describe the localisation functors on the A-- and B--sides before establishing Theorem \ref{WrappedTheorem}. As in the construction of Auslander orders, this is only a mild generalisation of the orbifold case studied in \cite{LekiliPolishchukAuslander}, and we include only the relevant alterations in the argument.
\subsection{Localisation on the B--side}\label{BSideLocalisation}
As in the non-stacky case, we consider the functor
\begin{align}\label{LocFunctor}
\mathcal{H}om_\mathcal{A}(\mathcal{F},-):\mathcal{A}_\mathcal{C}-\text{mod}\rightarrow \Coh\mathcal{C},
\end{align}
and construct a subcategory $\mathcal{T}$ on which this functor vanishes. We again work locally, and so the analysis follows from the non-stacky case by working equivariantly, as is demonstrated in the orbifold case \cite[Section 3.2]{LekiliPolishchukAuslander}. Note that this functor is exact since $\mathcal{F}$ is a summand of $\mathcal{A}_{\mathcal{C}}$, so is locally projective.\\

In order to construct $\mathcal{T}$, we define the modules
\begin{align*}
\widetilde{\mathcal{S}}_i(j,k)^{\pm}=\begin{pmatrix}
\pi_{i*}(\OO_{\widetilde{\mathcal{C}}_i}(jq_{i,\pm})\otimes \mathcal{N}_i^{\otimes k})|_{q}\\
\pi_{i*}(\OO_{\widetilde{\mathcal{C}}_i}(jq_{i,\pm})\otimes \mathcal{N}_i^{\otimes k})|_{q}
\end{pmatrix},
\end{align*}
where $q=\pi_i(q_{i,\pm})$. These modules fit into the short exact sequences
\begin{align}\label{ExcObjectLocProjRes}
\begin{split}
0\longrightarrow\mathcal{P}_i(j-1,m,k)\longrightarrow\mathcal{P}_i(j,m,k)\longrightarrow \widetilde{\mathcal{S}}_i(j,k)^-\longrightarrow 0\\
0\longrightarrow\mathcal{P}_i(j,m-1,k)\longrightarrow\mathcal{P}_i(j,m,k)\longrightarrow \widetilde{\mathcal{S}}_i(m,k)^+\longrightarrow 0,
\end{split}
\end{align}
and have support at $\pi_i(q_{\pm})$. When the point $q_{i,\pm}$ is not a node (i.e. the $q_{1,-}$ and $q_{n,+}$ in the chain case) we set $\mathcal{E}_i^{\pm}(j,k)=\widetilde{\mathcal{S}}_i(j,k)^{\pm}$. If $q_{i,\pm}$ is a node, then we have natural inclusions
\begin{align*}
\mathcal{S}_{q_{i}}\{\chi_{+}\}\hookrightarrow \widetilde{\mathcal{S}}_i(m,k)^+\\
\mathcal{S}_{q_{i-1}}\{\chi_{-}\}\hookrightarrow \widetilde{\mathcal{S}}_i(j,k)^-,
\end{align*}
where $\chi_{+}$ (resp. $\chi_{-}$) is the character through which $H_i$ (resp. $H_{i-1}$) acts on the fibre of the sheaf $\OO_{{\widetilde{\mathcal{C}}}_i}(mq_{i,+})\otimes\mathcal{N}_i^{\otimes k}$ (resp. $\OO_{{\widetilde{\mathcal{C}}}_i}(jq_{i,-})\otimes\mathcal{N}_i^{\otimes k}$) at $q_{i,+}$ (resp. $q_{i,-}$). We then define $\mathcal{E}_i(j,k)^{\pm}$ to fit into the short exact sequences
\begin{align*}
0\longrightarrow\mathcal{S}_{q_i}\{\chi_{+}\}\longrightarrow \widetilde{\mathcal{S}}_i(m,k)^{+}\longrightarrow \mathcal{E}_i(m,k)^+\longrightarrow 0,\\
0\longrightarrow\mathcal{S}_{q_{i-1}}\{\chi_{-}\}\longrightarrow \widetilde{\mathcal{S}}_i(j,k)^{-}\longrightarrow \mathcal{E}_i(j,k)^-\longrightarrow 0.
\end{align*}
As in the orbifold case, we find that the objects $\mathcal{E}_i(j,k)^{\pm}$ are exceptional unless $\pi_i(q_{i,\pm})$ is a smooth point with no stacky structure. At nodes, this follows from the presentation as a quotient of $xy=0$ by $H_i$. In this presentation, the relevant Ext-groups are the $H_i$-invariant classes of the Ext-groups computed on $xy=0$, and it is shown in \cite[Lemma 3.2.1]{LekiliPolishchukAuslander} that these groups vanish. In the case where the point is a smooth point with non-trivial stacky structure, we find that the objects $\mathcal{E}_i(j,k)^{\pm}$ are exceptional from the locally projective resolution \eqref{ExcObjectLocProjRes}. We define $\mathcal{T}$ to be the subcategory formed by direct sums of all the objects $\mathcal{E}_i(j,k)^{\pm}$ supported at the nodes. \\

With this we have that $\mathcal{T}\subseteq D^b(\mathcal{A}_\mathcal{C}-\mathrm{mod})$ is a Serre subcategory, and identifies
\begin{align*}
\Coh\mathcal{C}&\simeq \mathcal{A}_{\mathcal{C}}-\text{mod}/\mathcal{T}\\
D^b\Coh(\mathcal{C})&\simeq D^b(\mathcal{A}_{\mathcal{C}}-\text{mod})/\langle\mathcal{T}\rangle.
\end{align*}
To see this, note that the derived equivalence follows from the equivalence of abelian categories by \cite{Miyachi}. The equivalence of abelian categories is given for non-stacky curves in \cite[Theorem 4.8]{Burban2009TiltingON}, and the present situation follows from this. As explained in the orbifold case, \cite[Proposition 3.2.3]{LekiliPolishchukAuslander}, one must check that certain adjunctions are equivalences, and this boils down to checking  the statement locally at nodes. One can then use the presentation at a node as the quotient of $xy=0$ by $H_i$, and the argument follows from the non-stacky case.

\subsection{Localisation on the A--side}\label{ALocalistaion}
Part of the utility of the construction in \cite{HKK} is that it not only provides a categorical resolution of the compact Fukaya category of a surface, but also gives an explicit description of a functor
\begin{align*}
\mathcal{W}(\Sigma;\Lambda)\rightarrow \mathcal{W}(\Sigma;\Lambda'),
\end{align*}
where $\Lambda'$ is obtained from $\Lambda$ by removing stops. This map is given by taking the quotient of the partially wrapped Fukaya category by the category generated by Lagrangians which are supported near the stops being removed. In particular, by removing all of the stops in the case of circular gluing, one recovers a map to the wrapped Fukaya category of the surface. In the case of linear gluing, the situation is analogous, however, the stops on the distinguished boundary components are not removed.
It is in this context that the quasi-isomorphism \eqref{LESQuasiIsom} and the functor \eqref{CosheafInclusion} show their utility by giving the Lagrangian supported near a stop to be removed in terms of the generating Lagrangians on a disc. \\

For circular gluing, we will define the object $E_{i,j}^{+}$ (resp. $E_{i+1,j}^-$) to be Lagrangian supported near the stop on the bottom (resp. top) of the attaching strip beginning at the neighbourhood of the $j^{\text{th}}$ stop on a boundary component between the $i^{\text{th}}$ and $(i+1)^{\text{st}}$ columns. By again writing $j=k_+r_i+c_+$ and $\sigma_i(j)=k_-\ell_{i+1}+c_-$ for $k_+\in\{0,\dots, d_i-1\}$, $c_+\in\{0,\dots,r_{i}-1\}$, $k_-\in\{0,\dots d_{i+1}-1\}$, and $c_-\in\{0,\dots,\ell_{i+1}-1\}$, we have by \eqref{LESQuasiIsom} and \eqref{CosheafInclusion} 
\begin{align*}
E_{i,j}^+&\simeq\{S_{i,j}[3]\rightarrow P_{i,c_+,k_+}^+[2]\rightarrow P_{i,c_++1,k_+}^+[1]\}\\
E_{i+1,j}^-&\simeq\{S_{i,j}[3]\rightarrow P_{i+1,\ell_{i+1}-1-c_-,k_-}^-[2]\rightarrow P_{i+1,\ell_{i+1}-c_-,k_-}^-[1]\}.
\end{align*}
In the case of linear gluing we have the same iterated cones in for the hexagonal regions, as well as the cones 
\begin{align*}
E_{1,j}^-&\simeq\{P_{i,j}^-[2]\rightarrow P_{i,j+1}^-[1]\}\\
E_{n,j}^+&\simeq \{P_{i,j}^+[2]\rightarrow P_{i,j+1}^+[1]\}.
\end{align*}

\begin{proof}[Proof of Theorem \ref{WrappedTheorem}]
	In order to prove the statement, it suffices to match the generators of the categories in question. We begin with the case of a cycle of curves, or a chain where both $r_{1,-},r_{n,+}>0$. On the B--side, we fix an exceptional collection such that $j_i=0$ and $m_i=-1$. We again label the characters in $\widehat{H}_i$ such that $\chi_{k_+r_{i,+}+m}$ is the character of $\OO_{\widetilde{{\mathcal{C}}}_i}(mq_{i,+})\otimes \mathcal{N}_i^{\otimes k_+}$. On the A--side we construct the candidate mirror as follows: For each irreducible component $\mathcal{C}_i$ of $\mathcal{C}$, being a $\mu_{d_i}$-gerbe over $\mathbb{P}_{r_{i,-},r_{i,+}}$, we consider a column of annuli $A(r_{i,-},r_{i,+};d_i)$. Let $j, k_-$ solve 
	\begin{align*}
	-j\chi_{r_{i+1,-}}+k_-\chi_{d_{i+1},-}=\chi_{k_+r_{i,+}+m}, 
	\end{align*}
	as in \eqref{WeightMatching-}. We then define the permutation $\sigma_i$ to be given by 
	\begin{align*}
	k_+r_{i,+}+m\mapsto k_-r_{i+1,-}+(-j)\bmod r_{i+1,-}.
	\end{align*}
	Let $\Sigma$ be the surface constructed in this way, and let $\mathcal{W}(\Sigma; \Lambda)$ be its partially wrapped Fukaya category, as described in Section \ref{ComputationPwrapped}. The identification of the generation objects on both sides is given by:
	\begin{align*}
	P_{i,j,k}^-&\longleftrightarrow\mathcal{P}_i(j,-1,k)\\
	P_{i,m,k}^+&\longleftrightarrow\mathcal{P}_i(0,m-1,k)\\
	S_{i,j}&\longleftrightarrow\mathcal{S}_i\{\chi_j\}[-1].
	\end{align*}
	From this, we can see that the endomorphism algebras of the two exceptional collections which generate their respective categories are equivalent, which establishes the claim in this case.\\
	
	To complete the proof in the case of a cycle of curves, where either or both of $r_{1,-},r_{n,+}=0$, we must utilise \cite[Proposition 3.2.2]{LekiliPolishchukAuslander}, which, suitably reworded to our context, states that under the above equivalence, we have a correspondence 
	\begin{align}\label{IdentifyingModules}
	\{\mathcal{A}_\mathcal{C}-\text{modules} \quad \mathcal{E}^-_i(j,k)\}&\longleftrightarrow\{E^-_{i,r_{i,-}k+j}[-1]\}\\
		\{\mathcal{A}_\mathcal{C}-\text{modules} \quad \mathcal{E}^+_i(m,k)\}&\longleftrightarrow\{E^+_{i,r_{i,+}k+m-1}[-1]\}.
	\end{align}
	The proof of the alteration of the statement to our situation follows directly from the proof of the original statement. Namely, we let $\Exc$ be the direct sum of the objects in the exceptional collection in $D^b(\mathcal{A}_{\mathcal{C}}-\mathrm{mod})$ described in Section \ref{OrderSection} and $A$ its endomorphism algebra. One can describe the right $A$-modules of the form $\mathrm{R}\!\Hom(\Exc,-)$ corresponding to the objects $\mathcal{E}^-_i(j,k)$ and $\mathcal{E}^+_i(m,k)$ in the equivalence 
	\begin{align*}
	\mathrm{R}\!\Hom(\Exc,-):D^b(\mathcal{A}_{\mathcal{C}}-\text{mod})\xrightarrow{\sim}D^b(\text{mod}-A),
	\end{align*}
	 and see that they match with the objects in $D^b(\text{mod}-A)$ calculated using the presentations of $E_{i,j}^-$ (resp. $E_{i,j-1}^+$) given in Section \ref{ALocalistaion}.\\ 

	Now, consider the case of $r_{n,+}>r_{1,-}=0$, and define the stack $\overline{\mathcal{C}}$ to be the same curve as $\mathcal{C}$, but where $\mathcal{C}_1=\P_{1,r_{1,+}}$. Namely, we have $\mathcal{C}=\overline{\mathcal{C}}\setminus\{q_{1,-}\}$.
	Since $\mathcal{A}_{\overline{\mathcal{C}}}$ is isomorphic near $q_{1,-}$ to the matrix algebra over $\OO$, we have that the restriction functor
	\begin{align*}
	\mathcal{A}_{\overline{\mathcal{C}}}-\text{mod}\rightarrow \mathcal{A}_{\mathcal{C}}-\text{mod}
	\end{align*}
	identifies $\mathcal{A}_{\mathcal{C}}-\text{mod}$ with the quotient of $\mathcal{A}_{\overline{\mathcal{C}}}-\text{mod}$ by the Serre subcategory generated by $\bigoplus_{k=0}^{d_1-1}\mathcal{E}_1(0,k)^-$ (i.e. $\bigoplus_{k=0}^{d_1-1}\begin{pmatrix}
	\OO_{q_{1,-}}\\
	\OO_{q_{1,-}}
	\end{pmatrix}\otimes\mathcal{N}_1^{\otimes k}$). By the main result of \cite{Miyachi}, this yields a derived equivalence 
	\begin{align*}
	D^b(\mathcal{A}_{\overline{\mathcal{C}}}-\text{mod})/\langle\bigoplus_{k=0}^{d_1-1}\mathcal{E}_1(0,k)^-\rangle\simeq D^b(\mathcal{A}_{\mathcal{C}}-\text{mod}).
	\end{align*}
	
	From the first part of the proof, there is a graded surface $(\Sigma,\overline{\Lambda})$ such that
	\begin{align*}
	\mathcal{W}(\Sigma;\overline{\Lambda})\simeq D^b(\mathcal{A}_{\overline{\mathcal{C}}}-\text{mod}).
	\end{align*}
	Now, since $\mathcal{E}_1^-(0,k)$ is identified with $E^-_{1,k}[-1]$ in \eqref{IdentifyingModules}, removing the stops on the left distinguished boundary corresponds to localising $D^b(\mathcal{A}_{\overline{\mathcal{C}}}-\text{mod})$ by the category generated by $\bigoplus_{k=0}^{d_1-1}\mathcal{E}_1(0,k)^-$, which yields the result. The cases of $r_{1,-}>r_{n,+}=0$ or $r_{0,-}=r_{n-1,+}=0$ are analogous. 
\end{proof}
\begin{rmk}
Choosing different values for $m_i$ and $j_i$ in the above theorem corresponds to changing the identification of the cylinders on the A--side by a cyclic reordering. This yields homeomorphic mirrors, since cyclically changing the identification of an individual annulus, and/ or reordering the annuli in a column, does not change the number of cycles, or their length, in the cycle decomposition determining the topology of the surface. 
\end{rmk}
\section{Characterisation of perfect derived categories}\label{PerfectObjectsSection}
In order to establish the statement about perfect objects in Theorem \ref{HMSTheorem}, one must show that the compact Fukaya category and derived category of perfect complexes, considered as full subcategories of $\mathcal{W}(\Sigma;\Lambda)$ and $D^b(\mathcal{A}_{\mathcal{C}}-\text{mod})$, respectively, are identified with each other under the quasi-equivalence of Theorem \ref{WrappedTheorem}. The aim of this section is to characterise perfect complexes on the A-- and B--sides of the correspondence before establishing Theorem \ref{HMSTheorem}.
\subsection{The derived category of perfect complexes}
As in the localisation argument, our strategy closely follows that of \cite[Theorem 2]{Burban2009TiltingON} for the non-stacky case. Let $\mathcal{C}$ be a cycle or chain of curves with $r_{1,-},r_{n,+}>0$, $\mathcal{F}$ as in Section \ref{OrderSection}, and consider the functor
\begin{align*}
\perf \mathcal{C}&\rightarrow D^b(\mathcal{A}_{\mathcal{C}}-\text{mod})\\
G&\mapsto \mathcal{F}_\mathcal{C}\otimes_{\OO_\mathcal{C}} G.
\end{align*}
In the non-stacky case, it is shown that this functor is full and faithful in \cite[Theorem 2 (5)]{Burban2009TiltingON}, and this result is generalised to the orbifold case in \cite[Proposition 4.1.3]{LekiliPolishchukAuslander}. As in these cases, one can again identify the essential image of $\perf \mathcal{C}$ in $D^b(\mathcal{A}_{\mathcal{C}}-\text{mod})$ as the subcategory which is both left and right orthogonal to the category $\mathcal{T}$ defined in Section \ref{BSideLocalisation}. The proof of this follows verbatim from the proof of \cite[Proposition 4.1.3 (i)]{LekiliPolishchukAuslander} after replacing $\mu_r$ by $H$, an extension of $\mu_r$ by $\mu_d$. In the case of a cycle of curves with $r_{n,+}>r_{1,-}=0$, the category of compactly supported perfect complexes on $\mathcal{C}$ is identified with the category which is both left and right orthogonal to $\overline{\mathcal{T}}$, where this category is formed by the objects of $\mathcal{T}$, together with $\mathcal{E}_1(k)^-$ for $0\leq k\leq d_1-1$. To prove this, observe that $\mathcal{E}_1^-(0,k)\simeq \begin{pmatrix}
\OO_{q_{1,-}}(k\chi_{d_{1,-}})\\
\OO_{q_{1,-}}(k\chi_{d_{1,-}})
\end{pmatrix}$ near $q_{1,-}$, and so a module in $\perf\overline{\mathcal{C}}$ is left or right orthogonal to $\bigoplus_{k=0}^{d_1-1}\mathcal{E}_1^-(0,k)$ if and only if its support does not contain $q_{1,-}$. Then, the rest of the proof in this case follows as in \cite[Proposition 4.1.3 (ii)]{LekiliPolishchukAuslander}. The cases when $r_{1,-}>r_{n,+}=0$ and $r_{1,-}=r_{n,+}=0$ are considered similarly.
\subsection{Characterisation of the Fukaya category}
On the A--side of the correspondence, the characterisation of the Fukaya category as a subcategory of the partially wrapped Fukaya category remains unchanged from \cite[Section 4.2]{LekiliPolishchukAuslander}. We briefly recall the argument here, and refer to loc. cit. for the proof.\\

Let $\mathcal{T}_i$ be the collection of Lagrangians supported near the stops on the $i^{\text{th}}$ boundary component. It is shown (\cite[Proposition 4.2.1]{LekiliPolishchukAuslander}) that $\mathcal{T}_i^\perp=^{\perp}\mathcal{T}_i$ corresponds to those Lagrangians in $\mathcal{W}(\Sigma;\Lambda)$ \emph{not} ending on the $i^{\text{th}}$ boundary component. One direction of this argument is clear: if there is a Lagrangian which is either compact, or does not end on the $i^{\text{th}}$ boundary component, then the intersection with the geometric representatives of Lagrangians supported near the stops can be taken to be empty. In the other direction, one shows that if a Lagrangian \emph{does} end on a boundary component, then there is necessarily a non-trivial morphism at the level of cohomology between this Lagrangian and a Lagrangian in  $\mathcal{T}_i$. In the case where just one endpoint of the Lagrangian lies on the $i^{\text{th}}$ boundary component there is a chain level morphism between the Lagrangian and a Lagrangian in $\mathcal{T}_i$ which is of rank one, so the differential vanishes. In the case where both endpoints lie on the $i^{\text{th}}$ boundary component, the chain level morphism complex between the Lagrangian and a Lagrangian in $\mathcal{T}_i$ is either rank one or two. In the rank one case we again have that the differential must vanish, and in the rank two case one shows that the differential vanishes by a covering argument. This shows that any Lagrangian with at least one endpoint on the $i^{\text{th}}$ boundary component cannot belong to $\mathcal{T}_i^\perp$. Checking that a Lagrangian with at least one endpoint on the $i^{\text{th}}$ boundary component cannot belong to  $^{\perp}\mathcal{T}_i$ is done in the same way.\\

By summing over the boundary components of $\Sigma$ we define $\mathcal{T}=\bigoplus_{i}\mathcal{T}_i$. Then, \cite[Corollary 4.2.2]{LekiliPolishchukAuslander} shows:
\begin{itemize}
	\item In the case of a cycle of curves, the subcategory $\mathcal{F}(\Sigma)\subseteq \mathcal{W}(\Sigma;\Lambda)$ coincides with $\mathcal{T}^\perp=^\perp\mathcal{T}$, where $\mathcal{T}$ is the category generated by the objects $E^{\pm}_{i,j}$. 
	\item In the case of a chain of curves with $r_{1,-},r_{n,+}>0$, the subcategory\\ $\mathcal{F}(\Sigma;(r_{1,-})^{d_1},(0)^{b-d_1-d_n},(r_{n,+})^{d_n})\subseteq \mathcal{W}(\Sigma;\Lambda)$ coincides with $\mathcal{T}^\perp=^\perp\mathcal{T}$, where $\mathcal{T}$ is the category generated by $E_{i,j}^+$ for $i\in\{1,\dots, n-1\}$ and $E_{i,j}^-$ for $i\in\{2,\dots, n\}$.
\end{itemize}
\begin{proof}[Proof of Theorem \ref{HMSTheorem}]
In the case of a cycle of curves, or a chain where $r_{1,-},r_{n,+}>0$, the theorem follows from the observation that the generating objects of the category $\mathcal{T}$ on both sides of the correspondence are identified under the equivalence given in Theorem \ref{WrappedTheorem}. In the case where $r_{n,+}>r_{1,-}=0$ we again consider $\overline{\mathcal{C}}$ such that $\mathcal{C}=\overline{\mathcal{C}}\setminus\{q_{1,-}\}$. Then, the statement follows from using the characterisation of $\perf\mathcal{C}\subseteq D^b(\mathcal{A}_{\overline{\mathcal{C}}}-\mathrm{mod})\simeq\mathcal{W}(\Sigma;\overline{\Lambda})$ as the category which is both left and right perpendicular to $\overline{\mathcal{T}}$. 
\end{proof}
\section{Milnor fibres of invertible curve singularities}\label{Applications section}
Whilst more could be made of examples of mirror symmetry from Theorems \ref{WrappedTheorem} and \ref{HMSTheorem}, our primary motivation is in the study of (quotients of) Milnor fibres of invertible curve singularities. Indeed, it is the ability to handle equivariance on the A--side by dealing with the situation abstractly which was the impetus for generalising the strategy of Lekili and Polishchuk. \\
In this section, we establish Theorem \ref{InvertiblePolyTheorem} by firstly applying Theorems \ref{WrappedTheorem} and \ref{HMSTheorem} to the curves appearing as the B--model of invertible polynomials in two variables. We then show that the surfaces constructed are graded symplectomorphic to $\widecheck{V}/\widecheck{\Gamma}$. \\

To begin with, recall the definition of invertible polynomials and the maximal symmetry group, as defined in the introduction. For simplicity, we will restrict ourselves to the case of two variable invertible polynomials, although much of the following is true in generality (\cite[Section 2]{LekiliUeda}, \cite[Section 1]{EbelingTakahashi}, \cite[Section 3]{Krawitz}).\\

By construction, the map in \eqref{GradingInclusionEq} fits in to the short exact sequence
\begin{align}
1\rightarrow\C^*\xrightarrow{\phi}\Gamma_\w\rightarrow \text{ker}\ \chi_\w /\langle j_\w\rangle\rightarrow 1,
\end{align}
where $j_\w=(e^{2\pi \sqrt{-1}\frac{d_1}{h}},e^{2\pi \sqrt{-1}\frac{d_2}{h}})$ generates the cyclic group $\text{im}(\phi)\cap\ker\chi_{\w}$, and is called the grading element. Recall that a subgroup of $\Gamma\subseteq \Gamma_\w$ of finite index is called admissible if it contains $\phi(\C^*)$. For each $\Gamma$ we denote $\chi:=\chi_{\w}|_{\Gamma}$, and define $\overline{\Gamma}=\ker\chi$. Note that, by construction, $\langle j_{\w}\rangle\subseteq \overline{\Gamma}$, and $[\ker \chi_\w:\overline{\Gamma}]<\infty$. Moreover, such subgroups of finite index containing the group generated by the grading element are in bijection with finite index subgroups $\Gamma\subseteq \Gamma_{\w}$ containing $\text{im}(\phi)$. \\

Given an admissible subgroup $\Gamma\subseteq \Gamma_{\w}$ of index $\ell$, one defines the dual group as in \eqref{DualGpDefn}. This acts naturally on $\A^2$ through its inclusion in $\ker\chi_{\widecheck{\w}}$, and in each case, $\widecheck{\Gamma}=\mu_{\ell}$ acts on $\A^2$ by
\begin{align}\label{GammaCheckAction}
\xi\cdot (x,y)=(\xi x,\xi^{-1} y).
\end{align}
This can be checked directly, or deduced from the fact that $\widecheck{\Gamma}$ is a diagonal matrix in $\mathrm{SL}_2(\C)$, and so its two entries must be inverses of each other. Clearly, the only fixed point of this action is the origin, which is not a point in the Milnor fibre of any invertible polynomial in two variables, and so the quotient of the Milnor fibre by $\widecheck{\Gamma}$ is again a manifold. 
\subsection{Loop polynomials}\label{loopHMS}
For a loop polynomial $\w=x^py+y^qx$, where we take $p\geq q$, we consider $\W=x^py+y^qx+xyz$, and the corresponding stack
\begin{align*}
Z_{\w_{\mathrm{loop}},\Gamma}:=[\big(\W^{-1}(0)\setminus\{\boldsymbol{0}\}\big)/\Gamma],
\end{align*}
where we take the action of $\Gamma$ to be given by its inclusion to $\Gamma_\w$. Let $\ell=[\Gamma_\w:\Gamma]$ and identify $\Gamma\simeq\C^*\times\mu_{\frac{d}{\ell}}$, where $d=\gcd(p-1,q-1)$. The stack $Z_{\w_{\mathrm{loop}},\Gamma}$ has a natural interpretation as a codimension one closed substack in the toric DM orbifold $[\big(\A^3\setminus\{\boldsymbol{0}\}\big)/\Gamma]$. The unique stacky fan describing this DM orbifold is readily checked to be given by the data of 
\begin{align*}
\beta:\Z^3\xrightarrow{\begin{pmatrix}
	\frac{p-1}{\ell} & \frac{1-q}{\ell} & 0 \\
	0 & q-1 & -1
	\end{pmatrix}} \Z^2=:N,
\end{align*}
and each column corresponds to a ray of the fan $\Sigma$. The maximal cones of the fan are given by the span of any two rays. In general, this is a quotient of weighted projective space by $\mu_{\frac{d}{\ell}}$.
\begin{rmk} It is worth noting that we have made a choice in the identification $\Gamma\simeq\C^*\times\mu_{\frac{d}{\ell}}$, and thus how $\Gamma$ acts on $\A^3\setminus\{\boldsymbol{0}\}$; however, the above fan is independent of this choice. Choosing a different identification of $\Gamma$ corresponds to choosing different change-of-basis matrices in the Smith normal form decomposition of $[BQ]^\vee$ used to calculate its cokernel.
\end{rmk}
With this description, one can see that $\mathcal{C}_1=\{y=0\}\subseteq Z_{\w_{\mathrm{loop},\Gamma}}$ is the closed substack of $[\big(\A^3\setminus\{\boldsymbol{0}\}\big)/\Gamma]$ corresponding to the ray $\rho_2=\frac{1-q}{\ell}e_1+(q-1)e_2$, and similarly that $\mathcal{C}_3=\{x=0\}$ is the closed substack corresponding to the ray $\rho_1=\frac{p-1}{\ell}e_1$. The quotient fan $\boldsymbol{\Sigma}/\boldsymbol{\rho_2}$ is given by the complete fan in $\mathbb{Q}$, and
\begin{align*}
\beta(\rho_2):\Z^2\xrightarrow{\begin{pmatrix}
	p-1 & -1\\
	\frac{1-p}{\ell} & 0
	\end{pmatrix}} \Z\oplus \Z/(\frac{q-1}{\ell})=:N(\rho_2)
\end{align*}
This is a $\mu_{\frac{q-1}{\ell}}$-gerbe over $\P_{p-1,1}$, and \cite[Theorem 7.24]{FantechiMannNironi}  establishes that there is an isomorphism of toric DM stacks
\begin{align*}
\mathcal{C}_1\simeq \sqrt[\frac{q-1}{\ell}]{\OO(-\frac{p-1}{\ell}q_{1,-})/\P_{p-1,1}}.
\end{align*}
Similarly, we have an isomorphism of toric DM stacks
\begin{align*}
\mathcal{C}_3\simeq\sqrt[\frac{p-1}{\ell}]{\OO(\frac{q-1}{\ell}q_{3,+})/\P_{1,q-1}}.
\end{align*}
The curve $\mathcal{C}_2$ is always an orbifold, and can be identified with $\mathcal{C}_2\simeq\P_{\frac{q-1}{\ell},\frac{p-1}{\ell}}$.
\begin{rmk}
It is worth reiterating that we are not claiming that the gerbe structures of $\mathcal{C}_1$ and $\mathcal{C}_3$ are given as above, only that there is an isomorphism of DM stacks. Due to this, there is some freedom in the identifications, and we have chosen these for later convenience.
\end{rmk}
The majority of the analysis in studying the modules over the Auslander sheaf is at $q_3=|\mathcal{C}_3|\cap|\mathcal{C}_1|$, which corresponds to the point $[0:0:1]\in|\mathcal{X}|$. This node is presented as the quotient of $xy=0$ by the action of $\mu_{\frac{(p-1)(q-1)}{d}}\times\mu_{\frac{d}{\ell}}$ given by
\begin{align*}
(t,\xi)\cdot(x,y)=(t^{\frac{p-1}{d}}\xi^{-n}x, t^{\frac{q-1}{d}}\xi^{m}y),
\end{align*}
where $m,n$ are B\'ezout coefficients solving
\begin{align}\label{loopBezout}
m(p-1)+n(q-1)=d.
\end{align}
Therefore the gerbe structure of the point $q_{3,+}$ is determined by the cohomology class in  $\Z/\gcd(q-1,\frac{p-1}{\ell})\simeq H^2([\A^1/\mu_{q-1}],\mu_{\frac{p-1}{\ell}})$ corresponding to the $\bmod \frac{p-1}{\ell}$ reduction of $\frac{(\ell-1)(q-1)}{\ell}\in\Z$. Similarly, we have that the gerbe at $q_{1,-}\in|\mathcal{C}_1|$ is classified by the cohomology class in $\Z/\gcd(p-1,\frac{q-1}{\ell})\simeq H^2([\A^1/\mu_{p-1}],\mu_{\frac{q-1}{\ell}})$ corresponding to the $\bmod \frac{q-1}{\ell}$ reduction of $\frac{p-1}{\ell}\in\Z$. The corresponding short exact sequences at $q_{3,+}$ and $q_{1,-}$ are 
\begin{align}
&1\rightarrow\mu_{\frac{p-1}{\ell}}\xrightarrow{\varphi_{3,+}} \mu_{\frac{(p-1)(q-1)}{d}}\times\mu_{\frac{d}{\ell}}\xrightarrow{\psi_{3,+}}\mu_{q-1}\rightarrow 1,\ \text{and} \label{C3SES}\\
&1\rightarrow\mu_{\frac{q-1}{\ell}}\xrightarrow{\varphi_{1,-}} \mu_{\frac{(p-1)(q-1)}{d}}\times\mu_{\frac{d}{\ell}}\xrightarrow{\psi_{1,-}}\mu_{p-1}\rightarrow 1, \label{C1SES}
\end{align}
respectively. Here $\lambda_{\pm}$, $\eta$, and $\xi$ are
\begin{align*}
\lambda_{+}=e^{2\pi\sqrt{-1}\frac{\ell}{p-1}},\quad &\lambda_{-}=e^{2\pi\sqrt{-1}\frac{\ell}{q-1}},\quad \\
\eta=e^{2\pi\sqrt{-1}\frac{d}{(p-1)(q-1)}},\quad & \xi=e^{2\pi\sqrt{-1}\frac{\ell}{d}},
\end{align*}
and $\varphi_{3,+}$ is the map $\lambda_{+}\mapsto (\eta^{-n\frac{(q-1)\ell}{d}},\xi^{-1})$, $\psi_{3,+}$ is $(\eta^a,\xi^b)\mapsto \eta^{\frac{p-1}{d}a}\xi^{-nb}$, $\varphi_{1,-}$ is $\lambda_{-}\mapsto (\eta^{m\frac{(p-1)\ell}{d}},\xi^{-1})$, $\psi_{1,-}$ is $(\eta^a,\xi^b)\mapsto \eta^{\frac{q-1}{d}a}\xi^{mb}$, where $m,n$ are again the B\'ezout coefficients of \eqref{loopBezout}.\\

From this description, we have that the group $H_3$ acts on the fibre of $\OO_{\widetilde{\mathcal{C}}_3}(-q_{3,+})$ at $q_{3,+}$ with weight $\chi_{r_{3,+}}=(\frac{p-1}{d},-n)\in\Z/(\frac{(p-1)(q-1)}{d})\oplus \Z/(\frac{d}{\ell})\simeq\widehat{H}_3$ for $m,n$ solving \eqref{loopBezout}, and similarly $\OO_{\widetilde{\mathcal{C}}_1}(-q_{1,-})$ at $q_{1,-}$ is acted on with weight $\chi_{r_{1,-}}=(\frac{q-1}{d},m)$. The character with which $H_3$ acts on the fibre of $\mathcal{N}_{3}$ is (non-uniquely) determined by the condition that $\frac{p-1}{\ell}\chi_{d_{3,+}}=\frac{1-q}{\ell}\chi_{r_{3,+}}$, and maps to a unit in $\Z/(\frac{p-1}{\ell})$ under the dual of $\varphi_{3,+}$. The natural choice for this is $\chi_{d_{3,+}}=-\chi_{r_{1,-}}$, and similarly we choose $\chi_{d_{1,-}}=\chi_{r_{3,+}}$.\\

In $\widehat{H}_3$, we label the characters such that $\chi_{k_+(q-1)+i}=-i\chi_{r_{3,+}}+k_+\chi_{d_{3,+}}$ for $k_+\in\{0,\dots,\frac{p-1}{\ell}-1\}$ and $i\in\{0,\dots, q-2\}$. This is the B--side version of labelling the stops on the right side of the left column of cylinders top-to-bottom. With this ordering, the sheaf on $\widetilde{\mathcal{C}}_1$ whose fibre at $q_{1,-}$ is acted on by $H_3$ with character $\chi_{k_+(q-1)+i}$ is given by 
\begin{align*}
\OO_{\widetilde{\mathcal{C}}_1}(jq_{1,-})\otimes\mathcal{N}_1^{\otimes k_-},
\end{align*}
where $j\in\{0,\dots,p-2\}$ and $k_-\in\{0,\dots,\frac{q-1}{\ell}-1\}$ solves
\begin{align}\label{WeightCompatibility}
-j\chi_{r_{1,-}}+k_-\chi_{d_{1,-}}=-i\chi_{r_{3,+}}+k_+\chi_{d_{3,+}}.
\end{align}
A solution to this is readily checked to be given by 
\begin{align}
\begin{split}\label{WeightsEqn}
k_-&=-i\bmod\frac{q-1}{\ell}\\
j&=k_+-\frac{p-1}{\ell}\big\lfloor\frac{-i\ell}{q-1}\big\rfloor\bmod p-1.
\end{split}
\end{align}
Fixing $m_i=-1$ and $j_i=0$ as in the proof of Theorem \ref{WrappedTheorem}, one computes 
\begin{align*}
&\Ext^1(\mathcal{S}_{q_3}\{-i\chi_{r_{3,+}}+k_+\chi_{d_{3,+}}\},\mathcal{P}_3(0,(i-1)\bmod q-1,k_+))=\C\cdot a(i,k_+),\quad \text{and}\\
&\Ext^1(\mathcal{S}_{q_{3}}\{-i\chi_{r_{3,+}}+k_+\chi_{d_{3,+}}\},\mathcal{P}_1((j-1)\bmod {p-1},-1,k_-))=\C\cdot b(j,k_-)
\end{align*}
for $j,k_-$ as in \eqref{WeightsEqn}. \\

Consider now the nodes $q_1=|\mathcal{C}_1|\cap|\mathcal{C}_2|$ and $q_2=|\mathcal{C}_2|\cap|\mathcal{C}_3|$. The structure of these nodes is far more simple, and at $q_1$ we have the node is presented as the quotient of $xy=0$ by the action of $\mu_{\frac{q-1}{\ell}}$ given by 
\begin{align*}
t\cdot (x,y)=(x,ty),
\end{align*}
and analogously for $q_2$. Therefore, one has $\widehat{H}_1\simeq\Z/(\frac{q-1}{\ell})$ and $\widehat{H}_2\simeq \Z/(\frac{p-1}{\ell})$, and $\chi_{r_{2,-}}$  and $\chi_{r_{2,+}}$ are the identity in $\Z/(\frac{q-1}{\ell})$ and $\Z/(\frac{p-1}{\ell})$, respectively. The character with which $H_1$ acts on the fibre of $\mathcal{N}_1$ at $q_{1,+}$ (resp. on $\mathcal{N}_3$ at $q_{3,-}$) is any unit of $\widehat{H}_1$ (resp. $\widehat{H}_2$), and so we choose $\chi_{d_1,+}$ to be the identity and $\chi_{d_{3,-}}$ to be minus the identity in their respective character groups. With this, the morphisms between objects in the exceptional collection supported at $q_1$ are readily checked to be
\begin{align*}
&\Ext^1(\mathcal{S}_{q_1}\{c\},\mathcal{P}_1(0,-1,c))=\C\cdot a(0,c)\\
&\Ext^1(\mathcal{S}_{q_{1}}\{c\},\mathcal{P}_2((-1-c)\bmod \frac{q-1}{\ell},-1))=\C\cdot b(-c),
\end{align*}
and similarly for the morphisms between objects supported at $q_2$.\\

As the mirror to $\mathcal{C}$, we take the surface given by gluing $A(p-1,1;\frac{q-1}{\ell})$, $A(\frac{q-1}{\ell},\frac{p-1}{\ell};1)$ and $A(1,q-1;\frac{p-1}{\ell})$ via the permutations $\sigma_1=\text{id}\in\mathfrak{S}_{\frac{q-1}{\ell}}$, $\sigma_2=\text{id}\in\mathfrak{S}_{\frac{p-1}{\ell}}$, and $\sigma_3\in\mathfrak{S}_{\frac{(p-1)(q-1)}{\ell}}$ is given by 
\begin{align*}
k_+(q-1)+i\mapsto k_-(p-1)+(-j)\bmod p-1
\end{align*}
for $i,j$ solving $\eqref{WeightsEqn}$. From this, it is clear that one boundary component with winding number $-2\frac{q-1}{\ell}$ arises from $\sigma_1$, and similarly that one boundary component with winding number $-2\frac{p-1}{\ell}$ arises from $\sigma_2$. The number of boundary components, and their winding numbers, arising from $\sigma_3$ is given by the number of cycles, and their respective lengths, of  $\sigma^{-1}_3\tau_{\ell_1}\sigma_3\tau_{r_3}$. This permutation is given by 
\begin{align*}
k_+(q-1)+i\mapsto (q-1)\Big((k_+-1)\bmod\frac{p-1}{\ell}\Big)+\Big(i-1+\frac{q-1}{\ell}\Big\lfloor\frac{(k_+-1)\ell}{p-1}\Big\rfloor\Big)\bmod q-1
\end{align*}
and so there are $\gcd(q-1,\frac{p+q-2}{\ell})=\gcd(p-1,\frac{p+q-2}{\ell})$ cycles, each of length $\frac{(p-1)(q-1)}{\gcd(\ell(q-1),p+q-2)}$. Therefore, $\gcd(q-1,\frac{p+q-2}{\ell})$ boundary components arise from this gluing, and each has winding number $-2\frac{(p-1)(q-1)}{\gcd(\ell(q-1),p+q-2)}$. \\

Putting this all together, we have that the surface constructed, call it $\Sigma_{\w_{\text{loop}},\Gamma}$, has $2+\gcd(q-1,\frac{p+q-2}{\ell})$ components, and Euler characteristic given by
\begin{align*}
-\chi(\Sigma_{\w_{\text{loop}},\Gamma})=\frac{q-1}{\ell}+\frac{p-1}{\ell}+\gcd(q-1,\frac{p+q-2}{\ell})\frac{(p-1)(q-1)}{\ell\gcd(q-1,\frac{p+q-2}{\ell})}=\frac{pq-1}{\ell}. 
\end{align*}
Therefore, the genus is 
\begin{align*}
g(\Sigma_{\w_{\text{loop}},\Gamma})=\frac{1}{2\ell}(pq-1-\gcd(\ell(q-1),p+q-2)).
\end{align*}
Applying Theorem \ref{WrappedTheorem} yields a quasi-equivalence
\begin{align*}
D^b(\mathcal{A}_{\mathcal{C}}-\mathrm{mod})\simeq \mathcal{W}\bigg(\Sigma_{\w_{\text{loop}},\Gamma};2\frac{p-1}{\ell},\Big(2\frac{(p-1)(q-1)}{\gcd(\ell(q-1),p+q-2)}\Big)^{\gcd(q-1,\frac{p+q-2}{\ell})},2\frac{q-1}{\ell}\bigg),
\end{align*}
and then Theorem \ref{HMSTheorem} establishes quasi-equivalences
\begin{align*}
D^b\Coh(Z_{\w_{\mathrm{loop}},\Gamma})&\simeq \mathcal{W}(\Sigma_{\w_{\text{loop}},\Gamma})\\
\perf Z_{\w_{\mathrm{loop}},\Gamma}&\simeq\mathcal{F}(\Sigma_{\w_{\text{loop}},\Gamma}).
\end{align*}
In the case of $\Gamma=\Gamma_{\w}$, we observe that the graded surface constructed on the A--side is graded symplectomorphic to the Milnor fibre of the transpose invertible polynomial. To see this, we note that the above gluing is the same as the gluing permutation of \cite[Section 3.2.1]{HabermannHMS}, although where the identification of the cylinders in $A(p-1,1;\frac{q-1}{\ell})$ here have been rotated $-\frac{2\pi}{p-1}$ degrees. It was established in loc. cit. that the surface glued in this way is graded symplectomorphic to the Milnor fibre of $\widecheck{\w}$ by comparing the corresponding ribbon graphs. Building on this strategy, we establish a graded symplectomorphism $\widecheck{V}/\widecheck{\Gamma}\simeq\Sigma_{\w_\text{loop},\Gamma}$ by first making a topological identification via the quotient ribbon graphs, and then deducing that the grading structures match by elimination. \\

Recall the description of $\widecheck{V}$ as $\widecheck{\w}_{\varepsilon}^{-1}(-\delta)$ for $0<\delta \ll \eps$ given in \cite[Section 3]{HabermannSmith}, where 
\begin{align*}
\widecheck{\w}_{\varepsilon}=\widecheck{\w}-\varepsilon\check{x}\check{y}=\check{x}\check{y}(\check{x}^{p-1}+\check{y}^{p-1}-\varepsilon)=\check{x}\check{y}\check{w}.
\end{align*}
Firstly, observe that the Morsification chosen is $\widecheck{\Gamma}$-equivariant, and so taking the quotient commutes with Morsifying. Moreover, since the quotient map is an unramified cover and the deformation retract preserves equivalence classes of the quotient map, the deformation retract which takes $\widecheck{V}$ to its ribbon graph also commutes with the quotient map. With respect to the classification of critical points in \cite[Section 3.1]{HabermannSmith}, we refer to neck regions which form by smoothing critical points of type $(i)$ as neck regions of type $(i)$, and the corresponding node in the ribbon graph as a node of type $(i)$. We refer similarly to neck regions and nodes of type $(ii)$ and $(iii)$. We index the nodes of type $(i)$ and $(ii)$ according to the $\xt$ and $\yt$ argument of the corresponding critical points, respectively. Then, the $l^{\text{th}}$ node of type $(i)$ is identified with the $(l+\frac{p-1}{\ell})^{\text{th}}$ node of type $(i)$ under the action of $\widecheck{\Gamma}$. Similarly, the $m^{\text{th}}$ node of type $(ii)$ is identified with the $(m-\frac{q-1}{\ell})^{\text{th}}$ node of type $(ii)$. This partitions the nodes of the ribbon graph.\\

To understand how $\widecheck{\Gamma}$ partitions the edges, recall that part of the basis for the first homology group of $\widecheck{V}$ is given by the Lagrangians $\vc{\check{y}\check{w}}{l}$ (resp. $\vc{\check{x}\check{w}}{m}$ and $\vc{\check{x}\check{y}}{}$), which were defined as the waist curves which form in the $l^{\text{th}}$ neck region of type $(i)$ (resp. the $m^{\text{th}}$ neck region of type (ii), and the neck region of type $(iii)$) upon smoothing.  Since Morsification commutes with the action of $\widecheck{\Gamma}$, the Lagrangians $\vc{\check{y}\check{w}}{l}$ and $\vc{\check{y}\check{w}}{l+\frac{p-1}{\ell}}$ become identified in the quotient, and therefore so to do the edges of the ribbon graph onto which these Lagrangians deformation retract. The analogous statement for the Lagrangians $\vc{\check{x}\check{w}}{m}$ is also true, and so we see that two loops of the graph are identified with each other when the corresponding nodes are. \\

To understand the action of $\widecheck{\Gamma}$ on the remaining edges, recall that two nodes are connected by an edge if there is a vanishing cycle which passes through both corresponding neck regions. The cyclic ordering of the nodes is determined by the argument of the Lagrangian away from the neck regions which it connects -- see, for example, \cite[Figure 8]{HabermannSmith}.  From this, it is clear that edges between the node of type $(iii)$ and nodes of type $(i)$ (resp. type $(ii)$) are identified in the quotient when the corresponding nodes of type $(i)$ (resp. type $(ii)$) are. All that remains is to understand the action of $\widecheck{\Gamma}$ on edges which connect the nodes of type $(i)$ and $(ii)$. For this, recall (\cite[Section 3.5]{HabermannSmith}) that the remaining vanishing cycles which form a basis of the first homology of $\widecheck{V}$ are given by $\vc{0}{l,m}$ for $l\in\{0,\dots, p-2\}$, $m\in\{0,\dots, q-2\}$, and these are the Lagrangians which pass through the $l^{\text{th}}$ neck region of type $(i)$ and the $m^{\text{th}}$ neck region of type $(ii)$. By analysing the action of $\widecheck{\Gamma}$ on the $\xt$ and $\yt$ projections of the Milnor fibre, as given in \cite[Section 3.3]{HabermannSmith}, we see that $\vc{0}{l,m}$ gets identified with $\vc{0}{l+\frac{p-1}{\ell},m-\frac{q-1}{\ell}}$ as it enters the $\{\check{w}=\varepsilon\}$ component the Milnor fibre\footnote{Note that we are not claiming that this identification is made globally, just that these two Lagrangians agree as they enter the $\{\check{w}=\varepsilon\}$ part of the curve.}. Away from the neck regions which connect it to the smoothings of the $\{\xt=0\}$ and $\{\yt=0\}$ components, $\{\check{w}=\varepsilon\}$ is an unramified cover of $\big\{\{u+v=\varepsilon\}\setminus(B_{\delta}(\varepsilon,0)\cup B_{\delta}(0,\varepsilon))\big\}\subseteq \C^2$, and so the Lagrangians $\vc{0}{l,m}$ and $\vc{0}{l+\frac{p-1}{\ell},m-\frac{q-1}{\ell}}$ get identified in the component $\{\check{w}=\varepsilon\}$. Therefore, the edge of the ribbon graph connecting the $l^{\text{th}}$ node of type $(i)$ with the $m^{\text{th}}$ node of type $(ii)$ gets identified with the edge connecting the $(l+\frac{p-1}{\ell})^{\text{th}}$ node of type $(i)$ with the $(m-\frac{q-1}{\ell})^{\text{th}}$ node of type $(ii)$ -- see Figure \ref{Ribbon_graph_unquotiented} for an example. Note that this identifies the cyclic ordering of the two nodes in a non-trivial way. Moreover, the pushforward of the basis of the first homology for the ribbon graph of $\widecheck{V}$ given by the deformation retract of vanishing cycles spans the first homology of the quotient ribbon graph. Therefore, the pushforward of vanishing cycles spans the first homology of $\widecheck{V}/\widecheck{\Gamma}$. It should be emphasised, however, that we are making no attempt to precisely describe a basis of Lagrangians on $\widecheck{V}/\widecheck{\Gamma}$; we only claim that the vanishing cycles span the first homology of $\widecheck{V}/\widecheck{\Gamma}$. In general, two Lagrangians $\vc{0}{l,m}$ and $\vc{0}{l+\frac{p-1}{\ell},m-\frac{q-1}{\ell}}$ are not isotopic in the quotient, but are related by Dehn twists around the waist curves of the cylinders through which they both pass.

\begin{figure}[h!]
	\centering
	\begin{tikzpicture}
	\tikzset{every loop/.style={}}
	\node[circle,fill=black] (C)at (0,0) {} edge [in=250,out=110,loop] ();
	\node[circle,fill=blue] (R1) at (1.5,2) {} edge [in=340,out=200,loop] ();
	\node[circle,fill=yellow] (R2) at (3,2) {} edge [in=340,out=200,loop] ();
	\node[circle,fill=black] (R3) at (4.5,2) {} edge [in=340,out=200,loop] ();
	\node[circle,fill=black] (R4) at (6,2) {} edge [in=340,out=200,loop] ();
	\node[circle,fill=blue] (L1) at (-1.5,2) {} edge [in=340,out=200,loop] ();
	\node[circle,fill=yellow] (L2) at (-3,2) {} edge [in=340,out=200,loop] ();
	\node[circle,fill=black] (L3) at (-4.5,2) {} edge [in=340,out=200,loop] ();
	\node[circle,fill=black] (L4) at (-6,2) {} edge [in=340,out=200,loop] ();
	\draw[red] (C) to [in=270,out=40] (R1);
	\draw[black!30!green] (C) to [in=270,out=20] (R2);
	\draw[red] (C) to [in=270,out=0] (R3);
	\draw[black!30!green] (C) to [in=270,out=-20] (R4);
	\draw[red] (C) to [in=270,out=140](L1);
	\draw[black!30!green] (C) to [in=270,out=160](L2);
	\draw[red] (C) to [in=270,out=180](L3);
	\draw[black!30!green] (C) to [in=270,out=200](L4);
	
	\draw[red] (R1) edge [in=144,out=144] (L1) ;
	\draw[black!30!green] (R2) edge [in=36, out=36, looseness=0.5] (L4);
	\draw[red] (4.5,2) (R3) edge [in=72,out=72, looseness=0.5] (L3);
	\draw[black!30!green] (R4) edge [in=108,out=108, looseness=0.6] (L2);
	\end{tikzpicture}
	\caption{Part of the ribbon graph corresponding to $\widecheck{V}$ for $\widecheck{\w}=\check{x}^5\check{y}+\check{y}^5\check{x}$. For clarity, we have only drawn the edges which form the cycles onto which the vanishing cycles $\vc{0}{i,-i}$ for $i\in\{0,1,2,3\}$ deformation retract. In the quotient of $\widecheck{V}$ by $\widecheck{\Gamma}=\mu_2$, the two red cycles and two green cycles are identified, and the representatives of the nodes are given by the blue and yellow nodes (recall $\arg \check{x}=-\arg\check{y}$), together with the node of type $(iii)$. In the case of $\widecheck{\Gamma}=\mu_4$, all coloured cycles are identified, and the blue nodes, as well as the node of type $(iii)$, are taken as the representative in the quotient.}
	\label{Ribbon_graph_unquotiented}
\end{figure}
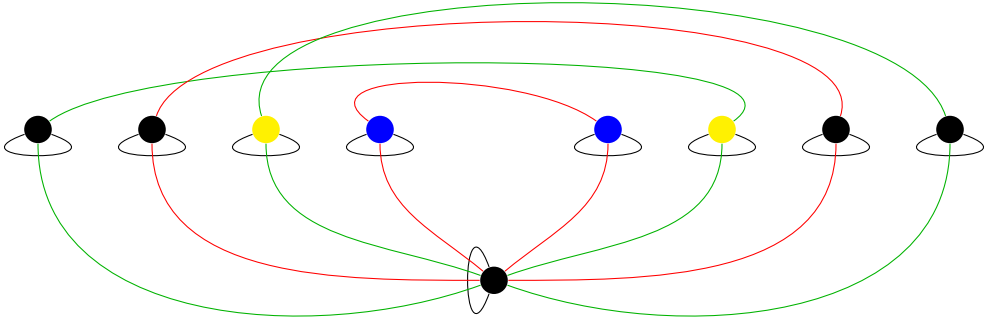

Since the cyclic ordering at nodes is identified in a non-trivial way, one must choose a representative of each equivalence class of nodes to work with a specific ordering. By convention, we will choose the nodes of type $(i)$ corresponding to the neck regions which arise from smoothing the critical points with argument $\arg \check{x}\in \{0,\frac{2\pi}{p-1},\dots, \frac{2\pi(p-1-\ell)}{\ell(p-1)}\}$, and similarly we choose the nodes of type $(ii)$ to correspond to the smoothing of the critical points of type $(ii)$ with $\arg \check{y}\in \{0,-\frac{2\pi}{q-1},\dots, -\frac{2\pi(q-1-\ell)}{\ell(q-1)}\}$. Figures \ref{Figure_mu2_quotient} and \ref{Figure_mu4_quotient} show the cases of $\widecheck{V}/\widecheck{\Gamma}$ for $\widecheck{V}$ the Milnor fibre of $\xt^5\yt+\yt^5\xt$ and $\widecheck{\Gamma}=\mu_2,\ \mu_4$, respectively. From this, we see that the surface corresponding to this quotient ribbon graph is given by gluing $A(p-1,1;\frac{q-1}{\ell})$, $A(\frac{q-1}{\ell},\frac{p-1}{\ell};1)$ and $A(1,q-1;\frac{p-1}{\ell})$ via the permutations $\sigma_1=\text{id}\in\mathfrak{S}_{\frac{q-1}{\ell}}$, $\sigma_2=\text{id}\in\mathfrak{S}_{\frac{p-1}{\ell}}$, and $\sigma_3\in\mathfrak{S}_{\frac{(p-1)(q-1)}{\ell}}$, where $\sigma_3$ is given by 
\begin{align*}
k_+(q-1)+i\mapsto \big((-i)\bmod \frac{q-1}{\ell}\big)(p-1)+(p-2-k_++\frac{p-1}{\ell}\lfloor\frac{-i\ell}{q-1}\rfloor).
\end{align*}
As in the maximally graded case, this only differs from the gluing given for $\Sigma_{\w_\text{loop},\Gamma}$ by changing the identification of the cylinders in the column $A(p-1,1;\frac{q-1}{\ell})$. \\
\begin{figure}[h!]
	\flushleft
	\begin{minipage}[b]{0.65\textwidth}
		\begin{tikzpicture}
		\tikzset{every loop/.style={}}
		\node[circle,fill=black] at (0,0) {} edge [in=250,out=110,loop] ();
		\node[circle,fill=black] at (2,2) {} edge [in=340,out=200,loop] ();
		\node[circle,fill=black] at (4,2) {} edge [in=340,out=200,loop] ();
		\node[circle,fill=black] at (-2,2) {} edge [in=340,out=200,loop] ();
		\node[circle,fill=black] at (-4,2) {} edge [in=340,out=200,loop] ();
		\draw (0,0) to [in=270,out=30] (2,2);
		\draw (0,0) to [in=270,out=-30] (4,2);
		\draw (0,0) to [in=270,out=150](-2,2);
		\draw (0,0) to [in=270,out=210](-4,2);
		
		\draw (2,2) to [in=144,out=144, ](-2,2);
		\draw (2,2) to [in=72,out=108,looseness=0.75](-4,2);
		\draw (2,2) to [in=72,out=72](-2,2);
		\draw (2,2) to [in=144,out=36](-4,2);
		
		\draw (4,2) to [in=108,out=144,looseness=1.25](-2,2);
		\draw (4,2) to [in=36,out=108,looseness=0.75](-4,2);
		\draw (4,2) to [in=36,out=72,looseness=0.75](-2,2);
		\draw (4,2) to [in=108,out=36,looseness=1](-4,2);
		\end{tikzpicture}	
		\caption{Ribbon graph corresponding to $\widecheck{V}/\mu_2$ for $\widecheck{\w}=\check{x}^5\check{y}+\check{y}^5\check{x}$.}
		\label{Figure_mu2_quotient}
	\end{minipage}%
	\begin{minipage}[b]{0.65\textwidth}
		\begin{tikzpicture}
		\tikzset{every loop/.style={}}
		\node[circle,fill=black] at (0,0) {} edge [in=250,out=110,loop] ();
		\node[circle,fill=black] at (2,2) {} edge [in=340,out=200,loop] ();
		\node[circle,fill=black] at (-2,2) {} edge [in=340,out=200,loop] ();
		\draw (0,0) to [in=270,out=0] (2,2);
		\draw (0,0) to [in=270,out=180](-2,2);
		
		\draw (2,2) to [in=144,out=144](-2,2);
		\draw (2,2) to [in=108,out=108](-2,2);
		\draw (2,2) to [in=72,out=72](-2,2);
		\draw (2,2) to [in=36,out=36](-2,2);
		\end{tikzpicture}
		\caption{Ribbon graph
			\newline corresponding to $\widecheck{V}/\mu_4$
			\newline for $\widecheck{\w}=\check{x}^5\check{y}+\check{y}^5\check{x}$.}
		\label{Figure_mu4_quotient}
	\end{minipage}
\end{figure}

To identify the line field used to grade $\widecheck{V}/\widecheck{\Gamma}$, observe that the pushforward of any vanishing cycle in $\widecheck{V}$ is gradable with respect to the line field which is horizontal on cylinders and parallel to the edges of the attaching strips, which we denote by $\eta$. Indeed, for the waist curves to be gradable, the only possible line field is the one which is horizontal on cylinders. To see that the pushforward of the Lagrangians $\vc{0}{l,m}$ are gradable with respect to $\eta$, observe that the pushforward of such a Lagrangian deformation retracts onto a cycle of the quotient ribbon graph which passes through three nodes, one each of type $(i)$, $(ii)$, and $(iii)$. Therefore, the pushforward Lagrangian is characterised by which attaching strips it passes through, as well as some number of Dehn twists about the waist curves in the cylinders which the attaching strips connect, and any such Lagrangian is gradable with respect to $\eta$. By the uniqueness (up to homotopy) of the line field with respect to which the pushforward of the vanishing cycles of $\widecheck{V}$ are all gradable, the line field on $\widecheck{V}/\widecheck{\Gamma}$ is homotopic to $\eta$. This completes the proof of Theorem \ref{InvertiblePolyTheorem} in the case of loop polynomials. 

\subsection{Chain polynomials} For a chain polynomial $\w=x^py+y^q$ we consider $\W=x^py+y^q+xyz$, and $\Gamma\subseteq \Gamma_{\w}$ of index $\ell$ with identification $\Gamma\simeq\C^*\times\mu_{\frac{d}{\ell}}$, where $d:=\gcd(p,q-1)$. We define the corresponding stack 
\begin{align*}
Z_{\w_{\mathrm{chain}},\Gamma}:=[\big(\W^{-1}(0)\setminus\{\boldsymbol{0}\}\big)/\Gamma],
\end{align*}
where $\Gamma$ acts by its inclusion into $\Gamma_{\w}$. This stack has two irreducible components -- the first is $\mathcal{C}_2=\{x^p+y^{q-1}+xz=0\}\simeq \P_{\frac{(p-1)(q-1)}{\ell},\frac{q-1}{\ell}}$, and the second we identify with a $\mu_{\frac{q-1}{\ell}}$-gerbe over $\P_{1,p-1}$ as follows: We identify $\mathcal{C}_1$ as the closed substack of $Z_{\w_{\mathrm{chain}},\Gamma}$ corresponding to the divisor $\{y=0\}$. Analogously to the loop case, we see that the quotient stack $[(\A^3\setminus\{\boldsymbol{0}\})/\Gamma]$ corresponds to the stacky fan given by the data of a morphism 
	\begin{align*}
	\beta:\Z^3\xrightarrow{\begin{pmatrix}
		1 & 1-q& 1 \\
		0 & \frac{(p-1)(q-1)}{\ell} & -\frac{p}{\ell}
		\end{pmatrix}} \Z^2=:N,
	\end{align*}
	and the rays of the fan $\Sigma$ correspond to the column vectors. The maximal cones of the fan are given by the span of any two rays. In general, this is a quotient of weighted projective space by $\mu_{\frac{d}{\ell}}$.\\
	
With this description, we see that $\mathcal{C}_1$ is the closed substack corresponding to the ray $\rho_2=(1-q)e_1+\frac{(p-1)(q-1)}{\ell}e_2$, and so $\mathcal{C}_1$ is given by the quotient fan consisting of the complete fan in $\mathbb{Q}$, $N=\Z\oplus\Z/\big(\frac{q-1}{\ell}\big)$, and 
\begin{align*}
\beta:\Z^2\xrightarrow{\begin{pmatrix}
	p-1 & -1\\
	-\frac{p}{\ell} &0
	\end{pmatrix}}\Z\oplus\Z/\big(\frac{q-1}{\ell}\big).
\end{align*}
Again, by \cite[Theorem 7.24]{FantechiMannNironi}, we see that there is an equivalence of toric DM stacks
\begin{align*}
\mathcal{C}_1\simeq\sqrt[\frac{q-1}{\ell}]{\OO(-\frac{p}{\ell}q_{1,-})/\P_{p-1,1}}.
\end{align*}
As in the loop case, the computation of the morphisms in the exceptional collection is done locally. To this end, consider a local presentation of the node $q_2=|\mathcal{C}_2|\cap|\mathcal{C}_1|=[0:0:1]$. This is given by the quotient of $xy=0$ by the action of $\mu_{\frac{(p-1)(q-1)}{\ell}}$ given by
\begin{align*}
t\cdot (x,y)=(tx,t^{q-1}y).
\end{align*}
This yields $\chi_{r_{1,-}}=q-1$ and $\chi_{r_{2,+}}=1$. Therefore, the presentation of the gerbe $\mathcal{C}_1$ at $q_{1,-}$ is determined by the class of $\frac{p}{\ell}\bmod \frac{q-1}{\ell}\in\Z/\gcd(p-1,\frac{q-1}{\ell})\simeq H^2([\A^1/\mu_{p-1}],\mu_{\frac{q-1}{\ell}})$. This gives the short exact sequence
\begin{align*}
1\rightarrow \mu_{\frac{q-1}{\ell}}\hookrightarrow\mu_{\frac{(p-1)(q-1)}{\ell}}\xrightarrow{\wedge^{q-1}}\mu_{p-1}\rightarrow 1.
\end{align*}
The action of $H_2$ on $\mathcal{N}_1$ at $q_{1,-}$ is such that $\frac{q-1}{\ell}\chi_{d_{1,-}}=\frac{p}{\ell}\chi_{r_{1,-}}$ in $\Z/(\frac{(p-1)(q-1)}{\ell})=\widehat{H}_2$, and a natural choice for this character is $\chi_{d_{1,-}}=1$. We order the characters in $\widehat{H}_2$ such that $\chi_c=-c$. With this ordering, the sheaf on $\widetilde{\mathcal{C}}_1$ whose fibre at $q_{1,-}$ is acted on by $H_2$ with character $\chi_c$ is given by
\begin{align*}
\OO_{\widetilde{\mathcal{C}}_1}(jq_{1,-})\otimes\mathcal{N}_1^{\otimes k_-},
\end{align*}
where 
\begin{align}\label{ChainWeights}
\begin{split}
k_-&=-c\bmod\frac{q-1}{\ell}\\
j&=-\frac{p}{\ell}\lfloor\frac{-c\ell}{q-1}\rfloor\bmod p-1.
\end{split}
\end{align}
From this, one can see that we have the following morphisms in the exceptional collection: 
\begin{align*}
&\Ext^1(\mathcal{S}_{q_{2}}\{\chi_c\},\mathcal{P}_2(0,c-1))=\C\cdot a_2(c)\\
&\Ext^1(\mathcal{S}_{q_{2}}\{\chi_c\},\mathcal{P}_1((-1-j)\bmod p-1,-1,k_-))=\C\cdot b_2(-j,k_-)
\end{align*}
for $j,k_-$ as in \eqref{ChainWeights}. \\

As in the loop case, the analysis of the node $q_1=|\mathcal{C}_1|\cap|\mathcal{C}_2|$ is determined by the choice of $\chi_{d_{1,-}}$. In particular, we have $\widehat{H}_1=\Z/(\frac{q-1}{\ell})$, $\chi_{r_{2,-}}=1$, and take $\chi_{d_{1,+}}=-1$. We again order the elements of $\widehat{H}_2$ such that $\chi_c=-c$, and with this we have the following morphisms in the exceptional collection:
\begin{align*}
&\Ext^1(\mathcal{S}_{q_1}\{\chi_c\},\mathcal{P}_1(0,-1,c))=\C\cdot a_1(0,c)\\
&\Ext^1(\mathcal{S}_{q_1}\{\chi_c\},\mathcal{P}_2((c-1)\bmod \frac{(p-1)(q-1)}{\ell},-1))=\C\cdot b_1(c).
\end{align*}
To construct the mirror to this curve, we glue together two columns, $A(p-1,1;\frac{q-1}{\ell})$ and $A(\frac{q-1}{\ell},\frac{(p-1)(q-1)}{\ell};1)$  via the permutation $\sigma_1=\text{id}\in\mathfrak{S}_{\frac{q-1}{\ell}}$ gluing the first column to the second, and the permutation $\sigma_2\in\mathfrak{S}_{\frac{(p-1)(q-1)}{\ell}}$ given by
\begin{align*}
c\mapsto k_-(p-1)+(-j)\bmod p-1
\end{align*}
for $k_-,j$ as in \eqref{ChainWeights} gluing the second column back to the first. From this, it is clear that there is one boundary component arising from the first gluing, and this has winding number $-2\frac{q-1}{\ell}$. From the second gluing, we have that $\sigma_2^{-1}\tau_{\ell_1}\sigma_2\tau_{r_2}$ is given by
\begin{align*}
c\mapsto c-q,
\end{align*}
and so there are $\gcd(q,\frac{p+q-1}{\ell})$ boundary components, each with winding number $-2\frac{(p-1)(q-1)}{\gcd(\ell q,p+q-1)}$.\\

Putting this all together, we have constructed a surface, call it $\Sigma_{\w_{\mathrm{chain}},\Gamma}$, which has $1+\gcd(q,\frac{p+q-1}{\ell})$ components, Euler characteristic 
\begin{align*}
-\chi(\Sigma_{\w_{\mathrm{chain}},\Gamma})=\frac{p(q-1)}{\ell},
\end{align*} 
and genus 
\begin{align*}
g_{\mathrm{chain}}=\frac{1}{2\ell}(pq-p+\ell-\gcd(\ell q,p+q-1)).
\end{align*}
Applying Theorem \ref{WrappedTheorem} yields a quasi equivalence 
\begin{align*}
D^b(\mathcal{A}_{Z_{\w_{\mathrm{chain}},\Gamma}}-\mathrm{mod})\simeq\mathcal{W}\bigg(\Sigma_{\w_{\mathrm{chain}},\Gamma}; 2\frac{q-1}{\ell},\Big(2\frac{(p-1)(q-1)}{\gcd(\ell q,p+q-1)}\Big)^{\gcd(q,\frac{p+q-1}{\ell})}\bigg).
\end{align*}
Applying Theorem \ref{HMSTheorem} yields 
\begin{align*}
D^b\Coh(Z_{\w_{\mathrm{chain}},\Gamma})\simeq\mathcal{W}(\Sigma_{\w_{\mathrm{chain}},\Gamma}),\\
\perf Z_{\w_{\mathrm{chain}},\Gamma}\simeq\mathcal{F}(\Sigma_{\w_{\mathrm{chain}},\Gamma}).
\end{align*}
In the case of maximally graded chain polynomials, observe that the above description differs from that of \cite[Section 3.2.2]{HabermannHMS} only by a rotation of the identification of the left boundary of the first annulus in the first column. Therefore, the surface constructed in the maximally graded case is graded symplectomorphic to the Milnor fibre of $\widecheck{\w}$ in the maximally graded case. In the case of $\ell>1$, we follow the same strategy as in Section \ref{loopHMS} to deduce that $\widecheck{V}/\widecheck{\Gamma}$ is graded symplectomorphic to $\Sigma_{\mathrm{chain},\Gamma}$, and this establishes Theorem \ref{InvertiblePolyTheorem} in the case of chain polynomials. 
\subsection{Brieskorn--Pham polynomials}
The case of Brieskorn--Pham polynomials is covered in \cite{LekiliPolishchukAuslander}, although we include it here for completeness. For each Brieskorn--Pham polynomial $\w=x^p+y^q$, we consider $\W=x^p+y^q+xyz$, and $\Gamma\subseteq \Gamma_{\w}$ a subgroup of index $\ell$ containing the group generated by the grading element with identification $\Gamma\simeq\C^*\times \mu_{\frac{d}{\ell}}$. As in the previous cases, we define
\begin{align*}
Z_{\w_{\mathrm{BP}},\Gamma}=[(\W^{-1}(0)\setminus\{\boldsymbol{0}\})/\Gamma],
\end{align*}
where $\Gamma$ acts by its inclusion into $\Gamma_{\w}$. This stack has one irreducible component, whose coarse moduli space is a nodal rational curve, and the normalisation is given by $\widetilde{\mathcal{C}}\simeq\P_{\frac{(p-1)(q-1)-1}{\ell},\frac{(p-1)(q-1)-1}{\ell}}$. We identify the coordinates in the patch of $\widetilde{\mathcal{C}}$ containing $q_{+}=\infty$ as $x$, and in the patch containing $q_{-}=0$ as $y$. Therefore, the presentation of $\mathcal{C}$ around the node $q$ is given by the quotient of $xy=0$ by $H=\mu_{\frac{(p-1)(q-1)-1}{\ell}}$, where the action is given by
\begin{align*}
t\cdot (x,y)=(t^{q-1}x,ty).
\end{align*}
Correspondingly, $H$ acts on the fibre of $\OO(-q_-)$ at $q_{-}$ with weight $1$, and with weight $q-1$ on the fibre $\OO(-q_{+})$ at $q_+$.\\

In $\widehat{H}=\Z/(\frac{(p-1)(q-1)-1}{\ell})$, we label the characters such that $\chi_c=-c(q-1)$. Then, for each $c\in\Z/(\frac{(p-1)(q-1)-1}{\ell})$, we have the following morphisms in the exceptional collection:
\begin{align*}
&\Ext^1(\mathcal{S}_q\{\chi_c\},\mathcal{P}(0,c-1))=\C\cdot a(c)\\
&\Ext^1(\mathcal{S}_q\{\chi_c\},\mathcal{P}((c(q-1)-1)\bmod \frac{(p-1)(q-1)-1}{\ell},-1))=\C\cdot b(c(q-1))
\end{align*}
Correspondingly, the mirror surface is given by gluing the annulus $A(\frac{(p-1)(q-1)-1}{\ell},\frac{(p-1)(q-1)-1}{\ell};1)$ to itself via the permutation $\sigma\in\mathfrak{S}_{\frac{(p-1)(q-1)-1}{\ell}}$ given by
\begin{align*}
c\mapsto -c(q-1).
\end{align*}
The commutator $[\sigma,\tau]\in\mathfrak{S}_{\frac{(p-1)(q-1)-1}{\ell}}$, where $\tau$ is the permutation $c\mapsto c-1$, is given by
\begin{align*}
c\mapsto c-p.
\end{align*}
Correspondingly, the constructed surface, call it $\Sigma_{\w_{\mathrm{BP}},\Gamma}$, has $\gcd(q,\frac{p+q}{\ell})=\gcd(p,\frac{p+q}{\ell})$ boundary components, each of winding number $-2\frac{(p-1)(q-1)-1}{\gcd(\ell q,p+q)}$. Therefore, the Euler characteristic is 
\begin{align*}
-\chi(\Sigma_{\w_{\mathrm{BP}},\Gamma})=\frac{(p-1)(q-1)-1}{\ell},
\end{align*}
and the genus is 
\begin{align*}
g_{\mathrm{BP}}=\frac{1}{2\ell}(2\ell-1+(p-1)(q-1)-\gcd(\ell q,p+q)).
\end{align*}
Applying Theorem \ref{WrappedTheorem} yields
\begin{align*}
D^b(\mathcal{A}_{Z_{\w_{\mathrm{BP}},\Gamma}}-\mathrm{mod})\simeq\mathcal{W}\bigg(\Sigma_{\w_{\mathrm{BP}},\Gamma};\Big(2\frac{(p-1)(q-1)-1}{\gcd(\ell q, p+q)}\Big)^{\gcd(q,\frac{p+q}{\ell})}\bigg),
\end{align*}
and applying Theorem \ref{HMSTheorem} yields 
\begin{align*}
D^b\Coh(Z_{\w_{\mathrm{BP}},\Gamma})\simeq \mathcal{W}(\Sigma_{\w_{\mathrm{BP}},\Gamma}),\\
\perf Z_{\w_{\mathrm{BP}},\Gamma} \simeq \mathcal{F}(\Sigma_{\w_{\mathrm{BP}},\Gamma}).
\end{align*}
In the maximally graded case, the description of the mirror surface matches that of \cite[Section 3.2.3]{HabermannHMS} on-the-nose, and so is graded symplectomorphic to the Milnor fibre of $\widecheck{\w}$. The proof that $\widecheck{V}/\widecheck{\Gamma}$ is graded symplectomorphic to $\Sigma_{\w_{\mathrm{BP}},\Gamma}$ follows as in the loop and chain cases, and this completes the proof of Theorem \ref{InvertiblePolyTheorem}.
\appendix
\section{Root stacks}\label{root stack section}
\setcounter{equation}{0}
\renewcommand{\theequation}{\thesection.\arabic{equation}}
In this appendix, we briefly recall the relevant notions of root stacks and gerbes. This theory is well studied and developed far beyond the scope of application in this paper; we aim here only to provide a self-contained account of the relevant aspects of the theory in the context of how we use it. The notion of a root stack was introduced independently in \cite{Cadman2003UsingST} and \cite{AGV}, to which we refer to for more details. In addition, the book \cite{olsson2016algebraic} also provides an excellent exposition.\\
There are two related notions of a root stack -- the first is a way to `insert stackiness' along an effective Cartier divisor, and the second defines a gerbe structure, which `inserts stackiness' everywhere, and also keeps track of the generic stabiliser. \\

Recall that the stack $[\A^1/\C^*]$ is the classifying stack of line bundles with section -- this can be seen by considering a morphism to this stack as a principal $\C^*$-bundle with a global section of the associated line bundle. To define the root stack of a line bundle with section, consider $X$ a scheme, $\mathscr{L}$ an invertible sheaf on $X$, $s\in \Gamma(X,\mathscr{L})$ a global section, and $r>0$ an integer. Moreover, let $\theta_r:\big[\A^1/\C^*\big] \rightarrow \big[\A^1/\C^*\big] $ be the $r^\text{th}$ power map on both $\A^1$ and $\C^*$. 
\begin{mdef}[{\cite[Definition 2.2.1]{Cadman2003UsingST}, \cite[Appendix B.2]{AGV}}]
	Define the stack $X_{(\mathscr{L},s,r)}$ to be the fibre product
	\begin{equation*}
	\begin{tikzcd}
	X_{(\mathscr{L},s,r)} \arrow[r,"\text{pr}_2"] \arrow[d,"\text{pr}_1"] & \big[\A^1/\C^*\big] \arrow[d,"\theta_r"]\\
	X \arrow[r,"(\mathscr{L}\comma s)"] & \big[\A^1/\C^*\big]
	\end{tikzcd}
	\end{equation*}
\end{mdef}
This is a Deligne--Mumford stack (\cite[Theorem 2.3.3]{Cadman2003UsingST}), and is isomorphic to $X$ away from the divisor $s^{-1}(0)$. By construction, $X_{(\mathscr{L},s,r)}$ comes with a line bundle ${N}$ and a section $t\in\Gamma(X_{(\mathscr{L},s,r)},{N})$ such that $\varphi: {N}^{\otimes r}\xrightarrow{\sim}\text{pr}_1^*\mathscr{L}$, and $\varphi(t^r)=\text{pr}_1^*s$. Moreover, the construction can be generalised for when $X$ is a Deligne--Mumford stack.\\

For an effective Cartier divisor, we will also use the notation $X_{(D,r)}$ to mean $X_{(\OO_X(D),1_D,r)}$, where $1_D$ is the tautological section vanishing along $D$. One can iterate this root construction, and for $\mathbb{D}=(D_1,\dots, D_n)$ and $\vec{r}= (r_1,\dots, r_n)$, we define $X_{\mathbb{D},\vec{r}}$ to be the root stack defined by iteratively applying the above construction. \\

An important example for us will be the following:
\begin{ex}[{\cite[Lemma 2.3.1]{Cadman2003UsingST}}]\label{RootStackExample}
	For $X=\A^1$, and $D=[0]$, there is an equivalence of categories $X_{(D,r)}\simeq \big[\A^1/\mu_r\big]$, where $\mu_r$ acts via its natural character.
\end{ex}
In fact, Example \ref{RootStackExample} can be generalised (\cite[Example 2.4.1]{Cadman2003UsingST}, cf. \cite[Theorem 10.3.10]{olsson2016algebraic}) to any $X=\Spec A$ and $\mathscr{L}=\OO_X$, with $s\in \Gamma(X,\OO_X)$ such that $D=s^{-1}(0)$, yielding
\begin{align*}
X_{(D,r)}\simeq \big[\big(\Spec A[x]/(x^r-s)\big)/\mu_r\big],
\end{align*}
where $\mu_r$ acts by $t\cdot x=t^{-1}x$, and $t\cdot a=ta$. In general, any root stack can be covered by such affine root stacks. For further exposition on root stacks of line bundles with section we refer to the original references \cite{AGV}, \cite{Cadman2003UsingST}, as well as \cite[Section 10.3]{olsson2016algebraic}. \\

The second flavour of root stack defines a gerbe over the original scheme (or stack), and we refer to \cite[Chapter 12]{olsson2016algebraic} for a definition and further discussion about generalities of gerbes. Gerbes were originally introduced in \cite{giraud1971cohomologie}, and can, roughly speaking, be thought of as a `$BG$-bundle' over $X$ for some group $G$. In particular, this means that not only does the isotropy group of each point contain a copy of $G$, but the identification of this copy of $G$ in the automorphism group of each point is a crucial part of the definition. In particular, an equivalence of gerbes is an equivalence of categories which is compatible with these identifications. Note that this means that two gerbes can be equivalent as stacks, but inequivalent as gerbes, in analogy with how two principal $G$-bundles can have diffeomorphic total spaces, but are not isomorphic $G$-bundles. For example, principal $S^3$ bundles over $S^4$ are classified by $\Z\oplus\Z$, and \cite{CrowleyEscher} establishes an explicit diffeomorphism between the total spaces of the bundles classified by $(1,1)$ and $(2,0)$.  In what follows, we will restrict ourselves to the case at hand and only consider trivially banded gerbes, which are classified by $H^2(X,G)$. 
\begin{ex}
	If one considers the topological setting, then a good example to have in mind is given by the observation that any principal $S^1$-bundle is in fact a $\Z$-gerbe, since $B\Z\simeq K(\Z,1)\simeq S^1$. From this, we recover the usual classification of principal $S^1$-bundles by the Euler class in $H^2(X,\Z)$.
\end{ex}

To define a root stack of a line bundle (without section), consider $\mathscr{L}\in\Pic X$. Recall that such a line bundle is equivalent to a map $X\xrightarrow{\mathscr{L}} B\C^*$, and let $B\C^*\xrightarrow{\wedge^d}B\C^*$ be the $d^{\text{th}}$ power map. Then, we have:
\begin{mdef}[{\cite[Definition 2.2.6]{Cadman2003UsingST}, \cite[Appendix B.1]{AGV}}]
	The stack $X_{(\mathscr{L},d)}$ is defined to be the fibre product
	\begin{equation*}
	\begin{tikzcd}
	X_{(\mathscr{L},d)} \arrow[r,"\text{pr}_2"] \arrow[d,"\text{pr}_1"] & B\C^* \arrow[d,"\wedge^d"]\\
	X \arrow[r,"\mathscr{L}"] & B\C^*
	\end{tikzcd}
	\end{equation*}
\end{mdef}
The stack $X_{(\mathscr{L},d)}$ is a $\mu_d$-gerbe over $X$, and, by construction, there is a line bundle $\mathcal{N}\in \Pic X_{(\mathscr{L},d)}$ such that
\begin{align*}
\mathcal{N}^{\otimes d}\simeq \text{pr}_1^*\mathscr{L}.
\end{align*}
Of course, there is also a corresponding iterated statement (see, for example \cite[Proposition 6.9]{FantechiMannNironi}), although we will not make use of it. We will mainly use the notation $X_{(\mathscr{L},d)}=\sqrt[d]{\mathscr{L}/X}$. \\

Perhaps a more geometric way to think of a root stack of a line bundle is given in \cite[Appendix B.1]{AGV}. Let $\mathscr{L}$ be a line bundle on a scheme $X$, and $\mathscr{L}^*$ be the total space minus the zero section (i.e. the principal $\C^*$-bundle associated to $\mathscr{L}$). Then, 
\begin{align*}
\sqrt[d]{\mathscr{L}/X}=[\mathscr{L}^*/\C^*],
\end{align*}
where $\C^*$ acts fibrewise with weight $d$. In particular, the usual description of the weighted projective stack $\P(d,d)$ is recovered as $\sqrt[d]{\OO(-1)/\P^1}$, since $\OO(-1)^*=\A^2\setminus\{(0,0)\}$.
\begin{rmk}
	It should be noted that $X_{(\mathscr{L},d)}$ and $X_{(\mathscr{L},0,d)}$ are \emph{not} equivalent. Indeed, as is demonstrated in \cite[Example 2.4.3]{Cadman2003UsingST}, the latter category is an infinitesimal thickening of the former. 
\end{rmk}
The Kummer sequence 
\begin{align}\label{Kummer}
1\rightarrow\mu_d\xrightarrow{\iota}\G \xrightarrow{\wedge^d}\G\rightarrow 1
\end{align}
induces a long exact sequence on cohomology 
\begin{align}
\dots \rightarrow H^1(X,\G)\xrightarrow{\partial} H^2(X, \mu_d)\xrightarrow{\iota_*} H^2(X,\G)\rightarrow \dots.
\end{align}
For a root stack $\sqrt[d]{\mathscr{L}/X}$, the corresponding class in $H^2(X,\mu_d)$ is the image of $\mathscr{L}\in H^1(X,\G)\simeq \Pic X$ under the connecting homomorphism. Conversely, a $\mu_d$-gerbe is called \emph{essentially trivial} if its corresponding class in $H^2(X,\mu_d)$ is in the image of the connecting homomorphism. In particular, in the case where $H^2(X,\G)= 0$, we make the identification 
\begin{align*}
H^2(X,\mu_d)\simeq \Pic X/ d\Pic X,
\end{align*}
and so the cohomology class classifying the $d^{\text{th}}$ root of $\mathscr{L}$ is given by the quotient of its corresponding class in the Picard group, namely its first Chern class. Moreover, in this case \cite[Lemma 6.5]{FantechiMannNironi} identifies $H^2(X,\mu_d) \simeq \text{Ext}^1_\Z(\Z/d,\Pic X)$, where a class $[\mathscr{L}]\in \Pic X/d\Pic X$ corresponds to the short exact sequence
\begin{align}\label{GerbeClassification SES}
0\rightarrow \Pic X\rightarrow \Pic X{\times_{\Pic X/d\Pic X}}\Z/d\rightarrow \Z/d\rightarrow 0,
\end{align}
where the map $\Pic X\rightarrow \Pic X/d\Pic X$ is the projection, the map $\Z/d\rightarrow \Pic X/d\Pic X$ is given by $1\mapsto[\mathscr{L}]$, and the first morphism of the extension is $\mathscr{L}\rightarrow (\mathscr{L}^{\otimes d},0)$. \\

For each $\mu_d$-gerbe $\mathcal{X}$, there is an \emph{underlying orbifold}. This is the stack which results from the stackification of the prestack whose objects are the same as the original stack, but whose isotropy groups are quotiented by $\mu_d$. This process is known as \emph{rigidification}, although we refer to Appendix C of \cite{AGV} for the precise details. It suffices for us to observe that, in the case where the gerbe is the stack of roots of a line bundle on a scheme or orbifold, the map $\text{pr}_1:\sqrt[d]{\mathscr{L}/X}\rightarrow X$ is the rigidification map. In particular, for $\mathcal{X}= \sqrt[d]{\mathscr{L}/X}$ and $D$ a Cartier divisor on $X$, by $\OO_{{\mathcal{X}}}(D)$ we mean $\text{pr}_1^*\OO_{X}(D)$. 

\begin{ex}
	The most basic example of a gerbe is given by considering $B\mu_d$ to be a $\mu_d$-gerbe over a point. 
\end{ex}
\begin{ex}
	Consider an orbifold $X$ and the trivial action of $\mu_d$ on $X$. Then the resulting quotient stack is given by $X\times B\mu_d$, and corresponds to the stack of $d^\text{th}$ roots of $\OO_X$, or indeed any line bundle on $X$ whose $d^{\text{th}}$ root exists in $\Pic X$. 
\end{ex}
\begin{ex}
	Consider the compactified moduli space of elliptic curves $\overline{\mathcal{M}}_{1,1}\simeq \P(4,6)$. This is a $\Z/2$-gerbe over $\P(2,3)$, where the $\Z/2$-torsor corresponds to the symmetry present in any lattice defining an elliptic curve. It can be constructed as the stack of square roots of any line bundle  $\mathscr{L}\in\Pic \P(2,3)\simeq\Z$ such that $[\mathscr{L}]\in\Pic \P(2,3)/2\Pic \P(2,3)\simeq \Z/2$ is non-trivial. In this case, $\P(2,3)$ is the rigidification of the moduli space of elliptic curves. 
\end{ex}
\begin{ex}
	Consider the short exact sequence
	\begin{align*}
	1\rightarrow K\rightarrow H\rightarrow G\rightarrow 1,
	\end{align*}
	where $K$, $H$, and $G$ are all finite abelian groups. Then $BH\rightarrow BG$ is a $K$-gerbe, so is classified by $H^2(BG,K)\simeq H^2(G,K)$, which recovers the usual classification of short exact sequences in terms of group cohomology. Moreover, this is the local structure at any geometric point of a $K$-gerbe. 
\end{ex}
\begin{rmk}
	Note that above, and in what follows, we are implicitly taking $K$ to have the structure of a trivial $G$-module since we are only considering the case of trivially banded gerbes. For the remainder of the paper, we will only consider the cases where $G$ and $K$ are cyclic groups, and so we have $H^2(G,K)\simeq\text{Ext}^1_\Z(G,K)$ by the universal coefficient theorem.
\end{rmk}

We will exclusively deal with root stacks, both with and without section, over $\P^1$. To this end, consider $D_1=[0]=q_-$ and $D_2=[\infty]=q_+$ and $\vec{r}=(a,b)$. Then we define 
\begin{align*}
\P_{a,b}:=\P^1_{\mathbb{D},\vec{r}}
\end{align*}
to be the weighted projective line with a stacky point of order $a$ at $q_-$ and of order $b$ at $q_+$. Unless $\gcd(a,b)=1$, then this is not a weighted projective space; however, if this is the case then we have 
\begin{align*}
\P_{a,b}\simeq \P(a,b):=\big[\A^2\setminus\{(0,0)\}/\C^*\big],
\end{align*}
where $\C^*$ acts on $\A^2\setminus\{(0,0)\}$ with weights $a$ and $b$. Note that $H^2(\P_{a,b},\G)=0$, and so all gerbes whose underlying orbifold is $\P_{a,b}$ are essentially trivial.\\

Given a $\mu_d$-gerbe over $\P_{a,b}$, $\mathcal{C}$, the structure of the gerbe at the points $q_{\pm}$ will be of central importance to us. Observe that there is a natural (surjective) map
\begin{align}\label{projectingMV}
H^2(\P_{a,b},\mu_d)\rightarrow H^2([\A^1/\mu_a],\mu_d)\oplus H^2([\A^1/\mu_b],\mu_d)
\end{align}
which comes from the Mayer--Vietoris sequence, and this determines the Ext-class at $q_{\pm}$ which locally describes the gerbe. Explicitly, let $\mathcal{U}_-=[\A^1/\mu_a]$, suppose that $\mathcal{C}=\sqrt[d]{\mathcal{L}/\P_{a,b}}$, and that $\mathcal{L}|_{\mathcal{{U}_-}}\simeq\OO_{\mathcal{U}_-}(nq_-)$ has class $\beta\in\Pic\mathcal{U}_{-}\simeq \Z/a$. Observe that $H^2([\A^1/\mu_a],\mu_d)\simeq\Z/\gcd(a,d)$, and that the reduction $\beta \bmod d$ yields an element $[\beta]\in \Z/\gcd(a,d)$ determining a short exact sequence
\begin{align}\label{HSES}
1\rightarrow\mu_d \rightarrow H_-\rightarrow \mu_a\rightarrow 1,
\end{align}
classifying the gerbe on the patch $\mathcal{U}_-$, and corresponding to the $d^{\text{th}}$ root of $\OO_{\mathcal{U}_-}(nq_-)$. By construction, there exists a (not unique!) character $\chi_{d_-}$ of $H_-$ such that $H_-$ acts via $d\chi_{d_-}$ on the fibre of $\text{pr}_1^*\OO_{\mathcal{U}_-}(nq)$ at the origin, and which pulls back via the inclusion of $\mu_d$ to $H_-$ to a unit in $\Z/d$. Therefore, as $\mathcal{N}|_{\mathcal{U}_-}$ we take the equivariant sheaf on $\A^1$ where $H_-$ acts via $\chi_{d_-}$ on the fibre at the origin. By construction, for any $\chi\in \widehat{H}_-$, there is a unique $k\in \{0,\dots, d-1\}$ and $j\in\{m,\dots, m+a-1\}$ such that $H_-$ acts on the fibre of the sheaf
\begin{align}\label{SheafWeight}
\text{pr}_1^*\OO_{\mathcal{U}_-}(jq)\otimes\mathcal{N}^{\otimes k}
\end{align}
at the origin with character $\chi$. The local description of the gerbe on the patch $\mathcal{U}_+=[\A^1/\mu_b]$ is analogous, giving the local description of the gerbe on the two patches of $\P_{a,b}$. Conversely, the description of a gerbe on $\P_{a,b}$ is given by the local description on $\mathcal{U}_{\pm}$, together with the information of how the two local descriptions get identified on the overlapping $\C^*=\mathcal{U}_+\cap \mathcal{U}_-$.\\

There is a strong link between the derived categories of root stacks and the representation theory of finite dimensional algebras. If one takes $a=b=1$, then this relationship is classical, and is Be\u{\i}linson's result (\cite{Beilinson1978CoherentSO}) that
\begin{align*}
D^b(\P^1)\simeq D^b(\Lambda^{\text{op}}-\text{mod}),
\end{align*} 
where $\Lambda$ is the path algebra of the Kronecker quiver. This was generalised by Geigle and Lenzing in \cite{GeigleLenzing} to the situation $\P^1_{\mathbb{D},\vec{r}}$, where $\mathbb{D}$ is a finite collection of disjoint points with multiplicity one, and $\vec{r}$ is a tuple of positive integers. In particular, for $\mathbb{D}=(q_-,q_+)$ and $\vec{r}=(a,b)$ as above, it was shown that 
\begin{align*}
D^b(\P_{a,b})\simeq D^b(\Lambda_{a,b}^{\text{op}}-\text{mod}),
\end{align*}
where $\Lambda_{a,b}$ is the path algebra of the quiver 
\begin{equation}\label{exceptional collection}
\begin{tikzcd}
\OO(-aq_{-})\arrow[r,"x"] \arrow[d,equal] & \OO(-(a-1)q_-)\arrow[r,"x"] & \dots \arrow[r,"x"] & \OO(-q_-)\arrow[r,"x"] & \OO \arrow[d,equal]\\
\OO(-bq_+)\arrow[r,"y"] &\OO(-(b-1)q_-)\arrow[r,"y"] & \dots \arrow[r,"y"] & \OO(-q_+) \arrow[r,"y"]& \OO.
\end{tikzcd}
\end{equation}

This can also be viewed as a simple example of a canonical algebra, as introduced by Ringel in \cite{Ringel_1990}.\\
As for sheaves on the gerbes constructed as the root stacks over orbifold curves, consider $\mathcal{C}=\sqrt[d]{\mathscr{L}/\P_{a,b}}$ for some $\mathscr{L}\in\Pic \P_{a,b}$. There are natural full and faithful functors
\begin{align*}
\Phi_i:\Coh\P^1_{a,b}&\rightarrow \Coh\mathcal{C}\\
\mathcal{F}&\mapsto \text{pr}_1^*\mathcal{F}\otimes\mathcal{N}^{\otimes i},
\end{align*}
where $\text{pr}_1:\mathcal{C}\rightarrow \P_{a,b}$ is again the rigidification map. Taking the direct sum yields a special case of \cite[Theorem 1.5]{IshiiUeda}, giving an equivalence
\begin{align}\label{IshiiUedaRootStack}
\Coh\mathcal{C}\simeq \big(\Coh\P_{a,b}\big)^{\oplus d}.
\end{align}
Note that is not just semi-orthogonal, but also orthogonal, and that the equivalence is at the level of abelian categories. Therefore, the derived category of coherent sheaves on a gerbe over a weighted projective line only depends on the generic stabiliser group and the underlying weighted projective line.\\

It is essentially because of \eqref{IshiiUedaRootStack} that our results are independent of the choice of gerbe structure on irreducible components. To elaborate, consider $\mathcal{C}$ to be a chain of curves with two irreducible components which has isotropy group $H$ at their intersection; the general case proceeds inductively. One can construct $\mathcal{C}$ as the pushout
\begin{equation}
\begin{tikzcd}
& \mathcal{C}_1\\
\mathcal{C}_2 & BH \arrow{l}[above]{\phi}, \ar[u, hook]
\end{tikzcd}
\end{equation}
where $\phi:BH\rightarrow \mathcal{C}_2$ is the composition of the autoequivalence of $BH$ induced from the action of $H$ on the node, followed by its inclusion into $\mathcal{C}_2$. Since the abelian (and hence derived) categories of $\mathcal{C}_1$ and $\mathcal{C}_2$ are independent of gerbe structures by \eqref{IshiiUedaRootStack}, the only information required to understand the category of coherent sheaves of $\mathcal{C}$ is the autoequivalence of $BH$, and this is independent of the gerbe structure chosen, as well as the characters $\chi_{d_{1,+}}$ and $\chi_{d_{2,-}}$. 
\subsection{Root stacks and stacky fans}
A key step in our argument of \cref{InvertiblePolyTheorem} is the identification of the B--model with a cycle of nodal stacky curves. To do this, we view the B--model as a hypersurface in a toric Deligne--Mumford stack, which we now briefly review the theory of. This theory was initiated in \cite{BCS}, and the relationship with gerbes and root stacks was explored in \cite{FantechiMannNironi}. As with gerbes in general, this theory is developed well beyond the scope of what is required here, and we aim only to briefly recount the required background; we refer to the original sources for more detail. \\

Analogously to a toric variety, which contains an open dense torus $T$, a toric Deligne--Mumford stack is defined to be a smooth, separated Deligne--Mumford stack with an open immersion of a Deligne--Mumford torus, $T\times BG$ for $G$ a finite abelian group, such that the action of $T\times BG$ on itself extends to the whole stack (\cite[Section 3]{FantechiMannNironi}). The data of a stacky fan is given by a triple $\boldsymbol{\Sigma}=(\Sigma, N,\beta)$, where: 
\begin{itemize}
	\item $N$ is a finitely generated abelian group (not necessarily torsion-free),
	\item $\Sigma$ is a fan in $N_\mathbb{Q}=N\otimes_{\mathbb{Z}}\mathbb{Q}$ with $n$ rays such that the rays span $N_{\mathbb{Q}}$, and 
	\item $\beta:\Z^n\rightarrow N$ is a morphism of groups such that the image of the $i^{\text{th}}$ basis vector of $\Z^n$ in $N_{\mathbb{Q}}$ is on the $i^{\text{th}}$ ray. 
\end{itemize}
For simplicity, we will always assume that $\Sigma$ is complete. From this data, one can construct a toric DM stack in analogy with the  Cox construction for toric varieties (\cite{CoxToricVar}) as follows.
Let $d$ be the rank of $N$, and choose a projective resolution 
\begin{align*}
0\rightarrow \Z^\ell\xrightarrow{Q}\Z^{d+\ell}\rightarrow N\rightarrow 0.
\end{align*}
Then, choose a map $B:\Z^n\rightarrow \Z^{d+\ell}$ which lifts $\beta$. The cone of $\beta$, considered as a morphism of complexes $[0\rightarrow \Z^n]\rightarrow [0\rightarrow \Z^{\ell}\xrightarrow{Q}\Z^{d+\ell}\rightarrow 0]$, is given by the complex 
\begin{align*}
0\rightarrow \Z^{n+\ell}\xrightarrow{[BQ]}\Z^{d+\ell}\rightarrow 0.
\end{align*}
We define $\mathrm{DG}(\beta):=\mathrm{coker}([BQ]^\vee)$, and define the map 
\begin{align*}
\beta^\vee:(\Z^n)^\vee\rightarrow \mathrm{DG}(\beta)
\end{align*}
by the composition $(\Z^n)^\vee\hookrightarrow(\Z^{n+\ell})^\vee\rightarrow \mathrm{DG}(\beta)$. We then have $Z_{\boldsymbol{\Sigma}}=\A^n\setminus\{\boldsymbol{0}\}$ (since $\Sigma$ is complete) is the quasi-affine variety associated to the fan. By defining $G_{\boldsymbol{\Sigma}}=\Hom_{\Z}(\mathrm{DG}(\beta),\C^*)$, we get a morphism $G_{\boldsymbol{\Sigma}}\rightarrow (\C^*)^n$, and this induces an action of $G_{\boldsymbol{\Sigma}}$ on $Z_{\boldsymbol{\Sigma}}$ via the natural action of $(\C^*)^n$ on $\C^n$. The resulting stack $\mathcal{X}(\boldsymbol{\Sigma}):=[Z_{\boldsymbol{\Sigma}}/G_{\boldsymbol{\Sigma}}]$ is called the toric Deligne--Mumford stack associated to $\boldsymbol{\Sigma}$.

\begin{ex}[{\cite[Example 3.5]{BCS}}]
	Let $\Sigma$ be the complete fan in $\mathbb{Q}$, and
	\begin{align*}
	\beta:\Z^2\xrightarrow{\begin{pmatrix}
		2 & -3\\
		1& 0
		\end{pmatrix}}\Z\oplus \Z/2=:N.
	\end{align*}
	Then one can check that 
	\begin{align*}
	\beta^\vee:(\Z^2)^\vee\xrightarrow{\begin{pmatrix}
		4 & 6 \\
		\end{pmatrix}} \mathrm{DG}(\beta)\simeq \Z,
	\end{align*}
	and so the $\C^*$ action on $Z_{\boldsymbol{\Sigma}}=\A^2\setminus\{(0,0)\}$ is $t\cdot (x,y)=(t^{4}x,t^6y)$, yielding $\mathcal{X}(\boldsymbol{\Sigma})\simeq\P(4,6)$.
\end{ex}
Given a stacky fan $\boldsymbol{\Sigma}=(\Sigma, \beta, N)$, one can associate its rigidification $\boldsymbol{\Sigma}^{\text{rig}}=(\Sigma, \beta^{\text{rig}},N/N_{\text{tor}})$ by defining $\beta^{\text{rig}}:\Z^n\rightarrow N/N_{\text{tor}}$ to be the composition of $\beta$ and the quotient morphism $N\rightarrow N/N_{\text{tor}}$. The stack $\mathcal{X}(\boldsymbol{\Sigma}^{\text{rig}})$ is the DM stack associated to this stacky fan, and, by construction, comes with the rigidification map $\mathcal{X}(\boldsymbol{\Sigma})\rightarrow \mathcal{X}(\boldsymbol{\Sigma}^{\text{rig}})$ induced from the injective morphism $\text{DG}(\beta^{\text{rig}})\rightarrow\text{DG}(\beta)$. \\

Closed substacks corresponding to cones of the fan are defined in \cite[Section 4]{BCS}. We will restrict ourselves to the case of rays (one-dimensional cones), and recall the basic construction here. Let $\rho_i$ be a ray of $\Sigma$, $e_i$ the positive generator of the $i^{\text{th}}$ summand of $\Z^n$, $N_{\rho_i}$ the subgroup of $N$ generated by $\beta(e_i)$, and $N(\rho_i)=N/N_{\rho_i}$ the quotient. This defines a surjection $N_\mathbb{Q}\rightarrow N(\rho_i)_{\mathbb{Q}}$, and the quotient fan $\Sigma/\rho_i$ in $N(\rho_i)_{\mathbb{Q}}$ is defined as the image of the cones in $\Sigma$ containing $\rho_i$ under this surjection. The link of $\rho_i$ is defined as $\mathrm{link}(\rho_i)=\{\tau|\ \tau+\rho_i\in\Sigma,\ \text{and}\ \rho_i\cap\tau=0\}$. Let $\ell$ be the number of rays in $\mathrm{link}(\rho_i)$. We define the closed substack associated to $\rho_i$ as the triple $\boldsymbol{\Sigma}/\boldsymbol{\rho_i}=(\Sigma/\rho_i,N(\rho_i),\beta(\rho_i))$, where
\begin{align*}
\beta(\rho_i):\Z^\ell\rightarrow N(\rho_i)
\end{align*}
is defined as the composition $\Z^\ell\hookrightarrow \Z^n\xrightarrow{\beta} N\rightarrow N/N_{\rho_i}=N(\rho_i)$. In particular, the divisor $\mathcal{D}_{\rho_i}$ corresponding to the ray $\rho_i$ is $\mathcal{X}(\boldsymbol{\Sigma}/\boldsymbol{\rho_i})$.\\

Of most importance to us is the fact that if $\mathcal{X}$ is a toric DM stack whose coarse moduli space is $\P^1$ or $\P^2$, then (amongst other things) \cite[Theorem II]{FantechiMannNironi} shows that there exists a stacky fan whose corresponding quotient stack is $\mathcal{X}$. Moreover, in the case that $\mathcal{X}$ is an orbifold, this fan is unique. This is far from true in the case where $N$ has torsion, as is demonstrated in \cite[Example 7.29]{FantechiMannNironi}. There are several sources of non-uniqueness, although in our situation it is essentially equivalent to the fact that it is possible to choose multiple lifts of an element in $\Z/n$ to $\Z$. \\

From now on, we will restrict ourselves the the case of toric Deligne--Mumford stacks whose coarse moduli space is given by $\P^1$ or $\P^2$. Let $\mathcal{C}=\mathcal{X}(\boldsymbol{\Sigma})$ be a toric Deligne--Mumford stack whose rigidification is $\P_{a,b}$. Then, \cite[Proposition 7.20]{FantechiMannNironi} shows that there is a unique class in $\Ext^1_\Z(N_{\text{tor}},\Pic \P_{a,b})$ such that the $\Hom(N_{\text{tor}},\C^*)-$banded\footnote{We mention banding only for completeness; we continue to only consider trivially banded gerbes.} gerbe over $\P_{a,b}$ associated to this class is equivalent to $\mathcal{C}$. The proof of this proposition is constructive, and it is straightforward to determine the short exact sequence \eqref{GerbeClassification SES} from the data of a stacky fan. The main ingredient, however, which we will use in our application to invertible polynomials is \cite[Theorem 7.24]{FantechiMannNironi}, which shows (as a special case) that if $\Sigma$ is the complete fan in $\mathbb{Q}$ and 
\begin{align*}
\beta:\Z^2\xrightarrow{\begin{pmatrix}
	a & -b\\
	n_-& n_+
	\end{pmatrix}}\Z\oplus \Z/d=:N,
\end{align*}
then 
\begin{align*}
X(\boldsymbol{\Sigma})\simeq \sqrt[d]{\mathscr{L}/\P_{a,b}}
\end{align*}
\emph{as toric DM stacks}, where $\mathscr{L}=\OO(q_{-})^{n_-}\otimes \OO(q_{+})^{n_+}$. 
\begin{rmk}\label{Gerbe/Stack RMK}
	It is important to emphasise that two inequivalent gerbes can be equivalent as toric DM stacks. This happens when the corresponding Ext-classes are isomorphic as sequences, but inequivalent as extensions --  see \cite[Remark 7.23]{FantechiMannNironi} and \cite[Proposition 6.2]{Uniformization}. In particular, the above application of \cite[Theorem 7.24]{FantechiMannNironi} only makes a claim about toric DM stacks. By \cite[Proposition 7.20]{FantechiMannNironi}, one can check when this equivalence is also an equivalence of gerbes, although this will not be necessary for our purposes. 
\end{rmk}
\begin{ex}(cf. \cite[Example 3.6]{BCS})
	Let $\Sigma$ be the complete fan in $\mathbb{Q}$, $N=\Z\oplus \Z/3$, and 
	\begin{align*}
	\beta_n:\Z^2\xrightarrow{\begin{pmatrix}
		1 & -1\\
		n& 0
		\end{pmatrix}}\Z\oplus \Z/3.
	\end{align*}
	Since $\Sigma$ is complete, we have $Z_{\boldsymbol{\Sigma}_n}=\A^2\setminus\{(0,0)\}$ for any $n$. In the case of $n\bmod 3=0$, we have
	\begin{align*}
	\beta_0^\vee:(\Z^2)^\vee\xrightarrow{\begin{pmatrix}
		1 & 1 \\
		0 & 0
		\end{pmatrix}} \mathrm{DG}(\beta)\simeq \Z\oplus \Z/3,
	\end{align*}
	and so $G_{\boldsymbol{\Sigma}_0}\simeq \C^*\times\mu_3$, and $\mathcal{X}(\boldsymbol{\Sigma}_0)\simeq\P^1\times B\mu_3$. In the case where $n\bmod 3\neq 0$, we have 
	\begin{align*}
	\beta_n^\vee:(\Z^2)^\vee\xrightarrow{\begin{pmatrix}
		3 & 3 
		\end{pmatrix}} \mathrm{DG}(\beta)\simeq \Z,
	\end{align*}
	and so $G_{\boldsymbol{\Sigma}_n}\simeq \C^*$, and $\mathcal{X}(\boldsymbol{\Sigma}_n)\simeq\P(3,3)$ \emph{as toric DM stacks} for any such $n$. However, the class in $\Ext^1_{\Z}(\Z/3,\Z)$ corresponding to $n$ is given by the sequence
	\begin{align*}
	0\rightarrow \Z\xrightarrow{ \times3}\Z\xrightarrow{\times n}\Z/3\rightarrow 0,
	\end{align*}
	and so $n_1$ and $n_2$ \emph{do not} define equivalent gerbes unless $n_1\equiv n_2\bmod 3$. Moreover, this shows
	\begin{align*}
	\mathcal{X}(\boldsymbol{\Sigma}_n)&\simeq \sqrt[3]{\OO(1)/\P^1}\quad \text{for } n\bmod 3 =1\text{, and}\\
	\mathcal{X}(\boldsymbol{\Sigma}_n)&\simeq \sqrt[3]{\OO(2)/\P^1}\quad \text{for }n\bmod 3 =2.
	\end{align*}
	This also demonstrates the non-uniqueness of the fan in the case where the DM stack is not an orbifold -- taking $\beta_n$ for any $n\in\Z$ yields a gerbe which is equivalent to the $3^{\text{rd}}$ root of $\OO(n\bmod 3)$. 
\end{ex}


\begin{thebibliography}{100} 
\bibitem[AGV08]{AGV}
D. Abramovich, T. Graber and A. Vistoli.
\newblock Gromov-Witten theory of Deligne-Mumford stacks.
\newblock{\em Amer. J. Math.} Vol. 130 (2008), pp. 1337-1398.

\bibitem[APS19]{Amiot2019ACD}
C. Amiot, P.-G. Plamondon and S. Schroll. 
\newblock A complete derived invariant for gentle algebras via winding numbers and Arf invariants.
\newblock{\em Sel. Math. (NS),} Vol. 29, no. 30 (2023).

\bibitem[AT20]{AramakiTakahashi}
D. Aramaki and A. Takahashi. 
\newblock Maximally-graded matrix factorizations for an invertible polynomial of chain type.
\newblock {\em Adv. Math.} Vol. 373 (2020).

\bibitem[AS87]{AssemSkowronski}
I. Assem and A. Sk\'owronski. 
\newblock Iterated tilted algebras of type $\widetilde{A}_n$.
\newblock {\em Math. Z.} Vol. 195 (1987), pp. 269-290.

\bibitem[Aur10]{Auroux_fukayacategories}
D. Auroux. 
\newblock Fukaya categories of symmetric products and bordered Heegaard-Floer homology.
\newblock {\em  J. G\"okova Geom. Top. GGT} Vol. 4 (2010), pp. 1-54. 

\bibitem[BN06]{Uniformization}
K. Behrend and B. Noohi. 
\newblock Uniformization of Deligne--Mumford curves. 
\newblock {\em J. Reine Angew. Math. (Crelle's Journal)} Vol. 599 (2006), pp. 111-153.

\bibitem[Be{\u{\i}}78]{Beilinson1978CoherentSO}
A. Be{\u{\i}}linson.
\newblock Coherent sheaves on $\mathbb{P}^n$ and problems of linear algebra.
\newblock {\em Funct. Anal. its Appl.} Vol. 12 (1978), pp. 214-216.

\bibitem[BH95]{Berglund1995LandauGinzburgOM} 
P.~Berglund and M.~Henningson.
\newblock Landau--Ginzburg orbifolds, mirror symmetry and the elliptic genus.
\newblock {\em Nucl. Phys. B.} Vol. 433 (1995), 311-332.

\bibitem[BH93]{BerglundHubsch}
P. Berglund and T. H\"ubsch. 
\newblock A generalized construction of mirror manifolds. 
\newblock {\em Nucl. Phys. B} Vol. 393 (1993), pp. 377-391. 

\bibitem[BCS03]{BCS}
L. Borisov, L. Chen and G. Smith. 
\newblock The orbifold Chow ring of toric Deligne--Mumford stacks. 
\newblock {\em J. Am. Math. Soc.} Vol. 18 (2009), pp. 193-215. 

\bibitem[BD09]{Burban2009TiltingON}
I. Burban and Yu. Drozd. 
\newblock Tilting on non-commutative rational projective curves. 
\newblock{\em Math. Ann.} Vol. 351 (2009), pp. 665-709. 


\bibitem[Cad07]{Cadman2003UsingST}
C. Cadman
\newblock Using stacks to impose tangency conditions on curves. 
\newblock{\em Amer. J. Math.} Vol. 129 (2007), pp. 405-427. 

\bibitem[CCJ20]{ChoChoaJeong}
C.-H. Cho, D. Choa and W. Jeong. 
\newblock Fukaya category for Landau-Ginzburg orbifolds.
\newblock{\em arXiv:2010.09198} (2020).

\bibitem[Cox95]{CoxToricVar}
D. Cox. 
\newblock The homogeneous coordinate ring of a toric variety. 
\newblock {\em J. Algeb. Geom.} Vol. 4 (1995), pp. 17-50. 

\bibitem[CE03]{CrowleyEscher}
D. Crowley and C. Escher. 
\newblock A classification of ${S}^3$-bundles over ${S}^4$.
\newblock{\em Differ. Geom. Appl.} Vol. 18 (2003), pp. 363-380. 

\bibitem[CT13]{CaldararuTu}
A. C\u{a}ld\u{a}raru and J. Tu. 
\newblock Curved ${A}_\infty$-algebras and Landau-Ginzburg models.
\newblock {\em New York J. Math.} Vol. 19 (2013), pp. 305-342.

\bibitem[ET13]{EbelingTakahashi}
W. Ebeling and A. Takahashi. 
\newblock Mirror symmetry between orbifold curves and cusp singularities with group action.
\newblock {\em Int. Math. Res. Not. } Vol. 2013, pp. 2240-2270. 

\bibitem[FMN10]{FantechiMannNironi}
B. Fantechi, \'E. Mann and F. Nironi.
\newblock Smooth toric Deligne-Mumford stacks.
\newblock {\em J. Reine Angew. Math. (Crelle's Journal)} Vol. 648 (2010), pp. 201-244.


\bibitem[FU11]{Futaki2011}
M. Futaki and K. Ueda. 
\newblock Homological mirror symmetry for Brieskorn--Pham singularities. 
\newblock {\em Selecta Math. (N.S.)} Vol. 17 (2011), pp. 435 - 452. 

\bibitem[FU13]{FU2}
M. Futaki and K. Ueda. 
\newblock Homological mirror symmetry for singularities of type D. 
\newblock{\em Math. Z.} Vol. 273 (2013), pp. 633-652. 

\bibitem[Gam20]{GammageHMS}
B. Gammage. 
\newblock Mirror symmetry for Berglund-H\"ubsch Milnor fibres. 
\newblock {\em arXiv:2010.15570} (2020).


\bibitem[GPS18]{GanatraPardonShendeMicrolocal}
S. Ganatra, J. Pardon and V. Shende. 
\newblock Microlocal Morse theory of wrapped Fukaya categories. 
\newblock {\em arXiv: 1809.08807} (2018). 

\bibitem[GPS23]{GPSDescent}
S. Ganatra, J. Pardon and V. Shende. 
\newblock Sectorial descent for wrapped Fukaya categories. 
\newblock {\em J. Am. Math. Soc.} to appear (2023). arXiv:1809.03427. 

\bibitem[GL87]{GeigleLenzing} 
W. Geigle and H. Lenzing. 
\newblock A class of weighted projective curves arising in the representation theory of finite dimensional algebras. 
\newblock {\em Lecture notes in mathematics -- Springer verlag} Vol. 1273 (1987), pp. 265-297. 

\bibitem[Gir71]{giraud1971cohomologie}
J. Giraud. 
\newblock {\em Cohomologie non--ab{\'e}lienne.}
\newblock Grundlehren der mathematischen Wissenschaften in Einzeldarstellungen mit besonderer Ber{\"u}cksichtigung der Anwendungsgebiete. Springer verlag 1971. 

\bibitem[Hab22]{HabermannHMS}
M.~Habermann.
\newblock Homological mirror symmetry for invertible polynomials in two variables. 
\newblock {\em Quantum Top.} Vol. 13 (2022), 207-253.

\bibitem[Hab22a]{HabBHHHMS}
M.~Habermann.
\newblock A note on homological Berglund-H\"ubsch-Henningson mirror symmetry for curve singularities.
\newblock {\em arXiv:2205.12947} (2022). 

\bibitem[HS20]{HabermannSmith}
M.~Habermann and J.~Smith.
\newblock Homological {B}erglund-{H}\"{u}bsch mirror symmetry for curve
singularities.
\newblock {\em J. Symp. Geom.} Vol. 18 (2020), pp. 1515-1574.

\bibitem[HKK14]{HKK}
F. Haiden, L. Katzarkov and M. Kontsevich. 
\newblock Flat surfaces and stability structures. 
\newblock {\em Publ. Math. Inst. Hautes {\'E}tudes Sci.} Vol. 126 (2014), pp. 247-318. 

\bibitem[HO18]{HiranoOuchi} 
Y. Hirano and G. Ouchi. 
\newblock Derived factorization categories of non-{T}hom--{S}ebastiani-type sum of potentials.
\newblock {\em Proc. Lond. Math. Soc.} Vol. 126 (2023), pp. 1-75. 

\bibitem[IU15]{IshiiUeda}
A. Ishii and K. Ueda. 
\newblock The special McKay correspondence and exceptional collection.
\newblock {\em Tohoku Math. J.} Vol. 67 (2015), pp. 585-609. 

\bibitem[Isi10]{Isikorlov}
M. Isik. 
\newblock Equivalence of the derived category of a variety with a singularity category. 
\newblock {\em Int. Math. Res. Not.} Vol. 2013, pp. 2787-2808. 


\bibitem[Kra10]{Krawitz} M.~Krawitz.
{\em FJRW rings and Landau--Ginzburg mirror symmetry.}
\newblock PhD Dissertation, University of Michigan, Ann Arbor (2010). Available at: \url{https://www.proquest.com/docview/762374402?parentSessionId=euqFJiQhhDbpX21Uun7IZFjqpZOBS\%2BEeBgrd5gaPrFw\%3D.}
	
\bibitem[KS92]{kreuzer1992} 
M.~Kreuzer and H.~Skarke.
\newblock On the classification of quasihomogeneous functions.
\newblock {\em Comm. Math. Phys.} Vol. 150 (1992), 137-147. 

\bibitem[LP11]{DehnSurgery}
Y. Lekili and T. Perutz. 
\newblock Fukaya categories of the torus and {D}ehn surgery. 
\newblock {\em Proc. Natl. Acad. Sci. USA} Vol. 108 (2011), pp. 8106-8113.

\bibitem[LP17]{LekiliPolishchukAuslander}
Y. Lekili and A. Polishchuk.
\newblock Auslander orders over nodal stacky curves and partially wrapped {F}ukaya categories. 
\newblock {\em J. Top.} Vol. 11 (2017), pp. 615-644.

\bibitem[LP20]{LPGentle}
Y. Lekili and A. Polishchuk.
\newblock Derived equivalences of gentle algebras via {F}ukaya categories 
\newblock {\em Math. Ann.} Vol. 376 (2020), pp. 187-225. 

\bibitem[LU22]{LekiliUeda}
Y.~Lekili and K.~Ueda. 
\newblock Homological mirror symmetry for Milnor fibers via moduli of  $A_\infty$-structures.
\newblock {\em J. Topol.} Vol. 15 (2022), pp. 1058-1106.

\bibitem[LU21]{LekiliUedaSimple}
Y.~Lekili and K.~Ueda. 
\newblock Homological mirror symmetry for Milnor fibers of simple singularities.
\newblock {\em J. Alg. Geom.} Vol. 8 (2021), pp. 562-586.

\bibitem[Li22]{LiNonexistence}
Y. Li.  
\newblock Nonexistence of exact Lagrangian tori in affine conic bundles over $\C^n$. 
\newblock {\em J. Symp. Geom.} Vol. 20 (2022), pp. 1067-1105.

\bibitem[LU10]{OrlovLuntsDG}
V. Lunts and D. Orlov. 
\newblock Uniqueness of enhancement for triangulated categories. 
\newblock {\em J. Am. Math. Soc.} Vol. 23 (2010), pp. 853-908. 

\bibitem[Miy91]{Miyachi}
J. Miyachi. 
\newblock Localization of triangulated categories and derived categories. 
\newblock {\em J. Algebra} Vol. 141 (1991), pp. 463-483. 

\bibitem[NZ06]{NadlerZaslow}
D. Nadler and E. Zaslow. 
\newblock Constructible sheaves and the Fukaya category. 
\newblock {\em J. Am. Math. Soc.} Vol. 22 (2006), pp. 233-286. 

\bibitem[Ols16]{olsson2016algebraic}
M. Olsson. 
\newblock {\em Algebraic spaces and stacks.}
\newblock Colloquium Publications. American Mathematical Society (2016).

\bibitem[OS02]{OlssonStarr}
M. Olsson and J. Starr. 
\newblock Quot functors for Deligne--Mumford Stacks. 
\newblock {\em Comm. Algebra,} Vol. 31 (2002), pp. 4069-4096. 

\bibitem[OPS18]{Opper2018AGM}
S. Opper, G.-P. Plamondon and S. Schroll. 
\newblock A geometric model for the derived category of gentle algebras.
\newblock {\em arXiv:1801.09659} (2018). 

	\bibitem[Orl09]{Orlov2009}
D.~Orlov.
\newblock Derived categories of coherent sheaves and triangulated
categories of singularities.
\newblock {\em  Algebra, arithmetic, and geometry: in
	honor of Yu. I. Manin. Vol. II,} Vol. 270 of Progr. Math. 503-531. Birkh\"auser Boston, Inc. Boston, MA (2009). 


\bibitem[PV21]{PolishchukVarolgunes}
A.~Polishchuk and U.~Varolgunes.
\newblock On homological mirror symmetry for chain type polynomials.
\newblock {\em Math. Ann.} (2023).

\bibitem[Rin90]{Ringel_1990}
C. M. Ringel.
\newblock The canonical algebras. 
\newblock{\em Polska Akademia Nauk. Instytut Matematyczny PAN} Vol. 26 (1990), pp. 407-432.

\bibitem[Sei08]{SeidelBook}
P.~Seidel. 
\newblock {\em Fukaya categories and Picard--Lefschetz theory. }
\newblock  Zurich Lectures in Advanced Mathematics. European
Mathematical Society (EMS), Z\"urich (2008).

\bibitem[Shi12]{Shipmandg}
I. Shipman. 
\newblock A geometric approach to Orlov's theorem. 
\newblock {\em Compos. Math.} Vol. 148 (2012), pp. 1365 - 1389. 

\bibitem[STZ14]{STZ}
N. Sibilla, D. Treumann and E. Zaslow.
\newblock Ribbon graphs and mirror symmetry.
\newblock{\em Selecta Math. (N.S.)} Vol. 20 (2014), pp. 979-1002.

\bibitem[Syl16]{SylvanPwrapped}
Z. Sylvan. 
\newblock On partially wrapped Fukaya categories. 
\newblock {\em J. Top.} Vol. 12 (2016), pp. 372-441. 

\bibitem[Tak10]{Weightedprojectivelines} 
A.~Takahashi.
\newblock Weighted projective lines associated to regular systems of weights of dual type.
\newblock {\em Adv. Stud. Pure Math.}, Vol. 59 (2010), 371-388.

\bibitem[Ued06]{2006math......4361U} 
K.~Ueda.
\newblock Homological mirror symmetry and simple elliptic singularities. 
\newblock {\em arXiv: math/0604361} (2006).

\end{thebibliography}
\end{document}